%% file: sn-article.tex
\documentclass[pdflatex,sn-mathphys-num]{sn-jnl}% Math and Physical Sciences Numbered Reference Style
\usepackage{graphicx}%
\usepackage{multirow}%
\usepackage{amsmath,amssymb,amsfonts}%
\usepackage{amsthm}%
\usepackage[title]{appendix}%
\usepackage{xcolor}%
\usepackage{textcomp}%
\usepackage{manyfoot}%
\usepackage{booktabs}%
\usepackage{listings}%
\theoremstyle{thmstyleone}%
\newtheorem{theorem}{Theorem}%  meant for continuous numbers
\theoremstyle{thmstyletwo}%
\newtheorem{remark}{Remark}%

\theoremstyle{thmstylethree}%

\usepackage{hyperref}       % hyperlinks
\usepackage{url}            % simple URL typesetting
\usepackage{booktabs}       % professional-quality tables
\usepackage{amsfonts}       % blackboard math symbols
\usepackage{nicefrac}       % compact symbols for 1/2, etc.
\usepackage{microtype}      % microtypography
\usepackage{xcolor}         % colors

\usepackage{microtype}
\usepackage{graphicx}
\usepackage{subfigure}
\usepackage{booktabs} % for professional tables

\usepackage{hyperref}

\usepackage{amsmath}
\usepackage{amssymb}
\usepackage{mathtools}
\usepackage{amsthm}
\usepackage{comment}
\usepackage[capitalize,noabbrev]{cleveref}

\usepackage{algorithm}
\usepackage[noend]{algorithmic}

\newtheorem{lemma}[theorem]{Lemma}
\newtheorem{corollary}[theorem]{Corollary}
\newtheorem{assumption}[theorem]{Assumption}

\def \PSI {\Psi_{c, \kappa, u}}
\def\vg#1{{#1}}

\input{defs}

\newcommand{\moh}[1]{{#1}}

\newcommand{\sX}{{\R^{d_x}}}
\newcommand{\sY}{{\R^{d_y}}}
\newcommand{\sZ}{{\R^{d}}}

\begin{document}

\title[On Solving Minimization and Min-Max Problems
by First-Order Methods with Relative Error in Gradients]{On Solving Minimization and Min-Max Problems
by First-Order Methods with Relative Error in Gradients}

%%=============================================================%%
%% GivenName	-> \fnm{Joergen W.}
%% Particle	-> \spfx{van der} -> surname prefix
%% FamilyName	-> \sur{Ploeg}
%% Suffix	-> \sfx{IV}
%% \author*[1,2]{\fnm{Joergen W.} \spfx{van der} \sur{Ploeg} 
%%  \sfx{IV}}\email{iauthor@gmail.com}
%%=============================================================%%

% \author*[1,2]{\fnm{First} \sur{Author}}\email{iauthor@gmail.com}

% \author[2,3]{\fnm{Second} \sur{Author}}\email{iiauthor@gmail.com}
% \equalcont{These authors contributed equally to this work.}

% \author[1,2]{\fnm{Third} \sur{Author}}\email{iiiauthor@gmail.com}
% \equalcont{These authors contributed equally to this work.}

% \affil*[1]{\orgdiv{Department}, \orgname{Organization}, \orgaddress{\street{Street}, \city{City}, \postcode{100190}, \state{State}, \country{Country}}}

% \affil[2]{\orgdiv{Department}, \orgname{Organization}, \orgaddress{\street{Street}, \city{City}, \postcode{10587}, \state{State}, \country{Country}}}

% \affil[3]{\orgdiv{Department}, \orgname{Organization}, \orgaddress{\street{Street}, \city{City}, \postcode{610101}, \state{State}, \country{Country}}}

\author[1]{\fnm{Artem} \sur{Vasin}}
\equalcont{These authors contributed equally to this work.}

\author[1]{\fnm{Valery} \sur{Krivchenko}}
\equalcont{These authors contributed equally to this work.}

\author[2]{\fnm{Dmitry} \sur{Kovalev}}

\author[1,3]{\fnm{Fedyor} \sur{Stonyakin}}

\author*[4]{\fnm{Nazarii} \sur{Tupitsa}}
\email{nazarii.tupitsa[at]mbzuai.ac.ae}

\author[5]{\fnm{Pavel} \sur{Dvurechensky}}

\author[3]{\fnm{Mohammad} \sur{Alkousa}}

\author[1,6]{\fnm{Nikita} \sur{Kornilov}}

\author[3,1,7,8]{\fnm{Alexander} \sur{Gasnikov}}

\affil[1]{\orgdiv{Department of Applied Mathematics}, \orgname{Moscow Institute of Physics and Technology (MIPT)}, \orgaddress{\city{Moscow}, \country{Russia}}}

\affil[2]{\orgname{Yandex \& Institute for System Programming of the Russian Academy of Sciences (ISP RAS)}, \orgaddress{\city{Moscow}, \country{Russia}}}

\affil[3]{\orgname{Innopolis University}, \orgaddress{\city{Innopolis}, \country{Russia}}}

\affil[4]{\orgname{Mohamed bin Zayed University of Artificial Intelligence (MBZUAI)}, \orgaddress{\city{Abu Dhabi}, \country{United Arab Emirates}}}

\affil[5]{\orgname{Weierstrass Institute for Applied Analysis and Stochastics (WIAS)}, \orgaddress{\city{Berlin}, \country{Germany}}}

\affil[6]{\orgname{Skolkovo Institute of Science and Technology (Skoltech)}, \orgaddress{\city{Moscow}, \country{Russia}}}

\affil[7]{\orgname{Steklov Mathematical Institute of the Russian Academy of Sciences}, \orgaddress{\city{Moscow}, \country{Russia}}}

\affil[8]{\orgname{Institute for Information Transmission Problems}, \orgaddress{\city{Moscow}, \country{Russia}}}

%%==================================%%
%% Sample for unstructured abstract %%
%%==================================%%

% \abstract{First-order methods for minimization and saddle point (min-max) problems are one of the cornerstones of modern ML. The majority of works obtain favorable complexity guarantees of such methods, assuming that exact gradient information is available. At the same time, even the use of floating-point representation of real numbers already leads to relative error in all the computations. Relative errors also arise in such applications as bilevel optimization, inverse problems, derivative-free optimization, and inexact proximal methods. This paper answers several theoretical open questions on first-order optimization methods under relative errors in the first-order oracle. We propose an explicit single-loop accelerated gradient method that preserves optimal linear convergence rate under maximal possible relative error in the gradient, and explore the tradeoff between the relative error and deterioration in the linear convergence rate. We further explore similar questions for saddle point problems and nonlinear equations, showing, for the first time in the literature, that a variant of gradient descent-ascent and the extragradient method are robust to such errors and providing estimates for the maximum level of noise that does not break linear convergence.}

\abstract{First-order methods for minimization and saddle point (min-max) problems are widely used for solving large-scale problems, in particular arising in machine learning. The majority of works obtain favorable complexity guarantees of such methods, assuming that exact gradient information is available. At the same time, even the use of floating-point representation of real numbers already leads to relative error in all the computations. Relative errors also arise in such applications as bilevel optimization, inverse problems, derivative-free optimization, and inexact proximal methods. This paper answers several theoretical open questions on first-order optimization methods under relative errors in the first-order oracle. We propose an explicit single-loop accelerated gradient method that preserves optimal linear convergence rate under maximal possible relative error in the gradient, and explore the tradeoff between the relative error and deterioration in the linear convergence rate. We further explore similar questions for saddle point problems and nonlinear equations, showing, for the first time in the literature, that a variant of gradient descent-ascent and the extragradient method are robust to such errors and providing estimates for the maximum level of noise that does not break linear convergence.}


%%\pacs[JEL Classification]{D8, H51}

%%\pacs[MSC Classification]{35A01, 65L10, 65L12, 65L20, 65L70}

\maketitle

\section{Introduction} \label{sec:introduction}

Numerical methods in general and continuous optimization algorithms in particular inevitably rely on the floating-point representation of real numbers, which is the most common and practical way to handle real numbers in computational systems. This approach allows for efficient storage and arithmetic operations while accommodating a wide range of values, albeit with some precision limitations, described by IEEE754 standard~\cite{IEEEStd}. As described in~\cite{overton2001numerical} the approximation error of a positive real number $x$ with a floating-point number $\hat x$ is given by 
\begin{eqnarray} \label{eq:fp-approx}
    |x - \hat{x} | \leqslant \delta |x|,
\end{eqnarray}
where $\delta = 2^{-p}$ is called machine delta, and a precision $p$ is a number of bits used for the mantissa (or significant digits) in the floating-point representation. Indeed, we can denote a number in floating-point format as:
\begin{equation} \label{floating def}
    \begin{gathered}    
        \hat{x} = (-1)^s B \cdot 2^{E - E_0}, \quad
        B = 1.b_0b_1 \dots b_{p}, \; b_j \in \lbrace 0, 1 \rbrace, \; s \in \lbrace 0, 1 \rbrace.
    \end{gathered}
\end{equation}
We call machine delta $\delta$ such gap between $1$ and next larger floating point number. One can see, that $\delta = 0.0\dots1 = 2^{-p}$. Using notation from~\cite{overton2001numerical} we can define unit in the last place function: $\text{ulp}(\hat{x}) = 0.0\dots1 \cdot 2^{E - E_0} = \delta 2^{E - E_0}$. Note, that for positive floating-point numbers $\hat{x}$, value $\hat{x} + \text{ulp}(\hat{x})$ will be next larger floating point number.

Estimating the error of approximation for a positive real number $x$ with floating-point numbers~\eqref{floating def}, we obtain:
%\begin{equation*}
    $|x - \hat{x} | \leqslant \text{ulp}(\hat{x}) = \delta 2^{E - E_0} \leqslant \delta B 2^{E - E_0} \leqslant \delta \hat{x} \leqslant \delta |x|$. 
%\end{equation*}
Here we assume, that $\hat{x}$ is a lower estimation for $x$. We see that $\delta$ is a  bound for the relative error of this  approximation. In other words, the use of floating-point numbers of the IEEE754 standard leads to relative errors in all the computations. For example, using single-precision 32-bit float (FP32) ($p = 23, E_0 = 127, 0 \leqslant E \leqslant 255$) we obtain $\delta = 2^{-23} \approx 1.19e-07$, double-precision 64-bit float (FP64) ($p = 52, E_0 = 1023, 0 \leqslant E \leqslant 2047$) we obtain $\delta = 2^{-52} \approx 2.22e-16$, half-precision 16-bit float (FP16) ($p = 10, E_0 = 15, 0 \leqslant E \leqslant 31$) we obtain $\delta = 2^{-10} \approx 9.77e-04$. 
%Although in the process of training neural networks, it is also possible to use not only FP32 and then use quantization, but also to immediately use FP16, which does not significantly reduce the quality of the model, as was shown at~\cite{yun2024standalone16bittrainingmissing}.

From inequality~\eqref{eq:fp-approx}, it is evident that the approximation error is relative rather than absolute, meaning that it depends on the magnitude of the number $x$. The study of the influence of relative error is motivated by modern applications, where a less accurate representation of real numbers such as quantization is widely used~\cite{alistarh2017qsgd,gholami2022survey,rabbat2005quantized} to reduce model size and improve inference speed, maintaining efficiency~\cite{yun2024standalone16bittrainingmissing}, but increasing relative error. Such relative errors potentially influence all the machine computations, of which we are particularly concerned with optimization in a wider sense.

\textbf{Optimization with inexact gradients.} First, motivated by the widespread use of optimization methods in ML, we consider an unconstrained optimization problem
\begin{equation}\label{optim}
    \min\limits_{x \in \mathbb{R}^d} f(x),
\end{equation}
where $f: \mathbb{R}^d \longrightarrow \mathbb{R}$ is continuously differentiable function, and focus on gradient-type methods with relative
error (or relative noise) in the gradient. More precisely, following~\cite{polyak1987introduction}, we assume that an algorithm has access to an inexact gradient $\widetilde{\nabla} f(x)$, that for some $\alpha \in [0,1)$ and all $x$ satisfies
\begin{eqnarray}\label{eq:relative_error_intro}
    \|\widetilde{\nabla} f(x) - \nabla f(x)\|_2 \leqslant \alpha \|\nabla f(x)\|_2.
\end{eqnarray}

Apart from floating-point representation, relative error is natural for implicit optimization problems where evaluating the gradient or objective value requires solving an auxiliary problem that cannot be solved exactly. For example, in PDE-constrained optimization, the constraints are given by partial differential equations (PDEs), which add significant computational challenges \cite{baraldi2023proximal, hintermuller2020convexity}. Similarly, in bilevel optimization, the upper-level problem includes constraints defined by the solution to a lower-level problem, leading to additional inexactness \cite{solodov2007explicit,sabach2017first,petrulionyte2024functional}. Inexact gradients are also common in optimal control and inverse problems, where solving ODEs or PDEs is needed to compute the gradient of the objective function \cite{matyukhin2021convex}. Also, relative errors appear in the analysis of gradient approximation in Derivative Free Optimization~\cite{berahas2021theoretical}.
Last but not least, relative error often appears in inexact proximal methods that are used as an envelope to construct optimal methods in different contexts, see, e.g.,~\cite{lin2018catalyst,kovalev2022first}. Then a non-optimal method is used to evaluate inexact proximal step and the inexactness is measured relatively. Yet, unlike our paper, the first-order oracle in these methods is assumed to be exact.
For a deeper exploration of gradient inexactness, refer to \cite{devolder2013intermediate, polyak1987introduction, stonyakin2021inexact, vasin2023accelerated, kornilov2023intermediate, hallak24study} and the references therein.
We underline that most of these works focus on absolute error in the gradients, and the literature covering the setting of relative error \eqref{eq:relative_error_intro} is quite scarce. 
Given the above applications, it is important to understand what theoretical guarantees it is possible to obtain for optimization algorithms under relative inexactness.

\begin{remark}
    It may seem that the relative inexactness assumption is quite strong since it implies that the quality of the approximate gradient $\widetilde{\nabla} f(x)$ improves as the algorithm converges to a stationary point where $\nabla f(x)=0$. At the same time, this assumption holds in a number of applications, in particular, listed above, i.e., it is possible to improve the quality of approximation to establish convergence. Moreover, if such improvement is not possible, then Assumption~\eqref{eq:relative_error_intro} may still hold at points where the values of the gradient are not too small, see, e.g., \cite{berahas2021theoretical}. In this case, our theoretical results establish convergence rates on early stages of the method's work. Finally, our main theoretical question can be formulated as follows: What are the conditions for first-order algorithms to converge under relative inexactness in the gradient? Intuitively, from the situation in the world of absolute inexactness \cite{devolder2013intermediate,stonyakin2021inexact}, one may expect that to establish convergence, the error in the gradient should decrease on the trajectory of the algorithm.  
\end{remark}
% \na{Should we keep details in green in the Appendix?}
% \pd{I would keep them here for now.}

\textbf{Related works on optimization.} 
Assumption~\eqref{eq:relative_error_intro} was first used  alongside the Polyak-{\L}ojasiewicz (P{\L}) condition and $L$-smoothness to establish a classical convergence result for the gradient method~\cite{polyak1987introduction}. The book~\cite{bertsekas1997nonlinear} examines gradient methods under different gradient error conditions, including relative error, and provides classical convergence guarantees for $L$-smooth convex functions.
For the same setting, the authors of~\cite{vernimmen2025worst} establish a lower bound $\| \nabla f(x^N) \|^2_2=\Omega((N(1 - \alpha))^{-1})$ after $N$ iterations using the performance estimation techniques.
A more recent result by~\cite{hallak24study} includes a coordinate-wise version of Assumption~\eqref{eq:relative_error_intro} and considers constrained optimization problems without any structural assumptions such as P{\L} condition or strong convexity.

Only a few results are available in the literature on the convergence of accelerated gradient methods under relative inexactness.
In this paper, we focus on the $L$-smooth and $\mu$-strongly convex case, as considered in the recent works~\cite{gannot2021frequency,vasin2023accelerated}. Under assumption that $\alpha \lesssim \left(\mu/L\right)^{3/4}$ these works propose accelerated gradient methods that require optimal number of oracle calls $\mathcal{O} \rbr*{\sqrt{\frac{L}{\mu}} \log_2 \left( \frac{1}{\varepsilon} \right)}$  to find $\hat{x}$ s.t. $f(\hat{x}) - f(x^*) \leqslant \varepsilon$. Later, using a nested-loop algorithm,~\cite{kornilov2023intermediate} improved to $\alpha \lesssim \left(\mu/L\right)^{1/2}$ the bound on maximum tolerated inexactness that still allows achieving optimal oracle complexity. We improve on these results by proposing a single-loop algorithm and exploring tradeoff between the error and convergence rate.

\textbf{Min-Max optimization and nonlinear equations under inexactness.} Another class of problems important for ML and other applications 
%\textcolor{red}{[...]} 
is the class of saddle point problems (SPPs)
\begin{equation}\label{spp}
    \min\limits_{x \in \mathbb{R}^{d_x}} \max\limits_{y \in \mathbb{R}^{d_y}} f(x,y),
\end{equation}
where $f$ is continuously differentiable function. 
This problem can be seen as a particular case of a more general variational inequality problem which in the unconstrained setting takes the form of nonlinear equation (NE):
\begin{equation}\label{ne}
    g(z)=0,
\end{equation}
where $g:\R^d \to \R^d$ is continuously differentiable vector-function. 
This problem is also very relevant to ML applications, see, e.g., the recent reviews \cite{tran-dinh2024extragradient,tran-dinh2025accelerated} and the references therein. 
%\textcolor{red}{[...]}. 
We are not aware of any previous works that study the behavior of first-order methods with relative error in the partial gradients of $f$ for SPPs and in the values of $g$ for NEs. Our work provides the first results in this direction for SPPs when $f$ is 
% \na{As Pavel suggerted, add referecnes to assumptions and maybe short comment (spp and ne)}
strongly-convex-strongly-concave and has Lipschitz gradients (Assumptions \ref{as:conv-conc} or \ref{as:long}) and for NEs when $g$ is quasi strongly monotone and Lipschitz (\cref{as:lipqsm}). 
%These assumption are standard for considered problems.
% \na{should it quasi strongly monotone}

All the above motivates us to consider in this paper the following research question: how robust to relative noise are first-order methods and, in particular, is there a tradeoff between the oracle complexity and error in the oracle. Toward this end, our contributions can be summarized as follows.
%In this paper, we propose an accelerated method that also requires $\mathcal{O} \rbr*{\sqrt{\frac{L}{\mu}} \log_2 \left( \frac{1}{\varepsilon} \right)}$ oracle calls and does not rely on a restart technique, under the condition that $\alpha \lesssim \left(\mu/L\right)^{1/2}$.

\textbf{Contributions.} \underline{Optimization.} For solving problem \eqref{optim}, we propose an explicit single-loop linearly convergent accelerated gradient method that has the optimal oracle complexity $\mathcal{O} \rbr*{\sqrt{\frac{L}{\mu}} \log_2 \left( \frac{1}{\varepsilon} \right)}$ for the error $\alpha \lesssim \left(\mu/L\right)^{1/2}$. We further show that for larger values of $\alpha$ up to $\alpha < 1/3$ this method still has linear convergence rate, we explicitly estimate this rate and quantify the tradeoff between oracle complexity and the value of $\alpha$. \underline{SPPs.} For solving problem \eqref{spp}, we provide the first analysis of a variant of gradient descent-ascent (called Sim-GDA) under relative error in the partial gradients of $f$. We show that this algorithm has linear convergence for the error $\alpha < \mu/L$ (see the precise definitions of $\alpha,\mu,L$ in Section \ref{section min max}). We further, refine this result for the setting where the condition numbers w.r.t. $x$ and w.r.t. $y$ are different. \underline{NEs.} For solving problem \eqref{ne}, we provide the first analysis of ExtraGradient algorithm under relative error in the operator $g$. Surprisingly, we show that this algorithm has optimal oracle complexity $\mathcal{O} \rbr*{\frac{L}{\mu} \log_2 \left( \frac{1}{\varepsilon} \right)}$ for the error $\alpha \lesssim \left(\mu/L\right)^{1/2}$ (see Section \ref{sec:ne} for the precise definition of $\alpha,\mu,L$), demonstrating that a deep connection between ExtraGradient and accelerated method goes beyond the setting of exact first-order information \cite{Cohen2020RelativeLI}. As a corollary, when applied to SPPs, this result uncovers an interesting phenomenon: ExtraGradient algorithm is more robust than Sim-GDA and Alt-GDA in the strongly-convex-strongly-concave setting. We confirm this also by numerical experiments.

%In this paper, we propose studies of an accelerated methods with relative error~\eqref{eq:relative_error_intro} .
% \begin{itemize}
%     \item The method requires $\mathcal{O} \rbr*{\sqrt{\frac{L}{\mu}} \log_2 \frac{1}{\varepsilon}}$ calls to achieve $\varepsilon$ accuracy for $\alpha \lesssim \left(\mu/L\right)^{1/2}$
%     \item The method is explicit, requiring no restarts
%     \item The method maintains linear convergence with a non-optimal rate for any $\alpha\lesssim\frac{\mu}{L}$
% \end{itemize}

\textbf{Paper ogranization.} 
% The paper is structured as follows. 
% Section~\ref{section accelerated main} presents the accelerated Algorithm~\ref{alg:nesterov corrected} and establishes its convergence guarantees through Theorems~\ref{nesterov acc corrected convergence} and \ref{nesterov acc corrected convergence tau}  for  minimization problems. Sections~\ref{section min max} and~\ref{sec:ne} explore the impact of inexact gradients on saddle point problems, analyzing the robustness of Sim-GDA and ExtraGradient (EG) methods respectively.
% At the begging of each section we introduce assumptions and preliminaries necessary for our analysis. 
% Finally, Section~\ref{sec:numerical_results} provides numerical experiments supporting our theoretical findings. Proofs and additional experiments can be found in the Appendix.
Section~\ref{section accelerated main} is devoted to optimization problem \eqref{optim} and gives theoretical results on our Relative Error Accelerated Gradient Method. Sections~\ref{section min max} and~\ref{sec:ne} explore the impact of inexact gradients on SPPs \eqref{spp} and NEs \eqref{ne} analyzing the robustness of Sim-GDA, Alt-GDA and ExtraGradient (EG) methods respectively.
% At the begging of each section we introduce assumptions and preliminaries necessary for our analysis. 
Finally, Section~\ref{sec:numerical_results} provides numerical experiments supporting our theoretical findings. Proofs and additional experiments can be found in the Appendix.

\section{Minimization problems}
\label{section accelerated main}

% \subsection{Preliminaries}\label{sec:preliminaries}
We make the following assumptions on the class of optimized functions in \eqref{optim}.
% We define $f^*$ - as minimum value of $f$ or solution for problem~\eqref{optim} and $x^*$: $f(x^*) = f^*$, also for iterative methods with starting point $x^0$ we can define $R =  \|x^0 - x^* \|_2$.
% \begin{assumption}[Convexity]\label{as:cvx}
%     The objective function $f$ is convex, i.e.
%     \begin{equation*}
%     f(y) \geq f(x)+\langle\nabla f(x), y-x\rangle, \quad \forall x, y \in \mathbb{R}^d.
%     \end{equation*}
% \end{assumption}

% \begin{assumption}[$L$-smootheness]\label{as:Lsmooth}
%     $f$ is $L$-smooth, i.e., $\forall\; x,y \in \R^d$
%     \begin{eqnarray}\label{smoothness_cond}
%         f(y) \leq f(x)+ \left\langle\nabla f(x), y-x\right\rangle + \frac{L}{2} \|y-x\|_2^2.
%     \end{eqnarray}
%     % or equivalently 
%     % \begin{equation}\label{eq_6}
%     %     \|\nabla f(y) - \nabla f(x)\|_2 \leq L \|y - x\|_2.
%     % \end{equation}
% \end{assumption}

%  \begin{assumption}[Strong convexity]\label{as:str_cvx}
%     $f$ is $\mu$-strongly convex, i.e., $\forall\; x,y \in \R^d$
%     \begin{eqnarray}\label{eq:str_cvx} 
%         f(y) \geq f(x) + \langle \nabla f(x), y - x \rangle + \frac{\mu}{2}\|y - x\|_2^2.
%     \end{eqnarray}
% \end{assumption}

\begin{assumption}\label{as:Lsmooth+muConv}
$f$ is $L$-smooth and $\mu$-strongly convex, i.e., $\forall x,y \in \R^d$
\begin{eqnarray*}
    \frac{\mu}{2}\|y - x\|_2^2 \leqslant f(y) - f(x) - \left\langle\nabla f(x), y-x\right\rangle\leqslant  \frac{L}{2} \|y-x\|_2^2.
    % \label{smoothness_cond}
    % \\&f(y) - f(x) - \left\langle\nabla f(x), y-x\right\rangle \geqslant  \frac{\mu}{2}\|y - x\|_2^2. \label{eq:str_cvx} 
\end{eqnarray*}
\end{assumption}

We also make the following assumption about the inexact gradients of the objective function $f$.
\begin{assumption}[Relative noise]\label{as:noise}
We assume access to an inexact gradient $\widetilde{\nabla} f(x)$ of $f$ with relative error, meaning that there exists $\alpha \in [0, 1)$ s.t.
\begin{eqnarray}\label{eq:relative_error}
    \|\widetilde{\nabla} f(x) - \nabla f(x)\|_2 \leqslant \alpha \|\nabla f(x)\|_2, \forall x \in \R^d.
\end{eqnarray}
\end{assumption}

\begin{corollary}\label{growth condition a} 
Assumption~\ref{as:noise} for all \moh{$\alpha \in [0, 1)$} implies the following
\begin{eqnarray*} 
% \label{growth condition}
% \begin{gathered}
    \nu \| \nabla f(x) \|_2 \leqslant \| \widetilde{\nabla} f(x) \|_2 \leqslant \rho \|\nabla f(x) \|_2, 
    \quad
    \langle \widetilde{\nabla} f(x), \nabla f(x) \rangle \geqslant \gamma \|\widetilde{\nabla} f(x) \|_2 \|\nabla f(x) \|_2,
    % \label{growth condition b}
    % \end{gathered}
\end{eqnarray*}
where $\gamma = \sqrt{1 - \alpha^2}, \nu = 1 - \alpha, \rho = 1 + \alpha$.
\end{corollary}
By strong convexity, there exists a unique solution $x^*$ of problem \eqref{optim}, we denote $f^* = f(x^*)$.

%\subsection{Accelerated gradient method} \label{section accelerated}
We next improve upon the results of \cite{gannot2021frequency,vasin2023accelerated,kornilov2023intermediate} by proposing a single-loop accelerated method with an inexact gradient that provably achieves optimal complexity bounds under Assumption~\ref{as:noise} with $\alpha \lesssim \sqrt{\frac{\mu}{L}}$. The algorithm is listed as Algorithm~\ref{alg:nesterov corrected} and can be interpreted as a special modification of an algorithm from~\cite{nesterov2018lectures}: we carefully construct stepsizes that take into account the error in the gradients.
The main results of this section are given by Theorems~\ref{nesterov acc corrected convergence} and~\ref{nesterov acc corrected convergence tau} that present the convergence guarantees for Algorithm~\ref{alg:nesterov corrected}. Denote $R := \|x^0 - x^* \|_2$ for a starting point $x^0$.

\begin{algorithm}[tb]
    \caption{ RE-AGM (Relative Error Accelerated Gradient Method).} 
    \label{alg:nesterov corrected}
    \begin{algorithmic}[1]
        \STATE {\bfseries Input:} $(L, \mu, x_{\operatorname{start}}, \alpha)$.
        \STATE {\bf Set} $h = \frac{1}{L}\left(\frac{1 - \alpha}{1 + \alpha} \right)^\frac{3}{2}$,
        %\State {\bf Set} $\kappa_0 = L$.
        % \STATE {\bf Set} 
        $\widehat{L} = \frac{L(1 + \alpha)}{(1 - \alpha)^3}$, 
        % \STATE {\bf Set} 
        $x^0 = u^0 = x_{\operatorname{start}}$, 
        % $x^0 = u^0$,
        % \STATE {\bf Set} 
        $s = 1+ 2\alpha + 2\alpha^2$, %$s = \rho^2 + \alpha^2$.
        % \STATE {\bf Set}
        $m = 1-2\alpha$, %$m = \nu^2 - \alpha^2$, 
        % \STATE {\bf Set}
        $q = \mu / \widehat{L}$.
        %\State {\bf Set} $\theta^2 = \frac{1}{1 + \zeta} - \frac{\alpha^2}{\zeta}$.
        %\State {\bf Set} $K = \frac{\widehat{L}}{\mu} \left(s - \frac{\theta^2}{1 + \zeta} \right)$.
        \FOR {$k = 0, \dots, N - 1$}
        \STATE $
            a_k = \frac{(m - s) + \sqrt{(s - m)^2 +4mq }}{2 m}
        $, $\quad$i.e. the largest root of $
           m a_k^2 + (s - m) a_k - q = 0,  %m a_k (1 - a_k) - s a_k = -q,
        $
        % \moh{\STATE Solve the quadratic equation 
        % %\State $(s \widehat{L} - K \kappa_k) a_k^2 + \left((1 + K) \kappa_k - \mu \right) a_k - % \kappa_k = 0,$
        % %\STATE 
        % $
        %    m a_k^2 + (s - m) a_k - q = 0,  %m a_k (1 - a_k) - s a_k = -q,
        % $
        % \\
        % and take the largest root, i.e., 
        % %\State $a_k = \frac{\left(\mu - (1 + K) \kappa_k \right) + \sqrt{\left(\mu - (1 + K) 
        % %\kappa_k \right)^2 + \kappa_k \left(s \widehat{L} - K \kappa_k \right)}}{2 \left(s %\widehat{L} - K \kappa_k \right)},$
        % %\STATE 
        % % \begin{equation}\label{a_kAlg1}
        % $
        %     a_k = \frac{(m - s) + \sqrt{(s - m)^2 +4mq }}{2 m},
        % $
        % % \end{equation}
        % }
        % \State $\eta_k = \frac{\kappa_k a_k (1 - a_k)}{\kappa_k (1 - a_k) + \mu a_k},$
        % \State $\kappa_{k + 1} = (1 - a_k) \kappa_k + a_k \mu,$
        \STATE $y^k = \frac{a_k u^k + x^k}{1 + a_k}$, 
        % \STATE
        $\quad u^{k + 1} = (1 - a_k) u^{k} + a_k y^k - \frac{a_k}{\mu} \widetilde{\nabla} f(y^{k})$, 
        % \STATE
        $\quad x^{k + 1} = y^{k} - h \widetilde{\nabla} f(y^{k})$.
        \ENDFOR
        \STATE 
        \noindent {\bf Output:} $x^N$.
    \end{algorithmic}
\end{algorithm}

% For Theorem~\ref{nesterov acc corrected convergence} and Theorem~\ref{nesterov acc corrected convergence tau} we will use Lemmas from Appendix~\ref{appendix accelerated}.

\begin{theorem} \label{nesterov acc corrected convergence}
Let Assumption~\ref{as:Lsmooth+muConv} hold.
% $L$-smooth~\eqref{smoothness_cond} and $\mu$ strongly convex~\eqref{as:str_cvx}, $\widetilde{\nabla}f$ satisfies~\eqref{eq:relative_error}, 
If Assumption~\ref{as:noise} holds with $\alpha < \frac{\sqrt{2} - 1}{18 \sqrt{2}} \sqrt{\frac{\mu}{L}}$, then Algorithm~\ref{alg:nesterov corrected} with parameters ($L, \frac{\mu}{2}, x^0, \alpha)$ generates $x^N$ s.t.
\begin{equation*}
    f(x^N) - f^* \leqslant L R^2 \left( 1 - \frac{1}{10\sqrt{2}} \sqrt{\frac{\mu}{L}} \right)^N.
\end{equation*}
\end{theorem}

\begin{theorem} \label{nesterov acc corrected convergence tau}
Let Assumption~\ref{as:Lsmooth+muConv} hold.
% $L$-smooth~\eqref{smoothness_cond} and $\mu$ strongly convex~\eqref{as:str_cvx}, $\widetilde{\nabla}f$ satisfies~\eqref{eq:relative_error}, 
If Assumption~\ref{as:noise} holds with $\alpha = \frac{1}{3}\left(\frac{\mu}{L}\right)^{\frac{1}{2} - \tau}$, where $0 \leqslant \tau \leqslant \frac{1}{2}$, then Algorithm~\ref{alg:nesterov corrected} with parameters ($L, \frac{\mu}{2}, x^0, \alpha)$ generates $x^N$, s.t. 
\begin{equation*}
    f(x^N) - f^* \leqslant L R^2 \left( 1 - \frac{1}{10\sqrt{2}} \left(\frac{\mu}{L}\right)^{\frac{1}{2} + \tau} \right)^N.
    \end{equation*}
\end{theorem}
The proofs can be found in Appendix~\ref{appendix accelerated}. The next remark explains the result of Theorem~\ref{nesterov acc corrected convergence tau}.
\begin{remark} \label{nesterov conv tau remark}
Substituting $\tau = 0$ and $\tau = \frac{1}{2}$ we obtain
% \begin{enumerate}

1. If $\tau = 0$ and $\alpha = \frac{1}{3} \sqrt{\frac{\mu}{L}}$, then
% \begin{equation*}
$
    f(x^N) - f^* \leqslant L R^2 \left( 1 - \frac{1}{10\sqrt{2}} \sqrt{\frac{\mu}{L}} \right)^N
$.
% \end{equation*}

This is the optimal convergence rate for the class of functions satisfying Assumption~\ref{as:Lsmooth+muConv} meaning that for the error $\alpha$ up to $ \frac{1}{3} \sqrt{\frac{\mu}{L}}$ our algorithm has optimal convergence rate, and, hence optimal iteration complexity.

2. If $\tau = \frac{1}{2}$ and $\alpha = \frac{1}{3}$, then
% \begin{equation*}
$
    f(x^N) - f^* \leqslant L R^2 \left( 1 - \frac{1}{10\sqrt{2}} \frac{\mu}{L} \right)^N
$.
% \end{equation*}

This convergence rate is the same as for gradient descent meaning that for values of the error $\alpha$ up to $ \frac{1}{3}$ our algorithm still has linear convergence with the rate no worse than for the gradient descent. 
% \end{enumerate}
For the intermediate values of $\tau$ our algorithm interpolates between the accelerated gradient method and gradient method, maintaining linear convergence in the whole range of values of $\tau \in [0,1/2]$. We can also observe a tradeoff between the convergence rate and the error in the gradient: the larger is the error the worse is the guaranteed linear convergence rate.
% This Remark shows that Algorithm~\ref{alg:nesterov corrected} can maintain convergence with $\alpha = \left(\frac{\mu}{L} \right)^{\frac{1}{2} + \tau}$ with speed rate reduction. In particular, we can guarantee linear convergence with the $1 - \frac{\mu}{L}$ rate for $\alpha \leqslant \frac{1}{3}$.
\end{remark}
\begin{remark} Our results can be applied to some extent in gradient-free optimization. As shown in~\cite{berahas2022theoretical}, given an access to inexact evaluations $\widetilde{f}(x)$ of $f(x)$ that satisfy $|\widetilde{f}(x) - f(x)| \leqslant \epsilon_f$, the following gradient estimator 
\begin{equation} \label{FFD est}
    \widetilde{\nabla}f(x) = \displaystyle\sum_{j = 1}^{d} \frac{\widetilde{f}(x + \sigma u^j) - \widetilde{f}(x)}{\sigma} u^j, 
\end{equation}
where $u^j$ are unit coordinate vectors, satisfies
\begin{equation*}
    \|\widetilde{\nabla}{f}(x) - \nabla f(x) \|_2 \leqslant \frac{\sqrt{d} L \sigma}{2} + \frac{2\sqrt{d} \epsilon_f}{\sigma}.
\end{equation*}
Thus, setting $\sigma = 2 \sqrt{\frac{\epsilon_f}{L}}$ on set $\left\lbrace x \; \bigg | \; \|\nabla{f}(x) \|_2 \geqslant \frac{2\sqrt{dL\epsilon_f}}{\alpha} \right\rbrace$ we can guarantee condition $\|\widetilde{\nabla}{f}(x) - \nabla f(x) \|_2 \leqslant \alpha \| \nabla{f}(x) \|_2$ and our results can be applied for the first stage of the algorithm's work while $x$ belongs to the above set and the value of the gradient is not too small.

If we can additionally guarantee a relative accuracy on the value of $f$
\begin{equation*}
    |\widetilde{f}(x) - f(x)| \leqslant \alpha_0 \sqrt{f(x) - f^*}, 
\end{equation*}
Using Polyak-{\L}ojasiewicz (P{\L}) condition we obtain $|\widetilde{f}(x) - f(x)| \leqslant \frac{\alpha_0}{\sqrt{2\mu}} \|\nabla f(x) \|_2$. Then for the gradient estimator~\eqref{FFD est} we derive
\begin{equation*}
     \|\widetilde{\nabla}{f}(x) - \nabla f(x) \|_2 \leqslant \sqrt{\frac{2d}{\mu}}\frac{\alpha_0}{\sigma} \| \nabla{f}(x) \|_2 + \sqrt{d}L \left(\sigma + \frac{\alpha_0}{\sqrt{2\mu}} \right).
\end{equation*}
This leads to a combined relative and absolute inexactness model that is interesting to study in future works.

\begin{remark}
    A variant of relative inexactness in the first-order information was considered in~\cite{vaswani2019fast} in the context of stochastic optimization. In that paper, the inexact gradient $\widetilde{\nabla}{f}(x)$ is assumed to satisfy
    \begin{equation*}
    \begin{gathered}
        \mathbb{E} \left[ \widetilde{\nabla}{f}(x) \; | \; x \right] = \nabla{f}(x), \\
        \mathbb{E} \left[ \| \widetilde{\nabla}{f}(x) \|_2^2 \; | \; x \right] \leqslant \rho \| \nabla{f}(x) \|_2^2 + \sigma^2.
    \end{gathered}
\end{equation*}
In this case, the authors guarantee the following convergence rate
\begin{equation*}
    \mathbb{E} f(x^{N + 1}) - f^* \leqslant \left(1 - \frac{1}{\rho} \sqrt{\frac{\mu}{L}} \right)^N \left(f(x^0) - f^* + \frac{\mu}{2} \|x^0 - x^* \|_2^2 \right) + \frac{\sigma^2}{\rho \sqrt{\mu}{L}}.
\end{equation*}
Thus, if $\sigma=0$ it is possible to obtain linear convergence in the inexact case. At the same time, their assumption is very different from ours in two ways. Firstly, they assume that the gradient approximation is unbiased, whereas in our case, a systematic bias is allowed. Secondly, their relative error is for the second moment of $\widetilde{\nabla}{f}(x)$  rather than any bound for $\|\widetilde{\nabla}{f}(x) - \nabla f(x)\|$ as in our case.
    
%     Another motivation for studying the results of relative inexactness is stochastic optimization. As it was shown in~\cite{vaswani2019fast}, the term corresponding to relative inexactness does not worsen the convergence. The following model is usually considered:
% \begin{equation*}
%     \begin{gathered}
%         \mathbb{E} \left[ \widetilde{\nabla}{f}(x) \; | \; x \right] = \nabla{f}(x), \\
%         \mathbb{E} \left[ \| \widetilde{\nabla}{f}(x) \|_2^2 \; | \; x \right] \leqslant \rho \| \nabla{f}(x) \|_2^2 + \sigma^2.
%     \end{gathered}
% \end{equation*}
% In this case, the authors guarantee convergence:
% \begin{equation*}
%     \mathbb{E} f(x^{N + 1}) - f^* \leqslant \left(1 - \frac{1}{\rho} \sqrt{\frac{\mu}{L}} \right)^N \left(f(x^0) - f^* + \frac{\mu}{2} \|x^0 - x^* \|_2^2 \right) + \frac{\sigma^2}{\rho \sqrt{\mu}{L}}.
% \end{equation*}

% This assumption is different than ours, especially unbiased, and allows to keep linear convergence if $\sigma=0$.
\end{remark}

\end{remark}

\section{Saddle Point Problems} \label{section min max}
%\section{Robustness of inexact gradient descent-ascent and extragradient method for smooth saddle point problems} 

Unlike the case of minimization problems considered in the previous section, convergence of algorithms with inexact oracle and relative errors for saddle point problems is unexplored, to the best of our knowledge. In this section, we obtain the first several results on how gradient inexactness affects the convergence of a basic first-order method called Sim-GDA and designed for smooth saddle point problems.
%and whether or not this method enjoys similar level of robustness as we observe for convex minimization problems. 
Particularly, we derive upper thresholds for relative inexactness level $\alpha$ so that the values of $\alpha$ smaller than this threshold do not break the linear convergence of  the algorithm. We also show that this threshold is exact, i.e. there is no worst-case linear convergence if inexactness $\alpha$ exceeds the threshold. Below, in the section on numerics, we perform numerical worst-case analysis of different algorithms for saddle point problems to find the actual thresholds for the level of inexactness $\alpha$  for composite saddle point problems with bilinear coupling. 
%To that end, we employ a PEP-like technique of automatic search for Lyapunov functions \cite{pmlr-v80-taylor18a}. \par

% \subsection{Preliminaries}
Let us consider problem \eqref{spp}. 
We use the following additional notation that combines partial gradients of $f$ into one operator: $g_z(z) = [g_x(z)^{\top}, -g_y(z)^{\top}]^{\top}$ where $g_x(z) = \nabla_x f(z)$, $g_y(z) = \nabla_y f(z)$, and $z = [x^{\top}, y^{\top}]^{\top}$. We start with the following assumptions.

%convex-concave with strongly-monotone, L-Lipshits gradients
\begin{assumption} \label{as:conv-conc} Function $f\left(x,y\right) \colon \sX \times \sY \to \R$ is convex-concave and has $\mu$-strongly-monotone, $L$-Lipschitz gradient (for shortness, we write  $f\in S_{\mu, L}$), i.e., $\forall z^0, z^1 \in \R^{d_x+d_y}$
\begin{align*}
    \langle g_z(z^1) - g_z(z^0), z^1-z^0 \rangle \geqslant 
        \mu \|z^1 - z^0\|_2^2,
        % \\
        \quad
    \|g_z(z^1) - g_z(z^0)\|_2 \leqslant 
        L \|z^1 - z^0\|_2.
\end{align*}
\end{assumption}
\begin{assumption}[Relative noise] \label{as:conv-conc-inexact}
We assume access to inexact partial gradients  $\widetilde{\nabla}_x f(x,y)$ and $ \widetilde{\nabla}_y f(x,y)$ of $f$ with relative error, meaning that there exists $\alpha \in [0, 1)$ s.t. $\forall x \in \sX, y\in \sY$
\begin{equation*}
    \|\widetilde{\nabla}_x f(x,y) - \nabla_x f(x,y)\|_2^2 + \|\widetilde{\nabla}_y f(x,y) - \nabla_y f(x,y)\|_2^2 
    % \\ 
    \leqslant \alpha^2 \left( \|\nabla_x f(x,y)\|_2^2 + \|\nabla_y f(x,y)\|_2^2 \right).
\end{equation*}
\end{assumption}
The latter assumption is a counterpart of Assumption \ref{as:noise} for saddle point problems.
%and quantifies the relative error in partial gradients of $f$.

% \subsection{Gradient Descent-Ascent}
We consider two versions of Gradient Descent-Ascent (GDA) algorithm, namely Sim-GDA~\cref{alg:sim} and Alt-GDA~\cref{alg:alt}, that use inexact partial gradients of the objective $f$ with relative  inexactness and apply them to solve Problem~\eqref{spp}. Our theoretical results are obtained for Sim-GDA and below, in Section \ref{sec:numerical_results}, we present experiments with both Sim-GDA and Alt-GDA.
% \begin{align*}
% & x^{k+1} = x^{k} - \eta_x \widetilde{\nabla}_x f(x^k,y^k),  \\
% & y^{k+1} = y^{k} + \eta_y \widetilde{\nabla}_y f(x^{k},y^k), 
% \end{align*}

\begin{minipage}[htp]{0.45\textwidth} 
    \begin{algorithm}[H]
    \caption{ Inexact Sim-GDA.}
    	\label{alg:sim}
    \begin{algorithmic}[1]
    
    \STATE {\bfseries Input:} $x^0 = x_{\operatorname{start}}$, $y^0 = y_{\operatorname{start}}$, $\eta_x, \eta_y.$
    %\STATE 
    %\noindent {\bf Input:} Starting point $x_{\operatorname{start}}$, number of steps $N$.
    \FOR {$k = 1, \dots, N$}
            \STATE $x^{k+1} = x^{k} - \eta_x \widetilde{\nabla}_x f(x^k,y^k),$
            \STATE $y^{k+1} = y^{k} + \eta_y \widetilde{\nabla}_y f(x^{k},y^k).$
    \ENDFOR
    \STATE \noindent {\bf Output:} $z^N.$
    \end{algorithmic}
    \end{algorithm}
\end{minipage} \hfill
% \begin{align*}
% & x^{k+1} = x^{k} - \eta_x \widetilde{\nabla}_x f(x^k,y^k),  \\
% & y^{k+1} = y^{k} + \eta_y \widetilde{\nabla}_y f(x^{k+1},y^k), 
% \end{align*}
\begin{minipage}[htp]{0.45\textwidth} 
    \begin{algorithm}[H]
    \caption{ Inexact Alt-GDA.}
    	\label{alg:alt}
    \begin{algorithmic}[1]
    \STATE {\bfseries Input:} $x^0 = x_{\operatorname{start}}$, $y^0 = y_{\operatorname{start}}$, $\eta_x, \eta_y.$
    % \STATE {\bfseries Input:} $(x^0 = x_{\operatorname{start}}, y^0 = y_{\operatorname{start}}, \eta_x, \eta_y).$
    %\STATE 
    %\noindent {\bf Input:} Starting point $x_{\operatorname{start}}$, number of steps $N$.
    \FOR {$k = 1, \dots, N$}
            \STATE $x^{k+1} = x^{k} - \eta_x \widetilde{\nabla}_x f(x^k,y^k),$
            \STATE $y^{k+1} = y^{k} + \eta_y \widetilde{\nabla}_y f(x^{k+1},y^k).$
    \ENDFOR
    \STATE \noindent {\bf Output:} $z^N.$
    \end{algorithmic}
    \end{algorithm}
\end{minipage}
\\~\\

For functions $f \in S_{\mu, L}$, we can guarantee the following iteration complexity of Inexact Sim-GDA.
% Consider convex-concave functions with $\mu$-strongly-monotone, $L$-Lipschitz gradients, i.e., functions that belong to $S_{\mu, L}$. We can guarantee the following iteration complexity for Inexact Sim-GDA.

\begin{theorem}
\label{corollary:upper_bound_stongly_monotone_Lipshitz}
Let Assumption \ref{as:conv-conc} hold. If Assumption \ref{as:conv-conc-inexact} holds with $\alpha<\frac{\mu}{L}$, then Algorithm \ref{alg:sim} with parameters $(z^0, \eta,\eta)$, where $\eta=\frac{\mu -\alpha L}{(1+\alpha)^2L^2}$ generates $z^N$ s.t. $\|z^N-z^*\|_2\leqslant \varepsilon$ in 
\begin{align*}
     N=\mathcal{O} \rbr*{\frac{L^2}{\mu^2}\frac{1}{(1 - \alpha L/\mu)^2}\ln\frac{\|z^0-z^*\|_2}{\varepsilon}} \;\; \text{iterations.}
\end{align*} 
% Let Inexact Sim-GDA~\cref{alg:sim} be applied to Problem~\eqref{spp} with $f \in S_{\mu, L}$ (\cref{as:conv-conc}). Then, the sequence $z^N$ generated by this algorithm linearly converges to the saddle point $z^*$, i.e., the algorithm guarantees $\|z^N-z^*\|\leq \varepsilon$ in 
% \begin{align*}
%      N=\mathcal{O} \rbr*{\frac{L^2}{\mu^2}\frac{1}{(1 - \alpha L/\mu)^2}\ln\frac{\|z^0-z^*\|}{\varepsilon}} \;\; \text{iterations.}
% \end{align*}
% For $f \in S_{\mu, L}$ (\cref{as:conv-conc}), Inexact Sim-GDA~\cref{alg:sim} linearly converges with iteration complexity
% \begin{align*}
%      \mathcal{\widetilde O} \rbr*{\frac{L^2}{\mu^2}\frac{1}{(1 - \alpha L/\mu)^2}},
% \end{align*}
% where we omit logarithmic factors.
% % \na{use tilde o above}
\end{theorem}

We defer the proof to Appendix~\ref{appendix proofs min max}. 
%From the iteration complexity estimate we see that Inexact Sim-GDA converges linearly whenever $\alpha < \mu/L$.

Our next result shows that the bound $\alpha < \frac{\mu}{L}$ in the previous theorem is tight, i.e., we prove that Inexact Sim-GDA does not converge linearly if $\alpha \geqslant\frac{\mu}{L}$ which tightly matches the guarantee in~\cref{corollary:upper_bound_stongly_monotone_Lipshitz}.

\begin{theorem}\label{teo:sim_gda_lower_bound_alpha}
There exists a function $f$ satisfying \cref{as:conv-conc} with $d_x = d_y = 1$ such that Inexact Sim-GDA with any constant step size $\eta$ loses linear convergence when $\alpha \geqslant \frac{\mu}{L}$.
\end{theorem}
We defer the proof to Appendix~\ref{appendix proofs min max}. 

We also prove the similar statements for Inexact Alt-GDA, though the guarantees for $\alpha$ only match up to a constant factor. We guarantee the following complexity of Inexact Alt-GDA.
\begin{theorem}
\label{corollary:upper_bound_stongly_monotone_Lipshitz_ALT}
Let Assumption \ref{as:conv-conc} hold. If Assumption \ref{as:conv-conc-inexact} holds with $\alpha<\frac{\mu}{\sqrt{2}L}$, then Algorithm \ref{alg:alt} with parameters $(z^0, \eta,\eta)$, where $\eta= \mathcal{O} \rbr*{\frac{\mu}{L^2}}$ generates $z^N$ s.t. $\|z^N-z^*\|_2\leq \varepsilon$ in 
\begin{align*}
     N=\mathcal{O} \rbr*{\frac{L^2}{\mu^2}\frac{1}{(1 - \alpha \sqrt{2}L/\mu)^2}\ln\frac{\|z^0-z^*\|_2}{\varepsilon}} \;\; \text{iterations.}
\end{align*} 
\end{theorem}
And show that Inexact Alt-GDA loses linear convergence if $\alpha \geqslant\frac{\mu}{L}$.
\begin{theorem}\label{teo:alt_gda_lower_bound_alpha}
There exists a function $f$ satisfying \cref{as:conv-conc} with $d_x = d_y = 1$ such that Inexact Alt-GDA with any constant step size $\eta$ loses linear convergence when $\alpha \geqslant \frac{\mu}{L}$.
\end{theorem}

We next move to a refined version of Assumption \ref{as:conv-conc}.
Let constants  $\left(\mu_x,\mu_y,L_x,L_y,L_{xy}\right)$ satisfy $0 \leqslant \mu_x \leqslant L_x$, $0 \leqslant \mu_y \leqslant L_y$, $L_{xy} \geqslant 0$.
\begin{assumption}\label{as:long}
% Function $f\left(x,y\right) \colon \sX \times \sY \to \R$ is differentiable, $\left(\mu_x,\mu_y\right)$-strong-convex-strong-concave (SCSC), and has $\left(L_x,L_y,L_{xy}\right)$-Lipschitz gradients (for shortness, we write $f\in S_{\mu_x\mu_yL_xL_yL_{xy}}$), i.e., $\forall x,x^0,x^1 \in \sX, y,y^0,y^1 \in \sY$
% \begin{align*}
%     &f\left(\cdot,y\right): \mu_x-\text{strongly convex}, \quad  \quad  \quad  \quad  \quad  \quad \, \, \; \; \,
%     f\left(x,\cdot\right): \mu_y-\text{strongly concave},
%     \\
%     &\|\nabla_x f(x^1,y) - \nabla_x f(x^0,y)\|_2 \leqslant L_x \|x^1-x^0\|_2,
%     \|\nabla_x f(x,y^1) - \nabla_x f(x,y^0)\|_2 \leqslant L_{xy} \|y^1 - y^0\|_2, 
%     \\
%     &\|\nabla_y f(x,y^1) - \nabla_y f(x,y^0)\|_2 \leqslant L_y \|y^1 - y^0\|_2 ,
%     \|\nabla_y f(x^1,y) - \nabla_y f(x^0,y)\|_2 \leqslant L_{xy} \|x^1 - x^0\|_2.
% \end{align*}
Function $f(x,y) \colon \sX \times \sY \to \R$ is differentiable, $(\mu_x,\mu_y)$-strong-convex-strong-concave (SCSC), and has $(L_x,L_y,L_{xy})$-Lipschitz gradients (for shortness, we write $f \in S_{\mu_x\mu_yL_xL_yL_{xy}}$), i.e., $\forall x,\widehat{x},\widetilde{x} \in \sX, y,\widehat{y},\widetilde{y} \in \sY$
\begin{align*}
    &f(\cdot,y): \mu_x\text{-strongly convex}, \quad  \quad  \quad  \quad  \quad  \quad \, \, \; \; \,  
    f(x,\cdot): \mu_y\text{-strongly concave}, 
    \\
    &\|\nabla_x f(\widetilde{x},y) - \nabla_x f(\widehat{x},y)\|_2 \leqslant L_x \|\widetilde{x}-\widehat{x}\|_2, \quad
    \|\nabla_x f(x,\widetilde{y}) - \nabla_x f(x,\widehat{y})\|_2 \leqslant L_{xy} \|\widetilde{y} - \widehat{y}\|_2,
    \\
    &\|\nabla_y f(x,\widetilde{y}) - \nabla_y f(x,\widehat{y})\|_2 \leqslant L_y \|\widetilde{y}-\widehat{y}\|_2, \quad \,
    \|\nabla_y f(\widetilde{x},y) - \nabla_y f(\widehat{x},y)\|_2 \leqslant L_{xy} \|\widetilde{x}-\widehat{x}\|_2.
\end{align*}

% For given constants $\left(\mu_x,\mu_y,L_x,L_y,L_{xy}\right)$ that satisfy $0 \leqslant \mu_x \leqslant L_x$, $0 \leqslant \mu_y \leqslant L_y$, $L_{xy} \geqslant 0$, we say that a differentiable function $f\left(x,y\right) \colon \sX \times \sY \to \R$ is $\left(\mu_x,\mu_y\right)$-strong-convex-strong-concave (SCSC) and has $\left(L_x,L_y,L_{xy}\right)$-Lipschitz gradients (i.e., belongs to the class of functions denoted as $S_{\mu_x\mu_yL_xL_yL_{xy}}$) if $\forall x,x^0,x^1 \in \sX, y,y^0,y^1 \in \sY$
% \begin{align*}
%     &f\left(\cdot,y\right): \mu_x-\text{strongly convex}, \quad  \quad  \quad  \quad  \quad  \quad \, \, \; \; \,
%     f\left(x,\cdot\right): \mu_y-\text{strongly concave},
%     \\
%     &\|\nabla_x f(x^1,y) - \nabla_x f(x^0,y)\|_2 \leqslant L_x \|x^1-x^0\|_2,
%     \|\nabla_x f(x,y^1) - \nabla_x f(x,y^0)\|_2 \leqslant L_{xy} \|y^1 - y^0\|_2, 
%     \\
%     &\|\nabla_y f(x,y^1) - \nabla_y f(x,y^0)\|_2 \leqslant L_y \|y^1 - y^0\|_2 ,
%     \|\nabla_y f(x^1,y) - \nabla_y f(x^0,y)\|_2 \leqslant L_{xy} \|x^1 - x^0\|_2.
% \end{align*}
\end{assumption}

\cref{corollary:upper_bound_stongly_monotone_Lipshitz} shows that Inexact Sim-GDA exhibits linear convergence for functions $f \in S_{\mu, L}$ if $\alpha < \frac{\mu}{L}$. Since $S_{\mu \mu LLL} \subset S_{\mu, 2L}$, as a corollary of \cref{corollary:upper_bound_stongly_monotone_Lipshitz}, we also obtain convergence guarantee for $S_{\mu_x\mu_yL_xL_yL_{xy}}$ with $\alpha < \frac{\mu}{2L}$, where $L = \max(L_x, L_y, L_{xy})$ and $\mu = \min(\mu_x, \mu_y)$. This bound may be pessimistic since it does not use the splitting of constants for variables $x$ and $y$. 
For exact Sim-GDA such splitting of constants was used in \cite{lee2024fundamentalbenefitalternatingupdates}, who for $f \in S_{\mu_x\mu_yL_xL_yL_{xy}}$  obtained an iteration complexity estimate (omitting logarithmic factors) 
\begin{align*}
    \mathcal{O} \rbr*{\frac{L_x}{\mu_x} + \frac{L_y}{\mu_y} + \frac{L_{xy}^2}{\mu_x\mu_y}}
\end{align*}
with separated condition numbers and matching the lower bound.
In our next result, we derive a condition for the relative error $\alpha$ that is sufficient for the proof scheme in \cite{lee2024fundamentalbenefitalternatingupdates} to be applicable in the inexact setting. To that end, we introduced relative inexactness to the said scheme and found $\alpha$ that does not break the contraction inequality (see Lemma~\ref{lem:contraction_scsc_smooth}). It is, however, possible that the separation of condition numbers in iteration complexity would require stricter thresholds for $\alpha$.

\begin{theorem} \label{corollary:upper_bound_scsc_smooth} 
Let \cref{as:long} hold with $\mu_x=\mu_y=\mu$ and $L_x=L_y=L$. If Assumption \ref{as:conv-conc-inexact} holds with $\alpha$ satisfying 
\begin{eqnarray*}
    \text{case} \quad  L_{xy}^2 \leqslant \frac{\mu L}{2}:&& \quad  \alpha \left((1+\alpha)^2 + \alpha\right) < \frac{\mu}{2L} - \frac{L_{xy}^2}{2L^2}, \\
    \text{case} \quad  L_{xy}^2 > \frac{\mu L}{2}:&& \quad     \alpha \left((1+\alpha)^2 + \alpha\right) < \frac{\mu^2}{8L_{xy}^2},
\end{eqnarray*}
then Algorithm \ref{alg:sim} with an appropriate choice of parameters has linear convergence.
% Consider $f \in S_{\mu\mu L LL_{xy}}$ (\cref{as:long}). Inexact Sim-GDA retains linear convergence if 
% \begin{eqnarray*}
%     \text{case} \quad  L_{xy}^2 \leqslant \frac{\mu L}{2}:&& \quad  \alpha \left((1+\alpha)^2 + \alpha\right) < \frac{\mu}{2L} - \frac{L_{xy}^2}{2L^2}, \\
%     \text{case} \quad  L_{xy}^2 > \frac{\mu L}{2}:&& \quad     \alpha \left((1+\alpha)^2 + \alpha\right) < \frac{\mu^2}{8L_{xy}^2} .
% \end{eqnarray*}
\end{theorem}
We defer the details on the parameters choice and the proof to Appendix~\ref{appendix proofs min max}.

\section{Nonlinear equations}\label{sec:ne}
Saddle point problem \eqref{spp} and minimization problem \eqref{optim} can both be equivalently reformulated as a special case of a more general nonlinear equation (NE) problem $g(z)=0$, by writing optimality conditions. Thus, NEs \eqref{ne} represent a more general class of problems of independent interest with many applications, see, e.g., the recent reviews \cite{tran-dinh2024extragradient,tran-dinh2025accelerated} and the references therein. This section is devoted to solution of such NEs directly. We start with the following assumptions.
%For given constants $\left(\mu, L\right)$ that satisfy $0 \leqslant \mu \leqslant L$, we say that an
\begin{assumption}\label{as:lipqsm}
Operator $g\left(z\right) \colon \sZ  \to \R^d$ is $\mu$-quasi-strongly-monotone (QSM) and $L$-Lipschitz (for shortness, we write $g \in S_{\mu, L}$), i.e., $\forall x, y \in \sZ$ and for any solution $z^*$ of NE 
    \begin{equation*}
    % \label{eq:lipschitzness}
		\|g(x) - g(y)\|_2 \le L\|x - y\|_2, \quad
	% \end{equation}
 %    \begin{equation}
 %    \label{eq:str_monotonicity}
        \langle g(x), x - z^*\rangle \ge \mu\|x - z^*\|_2^2.
    \end{equation*}
\end{assumption}
\begin{assumption}[Relative noise] \label{as:ne-inexact}
We assume access to inexact values $\widetilde{g}(z)$  of the operator $g(z)$ with relative error, meaning that there exists $\alpha \in [0, 1)$ s.t. $\forall z\in \sZ$
\begin{eqnarray}\label{eq:relative_error_intro_op}
    \|\widetilde{g}(z) - g(z)\|_2 \leqslant \alpha \|g(z)\|_2.
\end{eqnarray}
\end{assumption}

% We assume that inexact values $\widetilde{g}(z)$  of the operator $g(z)$ (with relative accuracy) are available for an algorithm and these values satisfy for some $\alpha \in [0,1)$ and all $z$ the inequality
% \begin{eqnarray}\label{eq:relative_error_intro_op}
%     \|\widetilde{g}(z) - g(z)\|_2 \leqslant \alpha \|g(z)\|_2.
% \end{eqnarray}

% \begin{equation*}\label{ne}
%     F (z)=0.
% \end{equation*}
We consider Inexact ExtraGradient method (EG) listed as~\cref{alg:eg}.
% extragradient method (EG)
% \begin{align*}
%     & z^{k+1/2} = z^{k} - \eta \hspace{2pt} \widetilde{g}(z^k) \\
%     & z^{k+1} = z^{k} - \eta \hspace{2pt} \widetilde{g}(z^{k+1/2}),
% \end{align*}

\begin{algorithm}[H]
\caption{Inexact ExtraGradient Method  (EG).}
	\label{alg:eg}
\begin{algorithmic}[1]
\STATE {\bfseries Input:} $(z^0 = z_{\operatorname{start}}, \eta).$
%\STATE 
%\noindent {\bf Input:} Starting point $x_{\operatorname{start}}$, number of steps $N$.
\FOR {$k = 1, \dots, N$}
        \STATE $z^{k+1/2} = z^{k} - \eta  \widetilde{g}(z^k)$, 
        % \STATE
        $\quad z^{k+1} = z^{k} - \eta \widetilde{g}(z^{k+1/2})$.
\ENDFOR
\STATE \noindent {\bf Output:} $z^N$.
\end{algorithmic}
\end{algorithm}
% \na{Why do we need $\widetilde{g}(z)$ z superscript???}
%where $\widetilde{g}(z)$ defined in~\eqref{eq:relative_error_intro_op}. 

% Our analysis of this algorithm is based on the following assumption.

Our main result for NEs is that linear convergence of Inexact ExtraGradient Method is preserved for $g \in S_{\mu, L}$ if $\alpha$ is lower than some threshold that is $\mathcal{O}\left(\sqrt{\frac{\mu}{L}}\right)$. Thus, for a particular case of min-max problems, this method is more robust than Sim-GDA.

\begin{theorem}\label{teo:eg_alpha_upper_bound}
Let \cref{as:lipqsm} hold. There exists $\widehat \alpha = \mathcal{O}\left(\sqrt{\nicefrac{\mu}{L}}\right)$ s.t. if Assumption \ref{as:ne-inexact} holds with $\alpha < \widehat \alpha$, then Algorithm \ref{alg:eg} with parameters $(z^0,\frac{1}{cL})$ with a numerical constant $c$, generates $z^N$ s.t.
\begin{equation*}
    \|z^N-z^*\|_2^2 \leqslant \|z^0-z^*\|_2^2\left( 1 - \frac{\mu}{2cL} \right)^N.
\end{equation*}
%For $g \in S_{\mu, L}$ (\cref{as:lipqsm}), $\exists \widehat \alpha$: $\widehat  \alpha = \mathcal{O}\left(\sqrt{\nicefrac{\mu}{L}}\right)$, s.t. if $\alpha < \widehat \alpha$ then inexact EG retains linear convergence with contraction factor of $1 - \frac{1}{2}\eta\mu$, $\eta \sim \frac{1}{L}$.
\end{theorem}
We defer additional details and the proof to Appendix~\ref{appendix ne}.

\section{Numerical experiments} \label{sec:numerical_results}
In this section we provide several numerical experiments illustrating the theoretical results of Sections~\ref{section accelerated main} and~\ref{section min max}. Our code is available at \url{https://anonymous.4open.science/r/RelativeError-FE6C}.

% \begin{figure}[H]
% \center{\includegraphics[width=1\linewidth]{PEP/convex/1000_0001.pdf}}
% \caption{\centering PEP comparison RE-AGM (left) and STM (right) , $L = 1000, \mu = 0.001$}
% \label{fig:convex:pep:1000:0001}
% \end{figure}

\begin{figure}[ht]
    \centering
    \begin{minipage}[htp]{0.77\textwidth}
        \centering
        \includegraphics[width=1\linewidth]{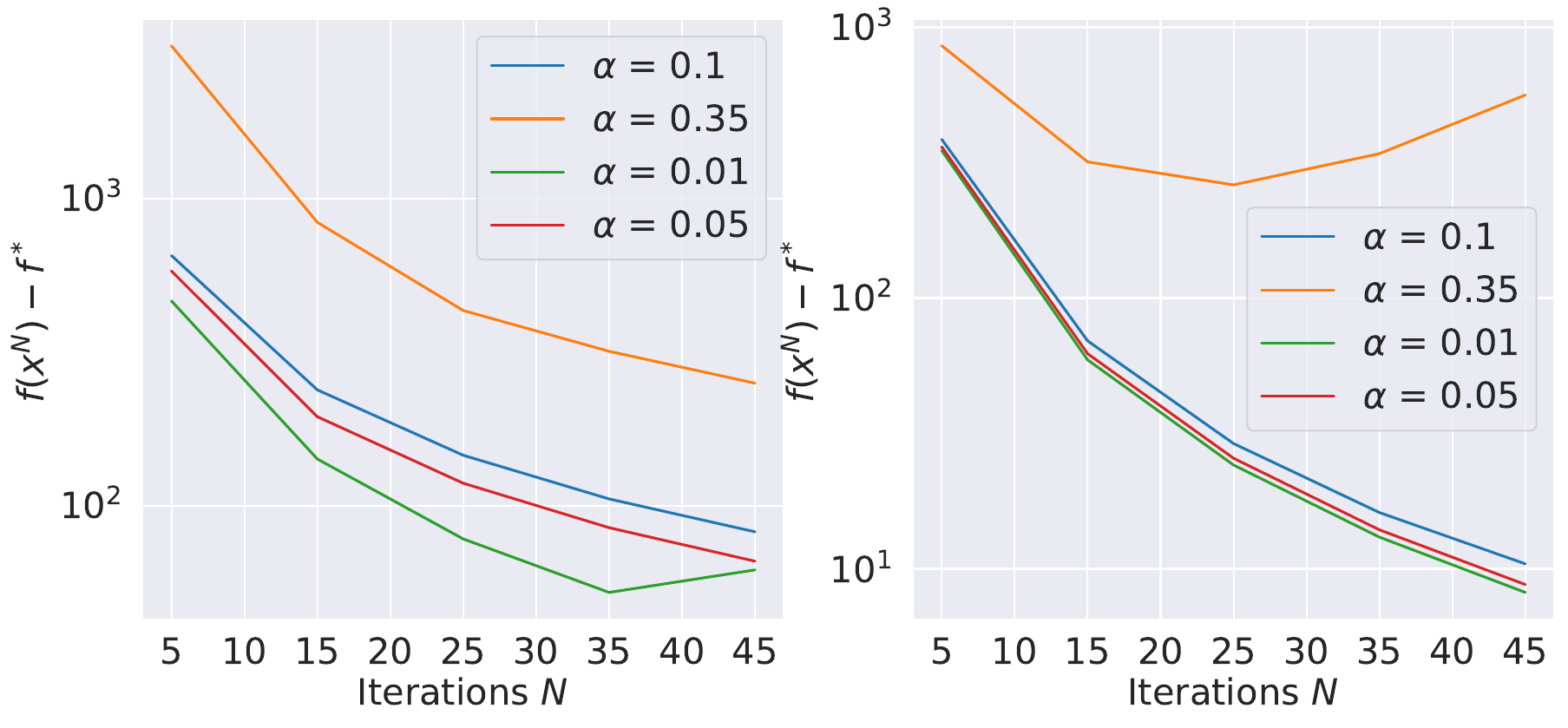}
        \vspace{-5mm}
    	\caption{\centering PEP comparison RE-AGM (left) and STM (right),  $L = 100, \mu = 0.01$}
        \label{fig:convex:pep:100_001}
    \end{minipage}
\end{figure}
\begin{figure}[t]
    \centering
    \begin{minipage}[htp]{0.77\textwidth}
        \centering
        \includegraphics[width=1\linewidth]{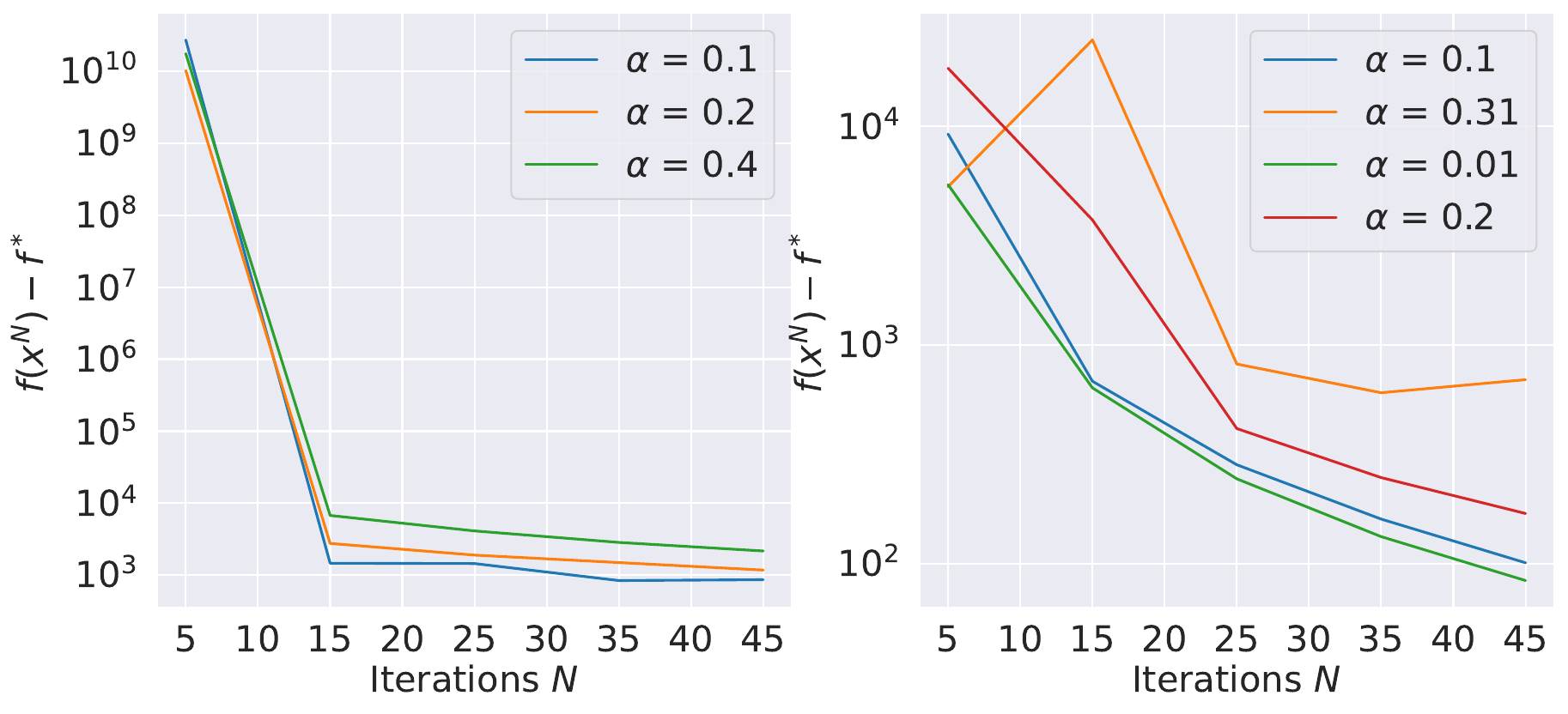}
        \vspace{-5mm}        
    	\caption{\centering PEP comparison RE-AGM (left) and STM (right),  $L = 1000, \mu = 0.001$}
        \label{fig:convex:pep:1000:0001}
    \end{minipage}
\end{figure}

% \begin{figure}[h]
% \center{\includegraphics[width=1\linewidth]{RandomNoise/convex/1000_0001.pdf}}
% \caption{\centering Comparison with random noise, $L = 1000, \mu = 0.001$}
% \label{fig:convex:1000:0001}
% \end{figure}
\begin{figure}[ht]
    \centering
    \begin{minipage}[htp]{0.77\textwidth}
        \centering
        \includegraphics[width=1\linewidth]{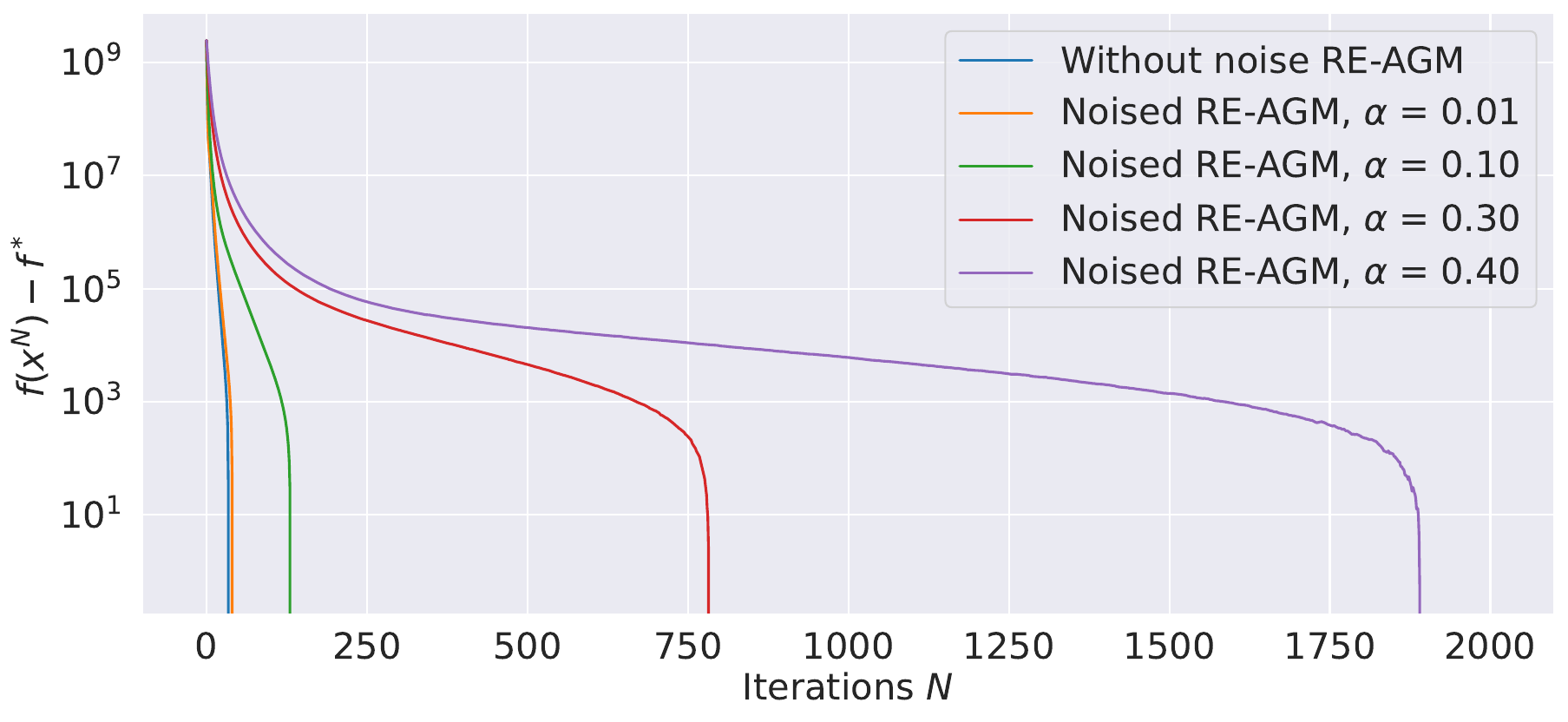}
        \vspace{-6mm}
        \caption{\centering RE-AGM, $L = 10000, \mu = 100$}
        \label{fig:convex:10000:100}
    \end{minipage}
\end{figure}
\begin{figure}[ht]
        \centering
    \begin{minipage}[htp]{0.77\textwidth}
        \centering
        \includegraphics[width=1\linewidth]{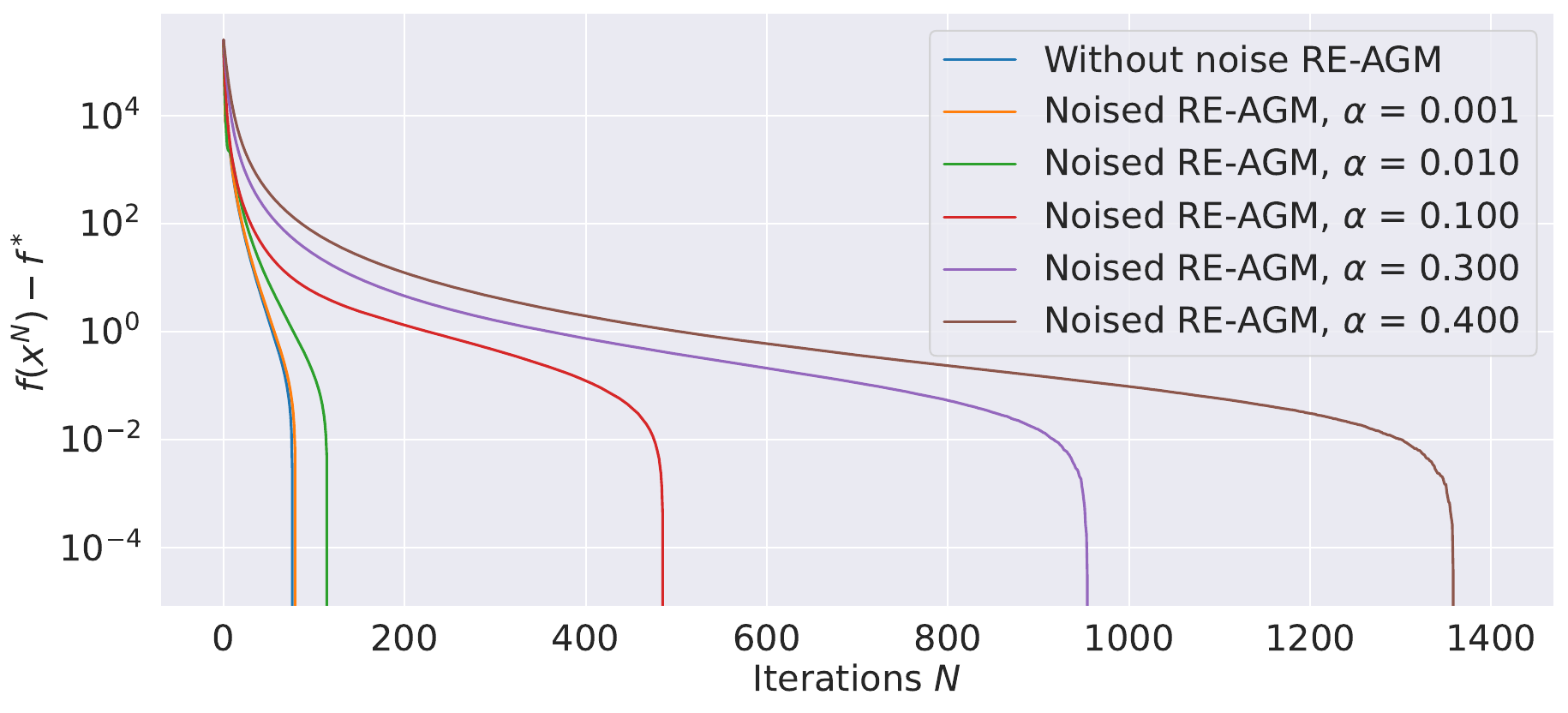}
        \vspace{-6mm}
        \caption{\centering RE-AGM, $L = 1, \mu = 0.0001$}
        \label{fig:convex:1:0001}
    \end{minipage}
\end{figure}
\begin{figure}[ht]
        \centering
    \begin{minipage}[htp]{0.77\textwidth}
        \centering        \includegraphics[width=1\linewidth]{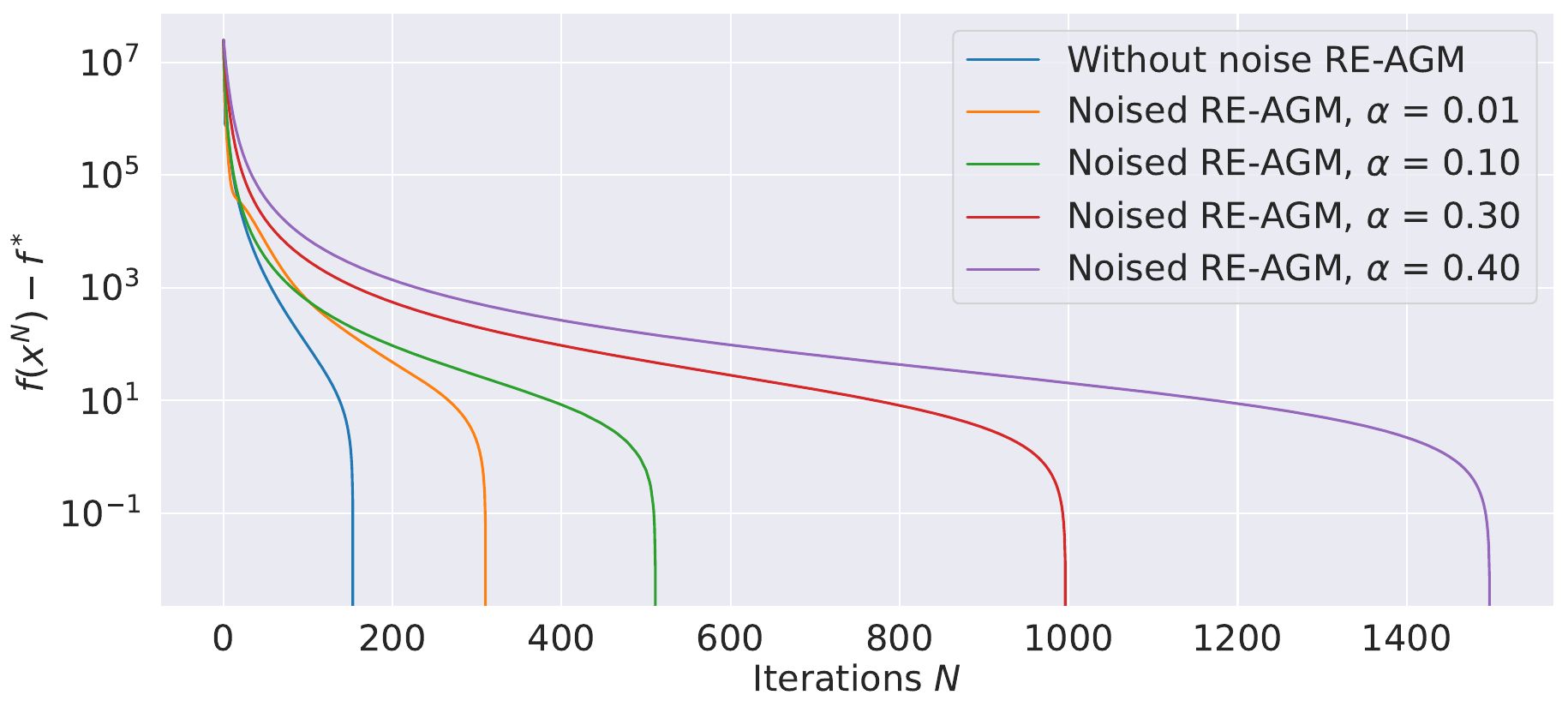}
            \vspace{-6mm}
        \caption{\centering RE-AGM, $L = 100, \mu = 0.01$}
        \label{fig:convex:100:001}
    \end{minipage}
\end{figure}
\begin{figure}[ht]
        \centering
    \begin{minipage}[htp]{0.77\textwidth}
        \centering
        \includegraphics[width=1\linewidth]{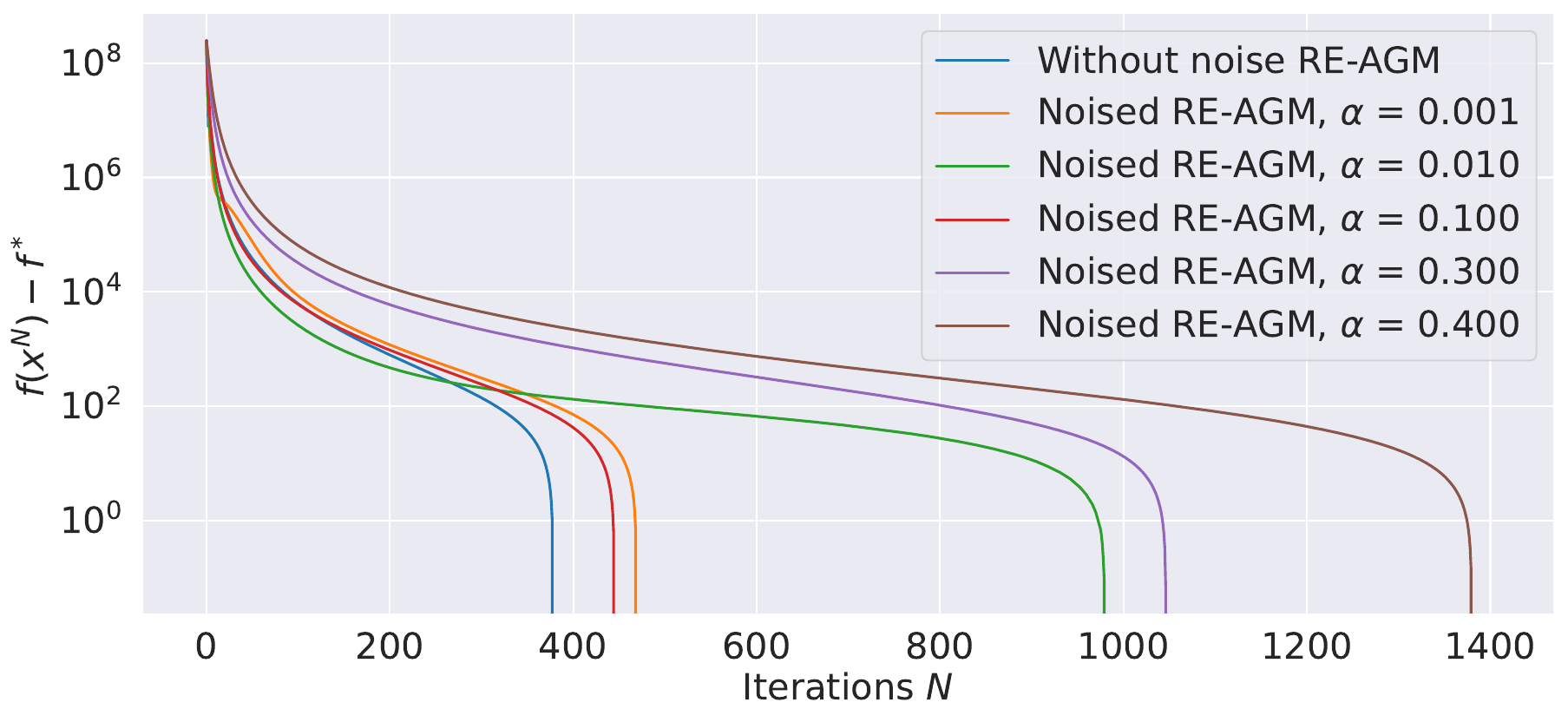}
        \vspace{-6mm}
        \caption{\centering RE-AGM, $L = 10000, \mu = 0.001$}
        \label{fig:convex:10000:0001}
    \end{minipage}
\end{figure}

% \begin{figure}[H]
% \center{\includegraphics[width=1\linewidth]{RandomNoise/convex/1_0001.pdf}}
% \caption{\centering RE-AGM, $L = 1, \mu = 0.001$}
% \label{fig:convex:1:0001app}
% \end{figure}

% \begin{figure}[H]
% \center{\includegraphics[width=1\linewidth]{RandomNoise/convex/10000_100.pdf}}
% \caption{\centering RE-AGM, $L = 10000, \mu = 100$}
%  \label{fig:convex:10000:100_app}
% \end{figure}

\begin{figure*}[ht]
    \centering
    \begin{minipage}[htp]{0.77\textwidth}
        \centering
        \includegraphics[width=1\linewidth]{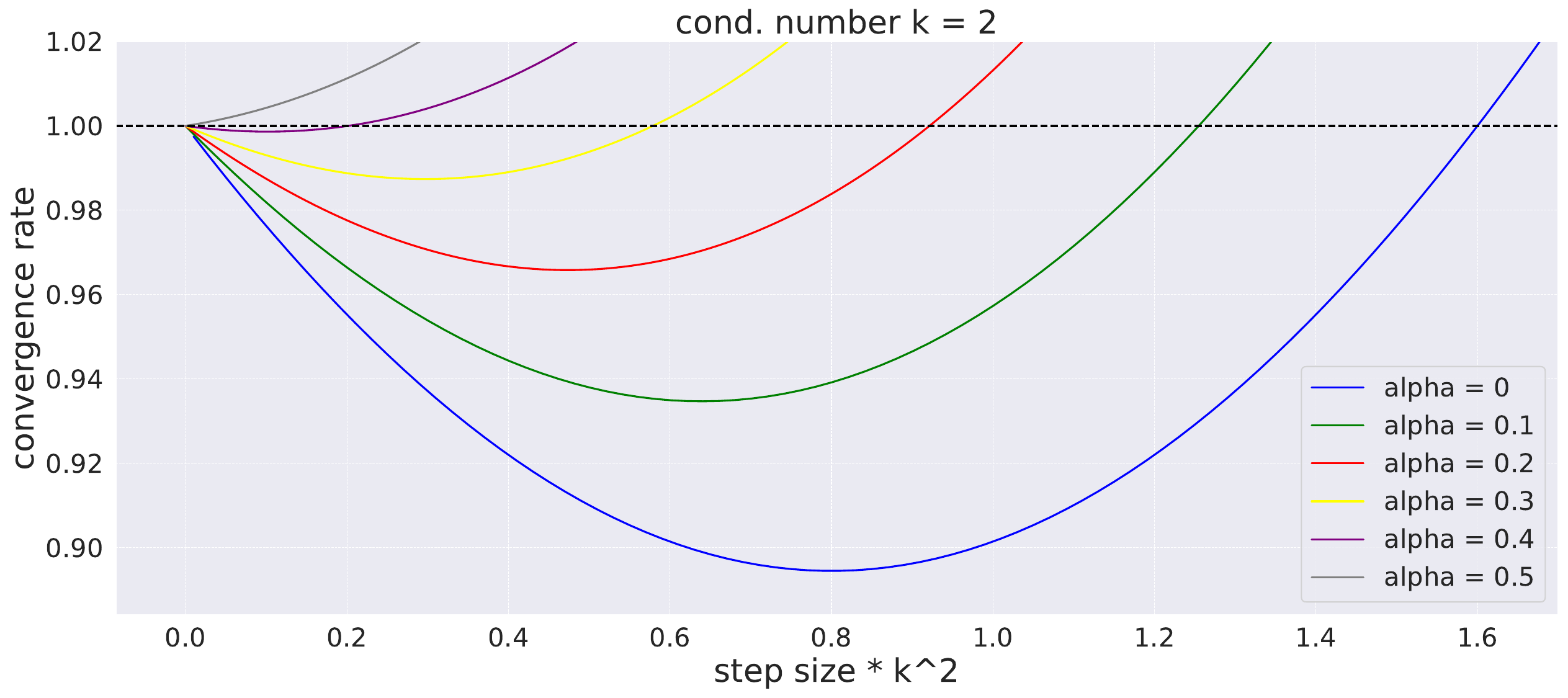}
        \vspace{-6mm}
        \caption{\centering Inexact Sim-GDA ($\kappa=2$)}
        \label{fig:lyapunov_sim2}
    \end{minipage}
    \begin{minipage}[htp]{0.77\textwidth}
        \centering
        \includegraphics[width=1\linewidth]{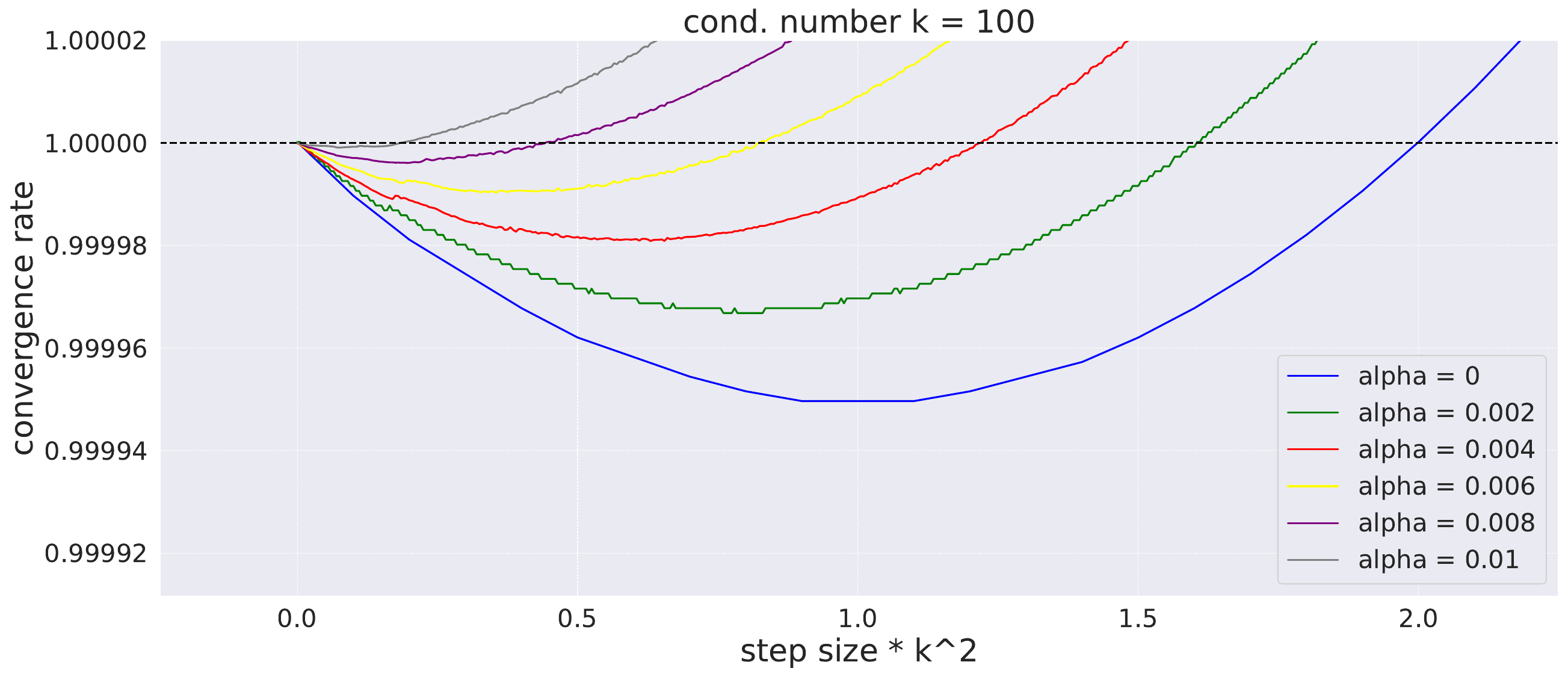}
        \vspace{-6mm}
        \caption{\centering Inexact Sim-GDA ($\kappa=100$)}
        \label{fig:lyapunov_sim100}
    \end{minipage}
\end{figure*}
\begin{figure*}[ht]
    \centering
    \begin{minipage}[htp]{0.77\textwidth}
        \centering
        \includegraphics[width=1\linewidth]{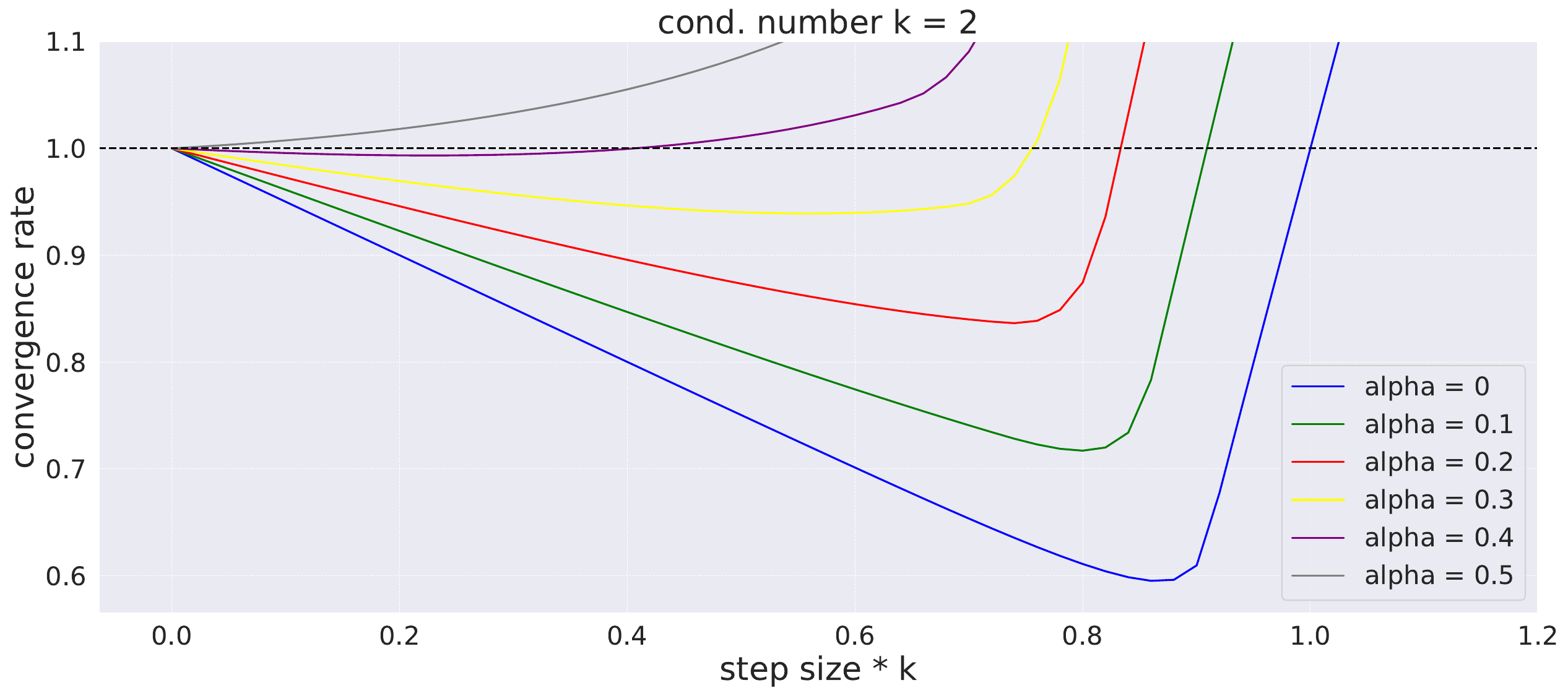}
        \vspace{-6mm}
        \caption{\centering Inexact Alt-GDA ($\kappa=2$)}
        \label{fig:lyapunov_alt2}
    \end{minipage}
    \begin{minipage}[htp]{0.77\textwidth}
        \centering
        \includegraphics[width=1\linewidth]{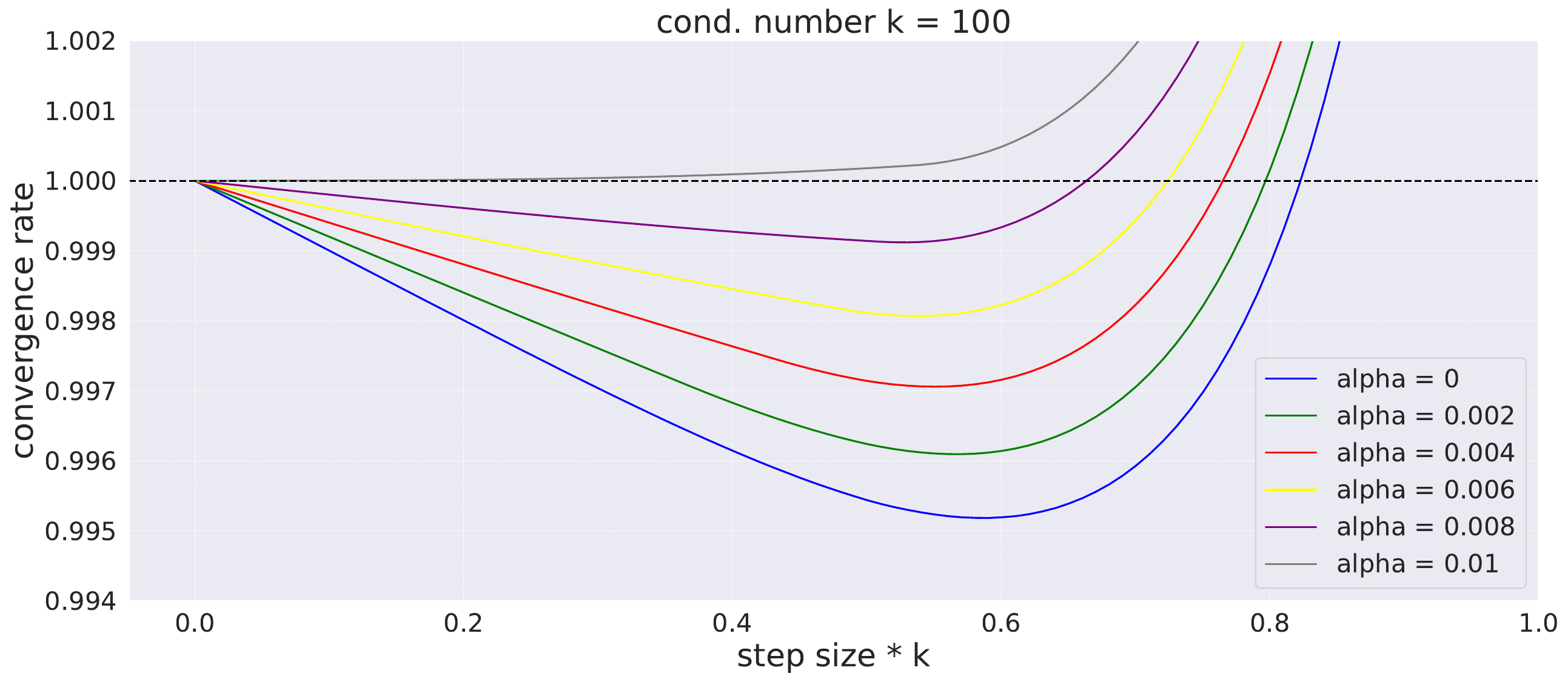}
        \vspace{-6mm}
        \caption{\centering Inexact Alt-GDA ($\kappa=100$)}
        \label{fig:lyapunov_alt100}
    \end{minipage}
\end{figure*}
\begin{figure*}[ht]
    \centering
    \begin{minipage}[htp]{0.77\textwidth}
        \centering
        \includegraphics[width=1\linewidth]{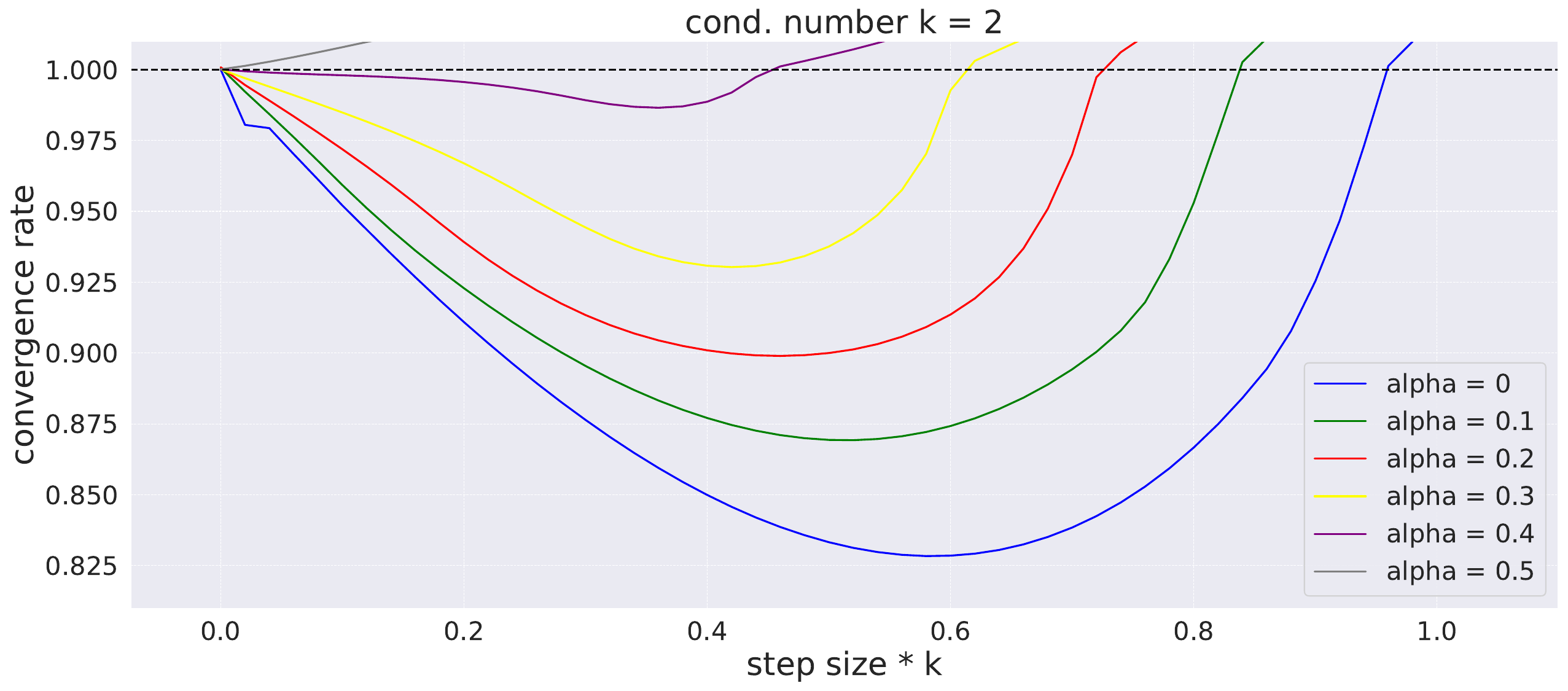}
        \vspace{-6mm}
        \caption{\centering Inexact EG ($\kappa=2$)}
        \label{fig:lyapunov_eg2}
    \end{minipage}
    \begin{minipage}[htp]{0.77\textwidth}
        \centering
        \includegraphics[width=1\linewidth]{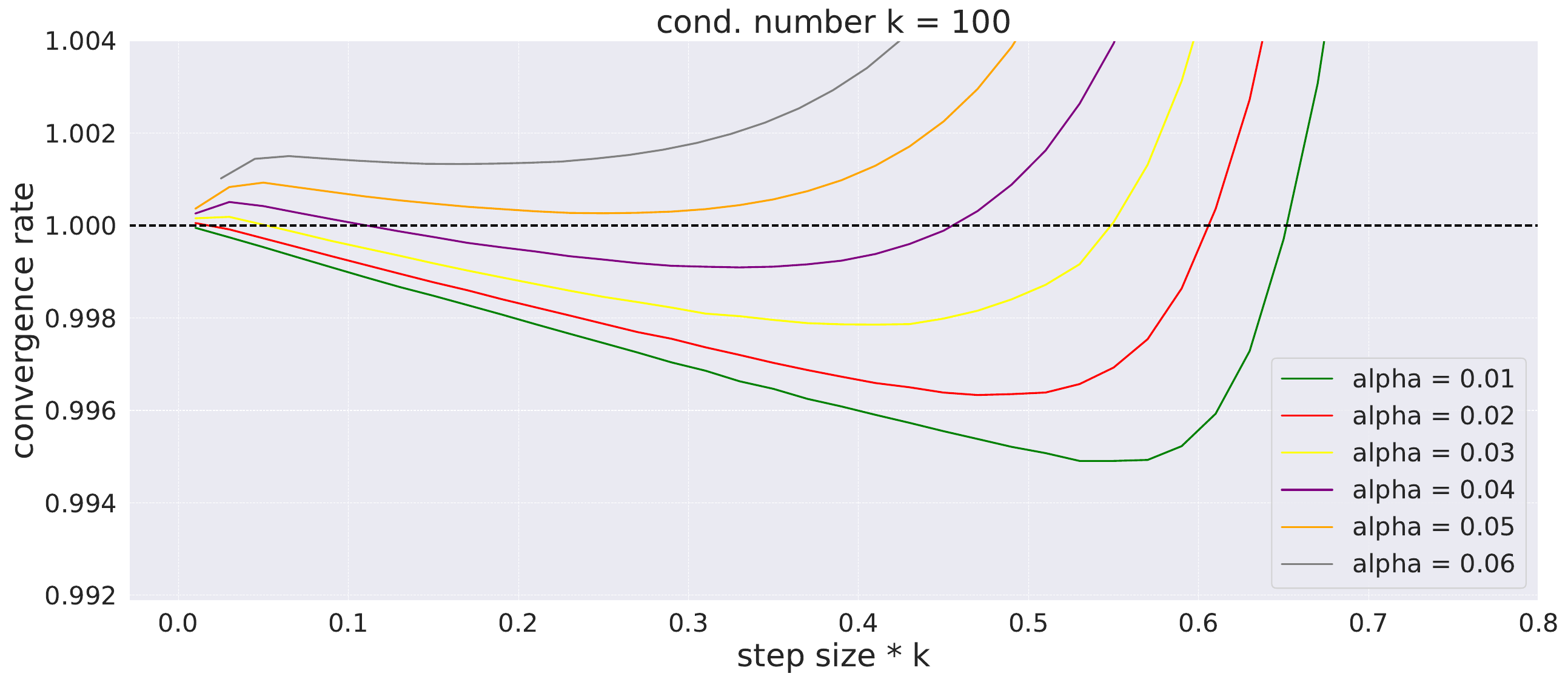}
        \vspace{-6mm}
        \caption{\centering Inexact EG ($\kappa=100$)}
        \label{fig:lyapunov_eg100}
    \end{minipage}
\end{figure*}

\subsection{Minimization problems}
For problem \eqref{optim}, we compare our RE-AGM Algorithm~\ref{alg:nesterov corrected} introduced in Section~\ref{section accelerated main} with the similar triangles method (STM) analyzed in the inexact setting in \cite{vasin2023accelerated,kornilov2023intermediate} and listed as Algorithm~\ref{alg:stm}. 

% convex case and method introduced at Section~\ref{section accelerated} a comparison is made with STM (similar triangles method) - Algorithm~\ref{alg:stm}, considered at~\cite{vasin2023accelerated}. For a similarly formulated method for a not strongly convex case in the paper~\cite{kornilov2023intermediate} accelerated convergence was achieved using restart techniques for $\alpha \lesssim \left(\mu/L\right)^{1/2}$.

% \begin{minipage}[htp]{0.48\textwidth}        
\begin{algorithm}[H]
\caption{ STM (Similar Triangles Method)}
	\label{alg:stm}
\begin{algorithmic}[1]
\STATE {\bfseries Input:} $(L, \mu, x_{\operatorname{start}}, \alpha)$.
%\STATE 
%\noindent {\bf Input:} Starting point $x_{\operatorname{start}}$, number of steps $N$.
\STATE {\bf Set} 
$y^0 = x_{start}$, $\alpha_0 = \frac{1}{L}$, $A_0 = \frac{1}{L}$, $\widehat \mu = \frac{\mu}{2}$,
% \STATE {\bf Set} 
$z^0 = y^0 - \alpha_0 \widetilde{\nabla} f(y^0)$, $x^0 = z^0$.
\FOR {$k = 1, \dots, N$}
        \STATE $\alpha_k = \frac{1 + \widehat \mu A_{k - 1}}{2L} + \sqrt{\frac{(1 + \widehat \mu A_{k - 1})^2}{4L^2} + \frac{A_{k - 1}(1 + \widehat \mu A_{k - 1})}{L}}$, 
        % \STATE
        $\quad A_k = A_{k - 1} + \alpha_k,$
        \STATE $y^k = \frac{A_{k - 1} x^{k - 1} + \alpha_k z^{k - 1}}{A_k}$, 
        % \STATE 
        $\quad x^k = \frac{A_{k - 1} x^{k - 1} + \alpha_k z^k}{A_k}$,
        \STATE 
        $z^k = z^{k - 1} - \frac{\alpha^{k}}{1 + \widehat \mu A_k}\left(\widetilde{\nabla}f(y^k) + \widehat \mu (z^{k - 1} - y^k) \right)$.
\ENDFOR
\STATE 
\noindent {\bf Output:} $x^N.$
\end{algorithmic}
\end{algorithm}
% \end{minipage}

The first experiment is based on the technique called Performance Estimation Problem (PEP)~\cite{pepit2022,drori2014performance}. This technique allows one to construct worst-case functions for optimization algorithms. We consider Algorithm~\ref{alg:nesterov corrected} and Algorithm~\ref{alg:stm} that are run for $50$ iterations and using the PEP technique with different values of $\alpha$  find such worst-case functions in order to establish the fact of convergence/divergence of the method. Such PEP-method is based on the optimization problem
\begin{eqnarray*}
    \max_{f, x^0, n} f(x^N) - f^* 
\end{eqnarray*}
subject to:
\begin{equation*}
    \begin{gathered}
        f: \mathbb{R}^d \to \mathbb{R} \text{ satisfies Assumption}~\ref{as:Lsmooth+muConv}, 
        % \\
        \|x^0 - x^* \|_2 \leqslant R, 
        \\
        \|\widetilde{\nabla} f(y^k) - \nabla f(y^k) \|_2 \leqslant \alpha \|\nabla f(y^k) \|_2, 
        \\
        y^{k + 1}, u^{k + 1}, x^{k + 1} = \text{RE-AGMstep}\left(y^k, u^k, x^k, \widetilde{\nabla} f(y^k)\right)
    \end{gathered}
\end{equation*}
where $R, N, \alpha$ are given parameters. 

% One of the methods for obtaining convergence estimates is the PEP (Performance Estimation Problem) method~\cite{pepit2022},\cite{drori2014performance}. By running Algorithm~\ref{alg:nesterov corrected} and Algorithm~\ref{alg:stm} for $50$ iterations using the PEP technique with different $\alpha$ values we can use the fact of convergence/divergence of the method. PEP-method considers optimization problem:
% \begin{equation*}
%     f(x^N) - f^* \to \max_{f, x^0, n}
% \end{equation*}
% subject to:
% \begin{equation*}
%     \begin{gathered}
%         f: \mathbb{R}^d \to \mathbb{R} \text{ satisfies Assumption}~\ref{as:Lsmooth+muConv}, \\
%         \|x^0 - x^* \|_2 \leqslant R, \\
%         \|\widetilde{\nabla} f(y^k) - \nabla f(y^k) \|_2 \leqslant \alpha \|\nabla f(y^k) \|_2, 
%         \\
%         y^{k + 1}, u^{k + 1}, x^{k + 1} = \text{RE-AGMstep}\left(y^k, u^k, x^k, \widetilde{\nabla} f(y^k)\right)
%     \end{gathered}
% \end{equation*}
% where $R, N, \alpha$ are given parameters. 

Such optimization problems can be reformulated as semi-definite programming problems (SDPs) and solved numerically using standard solvers.

Figure~\ref{fig:convex:pep:100_001} demonstrates the advantage of RE-AGM over STM with relative error (\cref{eq:relative_error}) when  $\alpha = 0.35$ since STM diverges in this case. 
% \begin{figure}[H]
% \center{\includegraphics[width=1\linewidth]{RandomNoise/convex/100_001.png}}
% \caption{\centering Comparison with random noise, $L = 100, \mu = 0.01$}
% \label{fig:convex:100:001}
% \end{figure}
% Additional plot to Figure~\ref{fig:convex:pep:100_001} we provide comparison to Algorithm~\ref{alg:stm} for smaller value $\frac{\mu}{L}$.
For a smaller value of $\nicefrac{\mu}{L}$, Figure~\ref{fig:convex:pep:1000:0001} demonstrates that under certain values of $\alpha$, STM may start to diverge while our RE-AGM converges, however, with a slower rate. 
We should note, that 
the current theoretical results both for STM and RE-AGM are that they are guaranteed to converge if  $\alpha \lesssim \sqrt{\nicefrac{\mu}{L}}$, but the values of $\sqrt{\nicefrac{\mu}{L}}$ for Figure~\ref{fig:convex:pep:100_001} and Figure~\ref{fig:convex:pep:1000:0001} are equal $0.01$ and $0.001$ respectively and both algorithms preserve convergence for larger value $\alpha = 0.2$. Yet, our proposed RE-AGM demonstrates better robustness.

% To visualize convergence of Algorithm~\ref{alg:nesterov corrected} one can use random, uniformly distributed noise of a given scale $\alpha$, since it is problematic to run PEP for a large number of steps.

In our second experiment, we illustrate the convergence of RE-AGM for a specific function constructed in~\cite{nesterov2018lectures}:
\begin{eqnarray} \label{mu > 0 nesterov func}
    f(x) = \frac{\mu \left(\kappa - 1\right)}{8} \left(x_1^2 + \displaystyle\sum_{j = 1}^{n - 1} {\left(x_j - x_{j + 1}\right)^2} - 2x_1\right) 
    % \\
    + \frac{\mu}{2} \|x \|_2^2,
\end{eqnarray}
where $ \kappa = \nicefrac{L}{\mu}$. This function is usually referred to as the worst-case function from the class of functions satisfying Assumption \ref{as:Lsmooth+muConv} since the lower iteration complexity bound for any first-order method in this class is attained on this function.
Given that it is computationally problematic to run PEP-technique for a large number of steps, we use random uniform relative noise for each $\alpha$, i.e., we sample noise $r^k$ s.t. $\|r^k\|_2 = \alpha \|\nabla f(y^k) \|_2$. 
% \na{move figres, add references}
Figures~\ref{fig:convex:10000:100}-\ref{fig:convex:10000:0001} demonstrate convergence with different values of $\alpha$ and 
% Figures~\ref{fig:convex:10000:100}-\ref{fig:convex:10000:0001} 
show that RE-AGM preserves convergence for different values $\nicefrac{\mu}{L}$, but has worse convergence rate, which experimentally proves Lemma~\ref{upper bound for lemma} and Theorem~\ref{nesterov acc corrected convergence tau}.
% , however on plot convergence is better, than result of Theorem~\ref{nesterov acc corrected convergence tau}.

% More experiments for convex optimization provided in Appendix~\ref{appendix convex exps}.

\subsection{Sadle Point Problems and Nonlinear Equations}
We now move to the numerical experiments that verify our theoretical results for Sim-GDA and allow finding the actual thresholds for $\alpha$  for Alt-GDA and EG. To that end, we used PEP-like approach of automatic generation of Lyapunov functions that was introduced and thoroughly described in \cite{pmlr-v80-taylor18a}. The key idea of the approach is to reformulate the existence of a quadratic Lyapunov function that would verify a given linear convergence rate as a convex feasibility problem (SDP). The approach is tight meaning that if the problem does not have a solution then there is no such quadratic Lyapunov function that will verify the rate. The linear convergence rate itself acts as a parameter of the problem and it is possible to find the best verifiable linear convergence rate by applying a binary search.

We consider a class of SPPs with the objective $f(x) + y^\top A x - g(y)$, where both functions $f$ and $g$ satisfy Assumption \ref{as:Lsmooth+muConv} and matrix $A$ has largest singular value bounded by $L$.  
%composite functions with bilinear coupling 
% \begin{assumption} \label{as:Kclass}
% We say that convex-concave function $F(x,y)$ is a composite with bilinear coupling (i.e. belongs to the class of functions denoted as $K_{\mu_x\mu_yL_xL_yL_{xy}}$) if there exist such $f(x)$: $\mu_x$-strongly-convex with $L_x$-Lipschitz gradient, $g(y)$: $\mu_y$-strongly-convex with $L_y$-Lipschitz gradient and matrix $A$ with largest singular value bounded by $L_{xy}$ such that
% \begin{align*}
%     F(x,y) = f(x) + y^\top A x - g(y).
% \end{align*}
% \end{assumption}
Because the used PEP-like approach is extremely technical, we omit its detailed description. We note that the only difference from~\cite{pmlr-v80-taylor18a}, which we used as a guideline, lies in the convex-concave setting and the interpolation conditions for functions from the above class.
% $K_{\mu_x\mu_yL_xL_yL_{xy}}$ (see \cref{as:Kclass}). 
These interpolation conditions were presented in~\cite{Krivchenko_Gasnikov_kovalev_2024} and are essentially a combination of smooth strongly-convex interpolation conditions of \cite{Taylor2016} and the interpolation conditions for linear operator with bounded largest singular value~\cite{Bousselmi2024}. Therefore, adaptation of the said approach to the the setting of our paper is fairly straightforward, yet requires some effort.

% As said, we consider $K_{\mu_x\mu_yL_xL_yL_{xy}}$ -- the class of functions defined in \cref{as:Kclass}. 
% We let $\mu_x = \mu_y = \mu$, $L_x= L_y = L_{xy} = L$ and 
We let $\eta_x = \eta_y = \eta$ for every method and plot the curves $\rho_{min}(\eta)$  where $\rho_{min}$ is the smallest contraction factor that we were able to verify (i.e. the PEP problem was feasible). Contraction $\rho$ (i.e., convergence rate) is defined as in~\cite{pmlr-v80-taylor18a}. Note that the smaller is the contraction factor, the faster is the convergence. \par

The results for Sim-GDA (Figures~\ref{fig:lyapunov_sim2},\ref{fig:lyapunov_sim100}) and Alt-GDA (Figures~\ref{fig:lyapunov_alt2},\ref{fig:lyapunov_alt100}) show that both methods lose linear convergence for $\alpha \gtrsim \frac{\mu}{L}$. EG (Figures~\ref{fig:lyapunov_eg2},\ref{fig:lyapunov_eg100}) on the other hand, converges when $\alpha \gtrsim 0.5 \sqrt{\frac{\mu}{L}}$. We additionally verified the results with PEP: the setup is identical to the one described in~\cite{Krivchenko_Gasnikov_kovalev_2024} except the inexactness of the gradients. Further results of numerical experiments can be found in Appendix~\ref{appendix Lyapunov min max}.

These results imply that, for smooth strongly monotone saddle point problems, even non-accelerated first-order methods cease to converge in the worst case if the relative inexactness of the gradient reaches a certain threshold that depends on the condition number. For instance, numerical experiments strongly suggest that linear convergence of Sim-GDA and Alt-GDA ceases when $\alpha$ crosses the threshold of $\sim \frac{\mu}{L}$. With EG, however, the threshold appears to be $\sim \sqrt{\frac{\mu}{L}}$ which is very surprising and counter-intuitive because it means that EG is more robust than Sim-GDA and Alt-GDA. Given the connection of EG with accelerated methods for convex functions \cite{DBLP:journals/corr/abs-2202-04640, Cohen2020RelativeLI}, it seems like an opposite situation to how, in convex setting, non-accelerated methods are more robust than accelerated ones when solving minimization problems \cite{devolder2013intermediate}. \par

\newpage

\section*{Declarations}

\begin{itemize}
\item Funding

The research is supported by the Ministry of Science and Higher Education of the Russian Federation, project No. FSMG-2024-0011.

\item Conflict of interest.

The authors have no further relevant financial or non-financial interests to disclose except the funding mentioned above.

\item Ethics approval and consent to participate.

Not applicable.
\item Consent for publication

Not applicable.
\item Data availability 

Not applicable.
\item Materials availability

Not applicable.
\item Code availability.

The code is available at \url{https://anonymous.4open.science/r/RelativeError-FE6C}.
\item Author contribution.

All the authors contributed to the formation of ideas underlying the paper, obtaining the main results, writing the text.
\end{itemize}

\begin{appendices}
% \section{Missing Proofs~\cref{section accelerated main}}\label{appendix accelerated}
\section[Missing Proofs]{Missing Proofs~\texorpdfstring{\cref{section accelerated main}}{Section~\ref{section accelerated main}}}
\label{appendix accelerated}

Following \cite{nesterov2018lectures}, let us introduce a parameterized set of functions $\Psi$, its element defined for $c \in \mathbb{R}, \kappa \in \mathbb{R}^{++},$ and $u \in \mathbb{R}^d$,  as follows
\begin{equation} \label{qudaratic family}
    \psi(x | c, \kappa, u) = c + \frac{\kappa}{2} \| x - u \|_2^2, \quad \forall x \in \mathbb{R}^d. 
\end{equation}

According to~\cite{nesterov2018lectures}, we can mention to the following useful property of the class $\Psi$.

\begin{lemma} \label{psi comb}
Let $\psi_1, \psi_2 \in \Psi$. Then $\forall \eta_1, \eta_2, c_1, c_2 \in \mathbb{R}, \forall \kappa_1, \kappa_2 \in \mathbb{R}^{++}$, and $\forall u, v \in \mathbb{R}^d$, we have 
\[
    \eta_1 \psi_1(x|c_1, \kappa_1, u) + \eta_2 \psi_2(x|c_2, \kappa_2, v) = \psi_3(x|c_3, \kappa_3, w), 
\]
where 
\begin{align*}
    c_3 = \eta_1 c_1 + \eta_2 c_2 + \frac{\eta_1\eta_2\kappa_1\kappa_2}{2(\eta_1\kappa_1 + \eta_2\kappa_2)} \|u  - v \|_2^2, \quad \kappa_3 = \eta_1\kappa_1 + \eta_2\kappa_2, \quad w = \frac{\eta_1\kappa_1 u +\eta_2\kappa_2 v}{\eta_1\kappa_1 + \eta_2\kappa_2}.
\end{align*}
\end{lemma}

Also from~\cite{nesterov2018lectures} we mention the following simple lemma.
\begin{lemma} \label{conv lemma}
Let $0 < \lambda < 1$, $\psi_0, \psi \in \Psi$, such that 
\begin{equation*}
f(z) \leqslant \min_{x \in \mathbb{R}^d} \psi(x) \quad \forall z \in \mathbb{R}^d,
\end{equation*}
and 
\begin{equation*}
     \psi(x) \leqslant \lambda \psi_0(x) + (1 - \lambda) f(x), \quad \forall x \in \mathbb{R}^d.
\end{equation*}
Then
\begin{equation*}
    f(z) - f^* \leqslant \lambda (\psi_0(x^*) - f^*), \quad \forall z \in \mathbb{R}^d. 
\end{equation*}
\end{lemma}

Now, let us proof the following lemma which provides a lower bound of the function $f$ from the set $\Psi$. 

\begin{lemma} \label{lower bound}
Let $f$ be a $2\mu$-strongly convex function and $\widetilde{\nabla}{f}$ satisfy  \eqref{eq:relative_error} for any $\alpha \in [0,1)$, and let $\rho = 1+ \alpha$. Then for any $z \in \mathbb{R}^d$, there is $\phi_{z} \in \Psi$, such that 
%If function $f$ is $2\mu$-strongly convex~\eqref{eq:str_cvx}, $\widetilde{\nabla}{f}$ satisfy~\eqref{eq:relative_error}, then $(\forall x^0 \in \mathbb{R}^d )\left( \exists \phi_{x^0} \in \PSI \right): \;\phi_{x^0}(x) \leqslant f(x)$. 
\begin{equation*}
    \phi_{z}(x) = f(z) - \frac{\rho^2 + \alpha^2}{2\mu} \|\nabla f(z) \|_2^2 + \frac{\mu}{2} \left\|x - z + \frac{1}{\mu} \widetilde{\nabla} f(z) \right \|_2^2 \leqslant f(x), \quad \forall x \in \mathbb{R}^d.
    \end{equation*}
\end{lemma}
\begin{proof}
Let us set $c := f(z) - \frac{\rho^2 + \alpha^2}{2\mu}\|\nabla{f}(z)\|_2^2,  \kappa := \mu, u := z - \frac{1}{\mu} \widetilde{\nabla}f(z)$. Then
\begin{eqnarray*}
    \phi_{z}(x)
    % \lefteqn{\phi_{z}(x)}
    % \\
    & = & f(z) - \frac{(\rho^2 + \alpha^2)\|\nabla{f}(z)\|_2^2}{2\mu} + \frac{\mu}{2} \left\|x - z + \frac{1}{\mu} \widetilde{\nabla}f(z) \right\|_2^2
    \\
    & = & f(z) - \frac{(\rho^2 + \alpha^2)\|\nabla{f}(z)\|_2^2}{2\mu} + \frac{\mu}{2} \|x - z \|_2^2 + \langle \widetilde{\nabla}f(z), x - z \rangle + \frac{1}{2\mu} \| \widetilde{\nabla}f(z) \|_2^2  
    \\
    & \overset{(i)}{\leqslant}& f(z) - \frac{(\rho^2 + \alpha^2)\|\nabla{f}(z)\|_2^2}{2\mu} + \frac{\mu}{2} \|x - z \|_2^2 + \langle \nabla{f}(z), x - z \rangle + \frac{\mu}{2} \|x - z \|_2^2 
    \\
    && + \frac{\alpha^2}{2\mu} \|\nabla{f}(z) \|_2^2 + \frac{\rho^2}{2\mu} \| \nabla f(z) \|_2^2
    \\
    & \overset{(ii)}{\leqslant} & f(x),
\end{eqnarray*}
where $(i)$ follows from $ \langle \widetilde{\nabla}f(z), x - z \rangle = \langle \widetilde{\nabla}f(z) + \nabla f(z) - \nabla f(z), x -z \rangle $, Fenchel-Young inequality, inequality~\eqref{eq:relative_error} and~\cref{growth condition a}; $(ii)$ follows from the fact that $f$ is $2\mu$-strongly convex.
\end{proof}

Now, for Algorithm \ref{alg:nesterov corrected}, we have $u^0 = x_{\operatorname{start}}$ and $x^0 \in \mathbb{R}^d$ s.t. $u^0 = x^0$. Let $u^k$ be defined as in item 11 of Algorithm \ref{alg:nesterov corrected}, and $a_k \, (\forall k \geqslant 0)$ defined in \cref{alg:nesterov corrected}. Let us define, recursively, a pair of sequences $\{\psi^k(x)\}_{k \geqslant 0}, \{\lambda_k\}_{k \geqslant 0}$ corresponding to Algorithm \ref{alg:nesterov corrected} as follows
\begin{align}\label{function seq}
    &  \lambda_0 = 1, \quad {c_0 = f(x^0)}, \quad \lambda^{k+1} = (1-a_k) \lambda_k, \quad \forall k \geqslant0,
    \\& \psi^k(x| c^k, \mu, u^k) = c_k+ \frac{\mu}{2} \|x - u^k\|_2^2, \quad \forall k \geqslant 0.  \label{psi^k_Alg1}
\end{align}

Let $y^k\, \forall k \geqslant 0$ be generated by Algorithm \ref{alg:nesterov corrected} (see item 10 of this algorithm), and $\phi_{y^k}$ be the function that is obtained from Lemma \ref{lower bound}. We define, recursively
\begin{align} \label{psi seq def}
    \psi_{k + 1}(x | c_{k + 1}, \mu, u^{k + 1}) := (1 - a_k) \psi^k(x | c^k, \mu, u^k)+  a_k \phi_{y^k}(x), \quad \forall k \geqslant 0.
\end{align}
Using Lemma~\ref{psi comb} we obtain parameters of $\{ \psi^k \}$ that correspond to $ u^k$ in Algorithm~\ref{alg:nesterov corrected}.

\begin{lemma} \label{a_k correctness}
\moh{Let $\alpha \in [0,1), m = 1- 2\alpha, s = 1+2\alpha + 2\alpha^2$, and $a_k$ be the largest root of the quadratic equation
\begin{equation}\label{eq:quadratic}
    m a_k^2 + (s - m) a_k - q = 0,
\end{equation}
where $q = \frac{\mu}{\widehat{L}}$ and $\widehat{L} = \frac{1+ \alpha}{(1-\alpha)^3} L$. Then, $a_k <1 \, \forall k \geqslant 0$. In addition,  
\begin{enumerate}
    \item If $\alpha \leqslant \frac{\sqrt{2} - 1}{18} \sqrt{\frac{\mu}{L}}$, then $a_k \geqslant \frac{1}{10} \sqrt{\frac{\mu}{L}} \; \forall k \geqslant 0. $ 

    \item If $\alpha = \frac{1}{3} \left(\frac{\mu}{L} \right)^{\frac{1}{2} - \tau}$, then $a_k \geqslant \frac{1}{10} \left(\frac{\mu}{L} \right)^{\frac{1}{2} + \tau} \; \forall k \geqslant 0,$ where $0 \leqslant \tau \leqslant \frac{1}{2}$.
\end{enumerate}
}
%Let $a_k$ the largest root of the equation:
%\begin{equation*}
%\begin{gathered}
    %m a_k (1 - a_k) - s a_k = -q, \\
    %m a_k^2 + (s - m) a_k - q = 0. 
%\end{gathered}
%\end{equation*}
%Then $a_k < 1$.
\end{lemma}

\begin{proof}
\moh{By solving the equation \eqref{eq:quadratic} and choosing the largest root we get
\begin{equation} \label{eq:quadsol}
    a_k = \frac{(m - s) + \sqrt{(s - m)^2 +4mq }}{2 m}.
\end{equation}

Let us assume that $a_k \geqslant 1\; \forall k \geqslant 0$. We have $\mu < L,$ thus $q< 1$. Also, we have $\alpha \in [0,1),$ thus $s = 1+ 2\alpha + 2\alpha^2 \geqslant 1$, and $0 < m = 1- 2\alpha \leqslant 1$. Therefore, 
\[
    m a_k^2 + (s - m) a_k -q  \geqslant m + s - m - q = s - q \geqslant 1 - q > 0,
\]
which is a contradiction with \eqref{eq:quadratic}. Hence $a_k < 1\; \forall k \geqslant 0$. 

Now, let us assume that $\alpha \leqslant \frac{\sqrt{2} - 1}{18} \sqrt{\frac{\mu}{L}}$.  Since $\frac{\mu}{L} \leqslant 1,$ then $\alpha \leqslant \frac{\sqrt{2} - 1}{18},$ and
\begin{equation*}
    q = \frac{\mu (1 - \alpha)^3}{L (1 + \alpha)} \geqslant \frac{27}{128} \frac{\mu}{L} %\geqslant \frac{1}{9} \frac{\mu}{L}. 
    \quad \Longrightarrow \quad \sqrt{q} \geqslant  \frac{3\sqrt{3}}{8\sqrt{2}} \sqrt{\frac{\mu}{L}}. %\geqslant \frac{1}{3} \sqrt{\frac{\mu}{L}}.
\end{equation*}

Thus, $\alpha \leqslant \frac{\sqrt{2} - 1}{18} \sqrt{\frac{\mu}{L}} \leqslant \frac{\sqrt{2} - 1}{6} \sqrt{q}$, and 
\begin{equation*}
    \begin{gathered}
        m - s = -4\alpha - 2\alpha^2 \geqslant -6\alpha, \\
         m = 1 - 2\alpha \geqslant \frac{1}{2}.
    \end{gathered}
\end{equation*}

We have $m = 1- 2 \alpha \leqslant 1$. Thus, from \eqref{eq:quadsol}, we get  $a_k \geqslant \frac{m - s}{2} + \frac{1}{2}\sqrt{(s - m)^2 + 4mq }$.  Therefore, by using inequality $\sqrt{x^2 + y^2} \geqslant\frac{1}{2} (x + y)$, we obtain
\begin{align*}
    a_k \geqslant \frac{1}{2}\left( \frac{m - s}{2} + \sqrt{m} \sqrt{q} \right) \geqslant \frac{1}{2} \left( -3\alpha + \frac{1}{\sqrt 2} \sqrt{q} \right) \geqslant \frac{1}{2} \sqrt{q} \geqslant \frac{3\sqrt{3}}{16\sqrt{2}} \sqrt{\frac{\mu}{L}} \geqslant \frac{1}{10} \sqrt{ \frac{\mu}{L} }.
\end{align*}

Lastly, let $0 \leqslant \tau \leqslant \frac{1}{2}$ and $\alpha = \frac{1}{3} \left(\frac{\mu}{L}\right)^{\frac{1}{2} - \tau}$, then $0 < \alpha \leqslant \frac{1}{3}$. Thus,
\begin{align*}
    & s - m = 4 \alpha + 2\alpha^2 \qquad \Longrightarrow  \quad 4 \alpha \leqslant s - m < 6\alpha,
    \\& q = \frac{\mu}{\widehat{L}} = \frac{\mu}{L} \frac{(1 - \alpha)^3}{1 + \alpha} \quad \Longrightarrow  \quad \frac{2}{9} \frac{\mu}{L} \leqslant q \leqslant \frac{\mu}{L}.
    %\\& m =  1 - 2\alpha \leqslant 1.
\end{align*}

Let us prove by contradiction. To that end, we assume $a_k < \frac{1}{10} \left(\frac{\mu}{L}\right)^{\tau + \frac{1}{2}}$. Then:

\begin{align*}
m a_k^2 + (s - m) a_k - q 
&< \frac{m}{100} \left(\frac{\mu}{L}\right)^{2\tau + 1} + \frac{s - m}{10}\left(\frac{\mu}{L}\right)^{\tau + \frac{1}{2}} - q \\
&\leqslant \frac{m}{100} \left(\frac{\mu}{L}\right)^{2\tau + 1} + \frac{s - m}{10}\left(\frac{\mu}{L}\right)^{\tau + \frac{1}{2}} - \frac{2}{9} \frac{\mu}{L} \\ 
&\leqslant
\frac{1}{100} \left(\frac{\mu}{L}\right)^{2\tau + 1} + \frac{6}{30} \left(\frac{\mu}{L}\right)^{\frac{1}{2} -\tau}  \left(\frac{\mu}{L}\right)^{\tau + \frac{1}{2}} - \frac{2}{9} \frac{\mu}{L} \\
&= \frac{\mu}{L} \left( \frac{1}{100} \left(\frac{\mu}{L}\right)^{2\tau} + \frac{2}{10} - \frac{2}{9} \right) < 0
\end{align*}
Thus, we come to a contradiction because $a_k$ is a root of the equation $m a_k^2 + (s - m) a_k - q = 0$.

Therefore, from $\sqrt{1 + x} \geqslant 1 + \frac{x}{2} - \frac{x^2}{8} \; \forall x \geqslant 0$, we get the following 
\begin{align*}
    a_k &= \frac{s - m}{2m} \left(-1 + \sqrt{1 + \frac{4mq}{(s - m)^2}} \right) \geqslant \frac{s - m}{2m} \left( \frac{2mq}{(s - m)^2} - \frac{2m^2q^2}{(s - m)^4} \right) \\
    & \geqslant \frac{q}{s - m} \left(1 - \frac{q m}{(s - m)^2} \right) \geqslant
    \frac{2}{9} \frac{\mu}{L} \frac{1}{6\alpha} \left(1 - \frac{\mu}{L} \frac{1}{16 \alpha^2} \right) \\
    &\geqslant 
    \frac{2}{9} \frac{\mu}{L} \frac{3}{6} \left(\frac{\mu}{L}\right)^{\tau - \frac{1}{2}} \left(1 - \frac{9}{16} \frac{\mu}{L} \left(\frac{\mu}{L} \right)^{2\tau - 1} \right) = \frac{1}{9} \left(\frac{\mu}{L}\right)^{\tau + \frac{1}{2}} \left(1 - \frac{9}{16} \left(\frac{\mu}{L} \right)^{2\tau} \right) \\
    &\geqslant \frac{7}{48}  \left(\frac{\mu}{L}\right)^{\tau + \frac{1}{2}} \geqslant \frac{1}{10} \left(\frac{\mu}{L}\right)^{\tau + \frac{1}{2}} .
\end{align*}
}
\end{proof}

\begin{lemma} \label{gd step alpha}
\moh{Let $f$ be an $L$-smooth function, $\alpha \in [0,1)$ and $h = \left(\frac{1-\alpha}{1+\alpha}\right)^{3/2} \frac{1}{L}$. Then for 
 the gradient step
\begin{equation}\label{grad_step}
    y = x - h \widetilde{\nabla} f(x),
\end{equation}
it holds that
\begin{equation*}
    f(y) \leqslant f(x) - \frac{1}{2\widehat{L}} \|\nabla f(x) \|_2^2,
\end{equation*}
where $\widehat{L} = \frac{1 + \alpha}{(1 - \alpha)^3} L$. }
%Let $f$ function satisfies condition~\eqref{smoothness_cond}, and $\widetilde{\nabla}f$ satisfies~\eqref{eq:relative_error}. Then gradient step:
%\begin{equation*}
%    y = x - h \widetilde{\nabla} f(x),
%\end{equation*}
%with $h = \frac{\nu\gamma}{L \rho^2} = \left(\frac{1 - \alpha}{1 + \alpha} \right)^{\frac{3}{2}} \frac{1}{L}$ produce:
%\begin{equation*}
%    f(y) \leqslant f(x) - \frac{1}{2\widehat{L}} \|\nabla f(x) \|_2^2
%\end{equation*}
%Where: $\widehat{L} = \frac{1 + \alpha}{(1 - \alpha)^3} L$
\end{lemma}
\begin{proof}
\moh{
Since $f$ is $L$-smooth, and by \eqref{grad_step}, for any $x,y \in \mathbb{R}^d$ we can write 
\begin{align*}
    f(y) & \leqslant f(x) + \langle \nabla{f}(x), y - x \rangle + \frac{L}{2} \|y - x \|_2^2
    \\& \leqslant f(x) + \langle \nabla{f}(x), -h\widetilde{\nabla}f(x) \rangle + \frac{h^2L}{2} \| \widetilde{\nabla} f(x) \|_2^2. 
\end{align*}
Let $\rho = 1 + \alpha, \nu = 1 - \alpha$ and $\gamma = \sqrt{1-\alpha^2}$. Then, we find $h = \frac{\nu \gamma}{L \rho^2}$, and from~\cref{growth condition a} we get
\begin{align*}
    f(y) & \leqslant f(x) - h \nu \gamma \| \nabla{f}(x) \|_2^2 + \frac{h^2L \rho^2}{2} \| \nabla{f}(x) \|_2^2
    \\& = f(x) - \frac{\nu^2 \gamma^2}{2L\rho^2} \| \nabla{f}(x) \|_2^2 
    \\& = f(x) - \frac{(1-\alpha)^3}{2L(1+\alpha)}\| \nabla{f}(x) \|_2^2
    \\&= f(x) - \frac{1}{2\widehat{L}} \| \nabla{f}(x) \|_2^2.
\end{align*}
}
%\begin{equation*}
    %\begin{gathered}
        %f(y) \leqslant f(x) + \langle \nabla{f}(x), y - x \rangle + \frac{L}{2} \|y - x \|_2^2, \\
        %f(y) \leqslant f(x) + \langle \nabla{f}(x), -h\widetilde{\nabla}f(x) \rangle + \frac{h^2L}{2} \| \widetilde{\nabla} f(x) \|_2^2, \\
        %f(y) \leqslant f(x) - h \nu \gamma \| \nabla{f}(x) \|_2^2 + \frac{h^2L \rho^2}{2} \| \nabla{f}(x) \|_2^2, \\
        %f(y) \leqslant f(x) - \frac{\nu^2 \gamma^2}{2L\rho^2} \| \nabla{f}(x) \|_2^2, \\
        %f(y) \leqslant f(x) - \frac{1}{2\widehat{L}} \| \nabla{f}(x) \|_2^2.
    %\end{gathered}
%\end{equation*}
\end{proof}

%%%%%%%%%%%%% Lemma A.5 %%%%%%%%%%%%%%%%%%
\begin{lemma} \label{c > f lemma}
\moh{Let $f$ be an $L$-smooth and $\mu$-strongly convex function, $\{x^k\}_{k \geqslant 0}$ be a sequence of points generated by Algorithm \ref{alg:nesterov corrected}, and $c_k\, \forall k \geqslant 0$ be a minimal value of the function $\psi^k$ corresponding to Algorithm \ref{alg:nesterov corrected}. Then $c_k\geqslant f(x^k) \, \forall k \geqslant 0$. 
}
%Let $c^k$ - minimum value of $\psi^k$, corresponding to Algorithm~\ref{alg:nesterov corrected}, $\lbrace x^k \rbrace$ - generated sequence, then $c_k\geqslant f(x^k)$.
\end{lemma}
\begin{proof}
We will use the principle of mathematical induction to prove the statement of this lemma. For $k = 0$, it is obvious that $c_0 \geqslant f(x^0)$. Now, let us assume that $c_k\geqslant f(x^k)$, and prove the statement for $k+1$. For this, from Lemma \ref{psi comb}, \eqref{psi seq def} and definition of $\phi_{y^k}$ from Lemma~\ref{lower bound} we have 
\begin{equation}\label{eq:ck1}
    c_{k + 1} = (1 - a_k) c_k+ a_k \left(f(y^k) - \frac{\rho^2 + \alpha^2}{2\mu} \| \nabla f(y^k) \|_2^2 \right)+ \frac{a_k (1 - a_k) \mu}{2} \left \|y^k - u^k - \frac{1}{\mu} \widetilde{\nabla} f(y^k) \right \|_2^2.
\end{equation}

%\artem{Answer: third term obtained from Lemma~\ref{psi comb}, \eqref{psi seq def} and definition of $\phi_{y^k}$ from Lemma~\ref{lower bound}. $c_k\not= f(x^k)$. We define $\psi^k$ recursively at ~\eqref{psi seq def} it defines $c^{k}$ and $u^k$ using Lemma~\ref{psi comb}.}

By using the Fenchel inequality and \eqref{eq:relative_error}, we find 
\begin{align}
    \langle \widetilde{\nabla}f(y^k) - \nabla f(y^k), u^k - y^k \rangle & \geqslant -\frac{1}{2\mu} \|\widetilde{\nabla}f(y^k) - \nabla f(y^k)\|_2^2 - \frac{\mu}{2} \|u^k - y^k \|_2^2 \nonumber
    \\& \geqslant -\frac{\alpha^2}{2\mu} \|\nabla f(y^k)\|_2^2 - \frac{\mu}{2} \|u^k - y^k \|_2^2. \label{eq:412}
\end{align}

We have
\begin{align}
    & \qquad \;\;\; \mu \left\|y^k - u^k - \frac{1}{\mu} \widetilde{\nabla} f(y^k) \right\|_2^2 \nonumber
    \\& \;\;\;=\;\;\; \mu \|y^k - u^k 
    \|_2^2 + 2 \langle \widetilde{\nabla} f(y^k), u^k - y^k \rangle + \frac{1}{\mu} \|\widetilde{\nabla} f(y^k)\|_2^2 \nonumber
    \\& \;\;\;=\;\;\; \mu \|y^k - u^k 
    \|_2^2 + 2 \langle \nabla f(y^k), u^k - y^k \rangle +  2 \langle \widetilde{\nabla}f(y^k) - \nabla f(y^k), u^k - y^k \rangle + \frac{1}{\mu} \|\widetilde{\nabla} f(y^k)\|_2^2 \nonumber
    \\& \overset{\eqref{eq:412},\, \eqref{growth condition a}}{\geqslant} \mu \|y^k - u^k \|_2^2 + 2 \langle \nabla f(y^k), u^k - y^k \rangle - \frac{\alpha^2}{\mu} \|\nabla f(y^k)\|_2^2 - \mu \|u^k - y^k\|_2^2 + \frac{\nu^2}{\mu} \|\nabla f(y^k) \|_2^2 \nonumber
    \\& \;\;\;=\;\;\; 2 \langle \nabla f(y^k), u^k - y^k \rangle  + \frac{\nu^2 - \alpha^2}{\mu} \| \nabla f(y^k) \|_2^2. \label{eq:000}
\end{align}

By using the convexity of $f$, we get $f(x^k) \geqslant f(y^k) + \langle \nabla f(y^k), x^k - y^k \rangle$, and since $c_k\geqslant f(x^k)$, then we find
\begin{equation}\label{eq:1m2m}
    c_k\geqslant f(y^k) + \langle \nabla f(y^k), x^k - y^k \rangle. 
\end{equation}
We have $\rho = 1+ \alpha$, thus $s = 1 + 2 \alpha + 2 \alpha^2 = \rho^2 + \alpha^2$, and $\nu = 1 - \alpha$, thus $m = 1 - 2 \alpha = \nu^2 - \alpha^2$. Therefore, from \eqref{eq:ck1}, we find
\begin{align*}
    c_{k + 1} & \quad = \quad (1 - a_k) c_k+ a_k f(y^k)  - \frac{s a_k}{2\mu} \| \nabla f(y^k) \|_2^2 + \frac{a_k (1 - a_k) \mu}{2} \left \|y^k - u^k - \frac{1}{\mu} \widetilde{\nabla} f(y^k) \right \|_2^2
    \\& \overset{\eqref{eq:000},\, \eqref{eq:1m2m}}{\geqslant} (1-a_k) \left(f(y^k) + \langle \nabla f(y^k), x^k - y^k \rangle\right) + a_k f(y^k) - \frac{s a_k }{2\mu} \| \nabla f(y^k) \|_2^2
    \\
    & \qquad \quad  + \frac{a_k(1-a_k)}{2}\left(2 \langle \nabla f(y^k), u^k - y^k \rangle  + \frac{m}{\mu} \| \nabla f(y^k) \|_2^2\right)
    \\
    & \quad \geqslant \quad f(y^k) - \frac{s a_k}{2\mu} \|\nabla f(y^k) \|_2^2 +  \langle \nabla f(y^k), (1 - a_k) (x^k - y^k ) + a_k (1 - a_k) (u^k - y^k) 
    \rangle 
    \\
    & \qquad \quad +  \frac{a_k (1 - a_k) m}{2\mu} \|\nabla f(y^k) \|_2^2.
\end{align*}

But $y^k = \frac{a_k u^k + x^k}{1 + a_k}$ (see item 10 in Algorithm \ref{alg:nesterov corrected}). Thus,~$\left\langle \nabla f(y^k), (1 - a_k) \left(x^k - y^k \right) + a_k (1 - a_k) \left(u^k - y^k\right) \right\rangle = 0$. Therefore, we have
\begin{align*}
    c_{k + 1} & \geqslant f(y^k) - \frac{1}{2\mu} \left( s a_k - m a_k (1 - a_k) \right) \|\nabla f(y^k) \|_2^2 
    \\& =  f(y^k) - \frac{1}{2\mu} \left( m a_k^2 + (s-m) a_k \right) \|\nabla f(y^k) \|_2^2.
\end{align*}

But $a_k$ is a solution of the equation $ma_k^2 + (s-m)a_k - q = 0,$ with $q = \mu / \widehat{L}$ (see items 7 and 9 in Algorithm \ref{alg:nesterov corrected}). Thus, $ma_k^2 + (s-m)a_k = \frac{\mu}{\widehat L}$, and 
\begin{equation}\label{eq:hj00}
    c^{k+1} \geqslant f(y^k) - \frac{1}{2 \widehat{L}} \|\nabla f(y^k)\|_2^2.
\end{equation}
From items 2 and 10 in Algorithm \ref{alg:nesterov corrected}, we have
\begin{equation*}
    x^{k + 1} = y^k - h \widetilde{\nabla} f(y^k), \; h = \left(\frac{1 - \alpha}{1 + \alpha} \right)^{\frac{3}{2}} \frac{1}{L}.
\end{equation*}
Therefore, by using Lemma~\ref{gd step alpha} for $x^{k + 1}$ and $y^k$, from \eqref{eq:hj00}, we obtain
\begin{equation*}
    c_{k + 1} \geqslant  f(y^k) - \frac{1}{2 \widehat{L}}\|\nabla f(y^k) \|_2^2 \geqslant f(x^{k + 1}).
\end{equation*}
Which is the desired result, and finish the proving the induction step.

\end{proof}

\begin{lemma} \label{upper bound for lemma}
\moh{Let $\{\lambda_k\}_{k \geqslant 0}$ be a sequence defined in \eqref{function seq}, $\{x^k\}_{k \geqslant 0}$ be a sequence of points generated by Algorithm \ref{alg:nesterov corrected}, and $\{\psi^k\}_{k \geqslant 0}$ be the corresponding functions. Then 
\begin{equation*}
    \psi^k(x) \leqslant \lambda_k \psi_0(x) + (1 - \lambda_k) f(x), \quad \forall k \geqslant 0 \;\; \text{and} \;\; \forall x \in \mathbb{R}^d.
\end{equation*}
}
%Let $\lbrace \lambda_k \rbrace$ sequence defined at~\eqref{function seq}, $\lbrace x^k \rbrace$ - generated by Algorithm~\ref{alg:nesterov corrected}, $\lbrace \psi^k \rbrace$ - sequence of corresponding functions, then:
%\begin{equation*}
%    \psi^k(x) \leqslant \lambda_k \psi_0(x) + (1 - \lambda_k) f(x).
%\end{equation*}
\end{lemma}
\begin{proof}
\moh{ By the principle of induction. For $k = 0$, the inequality is obvious. Assume that inequality takes place for $k$. From \eqref{psi seq def}, and Lemma \ref{lower bound}, we get 
\begin{align*}
    \psi_{k + 1}(x)  &= (1 - a_k) \psi^k(x) + a_k \phi_{y^k}(x) 
    \\& \leqslant (1 - a_k) \left( (1 - \lambda_k) f(x) + \lambda_k \psi_0(x) \right) + a_k f(x) \\
    & = \lambda_{k + 1} \psi_0(x) + (1 - \lambda_{k + 1}) f(x).
\end{align*}
}
%We will prove by induction. Base $k = 0$ is obvious. Assume that inequality takes place for $k$, then:
%\begin{align*}
    %\psi_{k + 1}(x)  &= (1 - a_k) \psi^k(x) + a_k \phi_{y^k}(x) 
    %\\& \leqslant (1 - a_k) \left( (1 - \lambda_k) f(x) + \lambda_k \psi_0(x) \right) + a_k f(x) \\
    %& = \lambda_{k + 1} \psi_0(x) + (1 - \lambda_{k + 1}) f(x).
%\end{align*}
\end{proof}

\textbf{\cref{nesterov acc corrected convergence}}.
\textit{
Let Assumption~\ref{as:Lsmooth+muConv} hold.
% $L$-smooth~\eqref{smoothness_cond} and $\mu$ strongly convex~\eqref{as:str_cvx}, $\widetilde{\nabla}f$ satisfies~\eqref{eq:relative_error}, 
If Assumption~\ref{as:noise} holds with $\alpha < \frac{\sqrt{2} - 1}{18 \sqrt{2}} \sqrt{\frac{\mu}{L}}$, then Algorithm~\ref{alg:nesterov corrected} with parameters ($L, \frac{\mu}{2}, x^0, \alpha)$ generates $x^N$, s.t.
\begin{equation*}
    f(x^N) - f^* \leqslant L R^2 \left( 1 - \frac{1}{10\sqrt{2}} \sqrt{\frac{\mu}{L}} \right)^N,
\end{equation*}
where $R := \|x^0 - x^* \|_2$.
}
\begin{proof}[Proof of Theorem~\ref{nesterov acc corrected convergence}]
From Lemma~\ref{c > f lemma}, we have $\forall k \geqslant 0: f(x^k) \leqslant \min \psi^k$. By Lemma~\ref{upper bound for lemma} we can use Lemma~\ref{conv lemma}, then
\begin{align*}
    f(x^N) - f^* & \leqslant \lambda_N \left(\psi_0(x^*) - f^* \right) 
    \\& = \lambda_N \left(f(x^0) - f^* + \frac{\mu}{2} \|x^0 - x^*\|_2^2 \right)
    \\& \leqslant \lambda_N \left( \frac{L}{2} R^2 + \frac{\mu}{2} R^2 \right) 
    \\& \leqslant \lambda_N L R^2.
\end{align*}
From Lemma~\ref{a_k correctness} \vg{ we obtain $a_k \geqslant \frac{1}{10\sqrt{2}} \sqrt{\frac{\mu}{L}}$ (we use algorithm with $\mu = \frac{\mu}{2}$) and definition of $\lambda_k$ we get }
\begin{equation*}
    f(x^N) - f^* \leqslant \left(1 -  \frac{1}{10\sqrt{2}} \sqrt{\frac{\mu}{L}} \right)^N L R^2.
\end{equation*}

\end{proof}

\textbf{Theorem~\ref{nesterov acc corrected convergence tau}}.
\textit{
Let Assumption~\ref{as:Lsmooth+muConv} hold.
% $L$-smooth~\eqref{smoothness_cond} and $\mu$ strongly convex~\eqref{as:str_cvx}, $\widetilde{\nabla}f$ satisfies~\eqref{eq:relative_error}, 
If Assumption~\ref{as:noise} holds with $\alpha = \frac{1}{3}\left(\frac{\mu}{L}\right)^{\frac{1}{2} - \tau}$, where $0 \leqslant \tau \leqslant \frac{1}{2}$, then Algorithm~\ref{alg:nesterov corrected} with parameters ($L, \frac{\mu}{2}, x^0, \alpha)$ generates $x^N$, s.t. 
\begin{equation*}
    f(x^N) - f^* \leqslant L R^2 \left( 1 - \frac{1}{10\sqrt{2}} \left(\frac{\mu}{L}\right)^{\frac{1}{2} + \tau} \right)^N,
    \end{equation*}
where $R := \|x^0 - x^* \|_2$.
}

\begin{proof}[Proof of Theorem~\ref{nesterov acc corrected convergence tau}]
Same steps as in proof Theorem~\ref{nesterov acc corrected convergence}, but using 
Lemma~\ref{a_k correctness} with $a_k \geqslant \frac{1}{10\sqrt{2}} \left(\frac{\mu}{L}\right)^{\frac{1}{2} + \tau}$.
\end{proof}

\section[Missing Proofs for Section~\ref{section min max}]%
        {Missing Proofs for~\texorpdfstring{\cref{section min max}}{Section~\ref{section min max}}}
\label{appendix proofs min max}

% \subsection{Proof of~\cref{corollary:upper_bound_scsc_smooth}}
\subsection[Proof of Corollary~\ref{corollary:upper_bound_scsc_smooth}]%
            {Proof of~\texorpdfstring{\cref{corollary:upper_bound_scsc_smooth}}{Corollary~\ref{corollary:upper_bound_scsc_smooth}}}

In the following Lemma~\ref{lem:contraction_scsc_smooth} we consider contraction in terms of Lyapunov function $\frac{1}{\eta_x}\|x^{k} - x^{*}\|_2^2 + \frac{1}{\eta_y}\|y^{k} - y^{*}\|_2^2$ that was successfully used for analysis of Sim-GDA in exact gradient setting \cite{lee2024fundamentalbenefitalternatingupdates}. 

%contraction lemma for scsc smooth saddle
\begin{lemma} \label{lem:contraction_scsc_smooth}
Let Assumptions \ref{as:conv-conc-inexact} and \ref{as:long} hold. Then, the sequence $(x^k,y^k)$ generated by Inexact Sim-GDA Algorithm \ref{alg:sim} satisfies
% Consider $f \in S_{\mu_x\mu_yL_xL_yL_{xy}}$ (\cref{as:long}). For inexact Sim-GDA the following inequality holds
\begin{eqnarray*}
    \lefteqn{\frac{1}{\eta_x}\|x^{k+1} - x^{*}\|_2^2 + \frac{1}{\eta_y}\|y^{k+1} - y^{*}\|_2^2}
    \\
    & \leqslant & \left(\frac{1}{\eta_x} + \frac{\alpha}{\max \{\eta_x,\eta_y\}} - \mu_x + 2\max \{\eta_x,\eta_y\}[(1+\alpha)^2 + \alpha]L_{xy}^2 \right)\|x^{k} - x^{*}\|_2^2  
    \\
    && + \left(\frac{1}{\eta_y} + \frac{\alpha}{\max \{\eta_x,\eta_y\}} - \mu_y + 2\max \{\eta_x,\eta_y\}[(1+\alpha)^2 + \alpha]L_{xy}^2 \right)\|y^{k} - y^{*}\|_2^2,
\end{eqnarray*}
whenever    
\begin{align*}
    &\max \{\eta_x,\eta_y\} \leqslant \frac{1}{2\left(\alpha+\left(\alpha+1\right)^2\right)}\min\left\{\frac{1}{L_x},\hspace{4pt} \frac{1}{L_y}\right\}.
\end{align*}
\end{lemma}

\begin{proof}
We will follow the proof structure of similar theorem from \cite{lee2024fundamentalbenefitalternatingupdates} except here we consider inexact gradients.
\begin{align*}
    \lefteqn{\frac{1}{\eta_x} \|x^{k+1} - x^{*}\|_2^2 + \frac{1}{\eta_x} \|y^{k+1} - y^{*}\|_2^2}
    \\
    &= \frac{1}{\eta_x} \|x^{k} - x^{*}\|_2^2 + \frac{2}{\eta_x} \langle x^{k+1} - x^{k}, x^{k} - x^{*} \rangle + \frac{1}{\eta_x} \|x^{k+1} - x^{k}\|_2^2  \\ 
    & \quad + \frac{1}{\eta_y} \|y^{k} - y^{*}\|_2^2 + \frac{2}{\eta_y} \langle y^{k+1} - y^{k}, y^{k} - y^{*} \rangle + \frac{1}{\eta_y} \|y^{k+1} - y^{k}\|_2^2  \\
    &= \frac{1}{\eta_x} \|x^{k} - x^{*}\|_2^2 - 2\langle \widetilde{\nabla}_x f(x^k,y^k), x^k-x^{*} \rangle + \eta_x \|\widetilde{\nabla}_x f(x^k,y^k)\|_2^2 \\
    & \quad + \frac{1}{\eta_y} \|y^{k} - y^{*}\|_2^2 + 2 \langle \widetilde{\nabla}_y f(x^k,y^k),y^k-y^{*}\rangle + \eta_y \|\widetilde{\nabla}_y f(x^k,y^k)\|_2^2.
\end{align*}
From the relative inexactness of the gradient, it follows that
\begin{multline*} 
    \eta_x \|\widetilde{\nabla}_x f(x^k,y^k)\|_2^2 + \eta_y \|\widetilde{\nabla}_y f(x^k,y^k)\|_2^2 \leqslant \max \{\eta_x,\eta_y\}\|\widetilde{g}^z(x^k,y^k)\|_2^2 
    \\
    \leqslant \max \{\eta_x,\eta_y\}(1+\alpha)^2\|g^z(x^k,y^k)\|_2^2,
\end{multline*}
and
\begin{align*}
    \lefteqn{-2 \langle \widetilde{\nabla}_x f(x^k,y^k),x^k-x^{*}\rangle + 2\langle \widetilde{\nabla}_y f(x^k,y^k),y^k-y^{*}\rangle }
    \\
    &= -2\langle \nabla_x f(x^k,y^k), x^k-x^{*}\rangle + 2\langle \nabla_y f(x^k,y^k),y^k-y^{*}\rangle 
    \\& \quad -2\langle \widetilde{g}^z(x^k,y^k) - g^z(x^k,y^k),z^k-z^*\rangle
    \\ &\leqslant -2\langle \nabla_x f(x^k,y^k),x^k-x^{*}\rangle + 2\langle \nabla_y f(x^k,y^k), y^k-y^{*}\rangle 
    \\& \quad + \alpha \max \{\eta_x,\eta_y\}\|g^z(x^k,y^k)\|_2^2 + \frac{\alpha}{\max \{\eta_x,\eta_y\}}\|z^k-z^0\|_2^2. 
\end{align*}
Substituting into the main inequality we get
\begin{align*}
    \lefteqn{\frac{1}{\eta_x} \|x^{k+1} - x^{*}\|_2^2 + \frac{1}{\eta_x} \|y^{k+1} - y^{*}\|_2^2 }
    \\
    & \leqslant \left(\frac{1}{\eta_x} + \frac{\alpha}{\max \{\eta_x,\eta_y\}}\right)\|x^{k} - x^{*}\|_2^2 + \left(\frac{1}{\eta_y} + \frac{\alpha}{\max \{\eta_x,\eta_y\}}\right)\|y^{k} - y^{*}\|_2^2  
    \\ & \quad -2\langle\nabla_x f(x^k,y^k),x^k-x^{*}\rangle + 2\langle \nabla_y f(x^k,y^k), y^k-y^{*}\rangle
    \\& \quad + \max \{\eta_x,\eta_y\} \left(( 1 + \alpha)^2  + \alpha\right)\|g^z(x^k,y^k)\|_2^2. 
\end{align*}
We now use properties of $S_{\mu_x\mu_yL_xL_yL_{xy}}$. First, $f(\cdot,y)$ is $\mu_x$-strongly-convex with $L_x$-Lipschitz gradient:
\begin{align*}
    &-2\langle \nabla_x f(x^k,y^k),x^k-x^{*} \rangle \leqslant -\mu_x \|x^k-x^{*}\|_2^2 - 2\left(f(x^k,y^k) - f(x^*,y^k)\right), \\
    &-2\left(f(x^k,y^*) - f(x^*,y^*)\right) \leqslant -\frac{\|\nabla_x f(x^k,y^*)\|_2^2}{L_x}.
\end{align*}
Second, $f(x,\cdot)$ is $\mu_y$-strongly-concave with $L_y$-Lipschitz gradient:
\begin{align*}
    &2\langle \nabla_y f(x^k,y^k), y^k-y^{*} \rangle \leqslant -\mu_y \|y^k-y^{*}\|_2^2 - 2\left(f(x^k,y^*) - f(x^k,y^k)\right), \\
    &-2\left(f(x^*,y^*) - f(x^*,y^k)\right) \leqslant -\frac{\|\nabla_y f(x^*,y^k)\|_2^2}{L_y}.
\end{align*}
By applying the inequalities we get
\begin{align*}
    & \quad\;\; \frac{1}{\eta_x} \|x^{k+1} - x^{*}\|_2^2 + \frac{1}{\eta_x} \|y^{k+1} - y^{*}\|_2^2 
    \\& \leqslant \left(\frac{1}{\eta_x} + \frac{\alpha}{\max \{\eta_x,\eta_y\}} -\mu_x \right)\|x^{k} - x^{*}\|_2^2 + \left(\frac{1}{\eta_y} + \frac{\alpha}{\max \{\eta_x,\eta_y\}} - \mu_y\right)\|y^{k} - y^{*}\|_2^2 
    \\& \quad + \max \{\eta_x,\eta_y\}\left((1+\alpha)^2 + \alpha\right)\|g^z(x^k,y^k)\|_2^2 - \frac{\|\nabla_x f(x^k,y^*)\|_2^2}{L_x} -\frac{\|\nabla_y f(x^*,y^k)\|_2^2}{L_y}. 
\end{align*}
Now, suppose 
\begin{align*}
    &\max \{\eta_x,\eta_y\} \leqslant \frac{1}{2\left(\alpha+\left(\alpha+1\right)^2\right)}\min\left\{\frac{1}{L_x},\hspace{4pt} \frac{1}{L_y}\right\}. 
\end{align*}
Then we can now use $L_{xy}$-Lipschitz condition for gradients
\begin{align*}
    &\quad \max \{\eta_x,\eta_y\}\left((1+\alpha)^2 + \alpha\right)\|\nabla_x f(x^k,y^k)\|_2^2 -\frac{1}{L_x}\|\nabla_x f(x^k,y^*)\|_2^2 
    \\&\leqslant \max \{\eta_x,\eta_y\}\left((1+\alpha)^2 + \alpha\right)\left(\|\nabla_x f(x^k,y^k)\|_2^2 - 2\|\nabla_x f(x^k,y^*)\|_2^2\right) 
    \\& \leqslant 2\max \{\eta_x,\eta_y\}\left((1+\alpha)^2 + \alpha\right)L_{xy}^2\|y^k-y^*\|_2^2, 
\end{align*}
and \begin{align*}
    & \quad \max \{\eta_x,\eta_y\}\left((1+\alpha)^2 + \alpha\right)\|\nabla_y f(x^k,y^k)\|_2^2 -\frac{1}{L_y}\|\nabla_y f(x^*,y^k)\|_2^2 
    \\&\leqslant \max \{\eta_x,\eta_y\}\left((1+\alpha)^2 + \alpha\right)\left(\|\nabla_y f(x^k,y^k)\|_2^2 - 2\|\nabla_y f(x^*,y^k)\|_2^2\right) 
    \\&\leqslant 2\max \{\eta_x,\eta_y\}\left((1+\alpha)^2 + \alpha\right)L_{xy}^2\|x^k-x^*\|_2^2.
\end{align*}
By using this last pair of inequalities we finally prove the statement of the Lemma.
\end{proof}

We are now in a position to find conditions that are sufficient for linear convergence of Inexact Sim-GDA. To simplify the analysis we let $L_x = L_y$, $\mu_x = \mu_y$, $\eta_x = \eta_y$ which does not impact the result's value that much since we still leave $L_{xy}$ to be arbitrary.

% corollary for iteration complexity
\textbf{\cref{corollary:upper_bound_scsc_smooth}}
\textit{
Let \cref{as:long} hold with $\mu_x=\mu_y=\mu$ and $L_x=L_y=L$. If Assumption \ref{as:conv-conc-inexact} holds with $\alpha$ satisfying 
\begin{eqnarray*}
    \text{case} \quad  L_{xy}^2 \leqslant \frac{\mu L}{2}:&& \quad  \alpha \left((1+\alpha)^2 + \alpha\right) < \frac{\mu}{2L} - \frac{L_{xy}^2}{2L^2}, \\
    \text{case} \quad  L_{xy}^2 > \frac{\mu L}{2}:&& \quad     \alpha \left((1+\alpha)^2 + \alpha\right) < \frac{\mu^2}{8L_{xy}^2},
\end{eqnarray*}
then Algorithm \ref{alg:sim} with an appropriate choice of $\eta_x=\eta_y=\eta$ has linear convergence.
% Consider $f \in S_{\mu\mu L LL_{xy}}$ (\cref{as:long}). Inexact Sim-GDA retains linear convergence if 
% \begin{align*}
%     &\text{case} \quad  L_{xy}^2 \leqslant \frac{\mu L}{2}: \quad  \alpha \left((1+\alpha)^2 + \alpha\right) < \frac{\mu}{2L} - \frac{L_{xy}^2}{2L^2}, \\
%     &\text{case} \quad  L_{xy}^2 > \frac{\mu L}{2}: \quad     \alpha \left((1+\alpha)^2 + \alpha\right) < \frac{\mu^2}{8L_{xy}^2}.
% \end{align*}
}

\begin{proof}[Proof of~\cref{corollary:upper_bound_scsc_smooth}]
First, we write the contraction inequality from Lemma~\ref{lem:contraction_scsc_smooth}
\begin{align*}
    \lefteqn{\|z^{k+1} - z^{*}\|_2^2}
    \\
    & \leqslant 
    \left(1 + \alpha - \mu \eta + 2\left((1+\alpha)^2 + \alpha\right)L_{xy}^2 \eta^2 \right)\|z^{k} - z^{*}\|_2^2 
    \\&= \left(1+\alpha - \frac{\mu^2}{8\left(\alpha+(1+\alpha)^2\right)L_{xy}^2} + 2\left((1+\alpha)^2 + \alpha\right)L_{xy}^2 \left(\eta - \frac{\mu}{4\left((1+\alpha)^2 + \alpha\right)L_{xy}^2}\right)^2\right)
    \\&\quad \times \|z^{k} - z^{*}\|_2^2. 
\end{align*}

If $\mu L \leqslant 2L_{xy}^2$ then optimal step is
\begin{align*}
    \eta = \frac{\mu}{4\left((1+\alpha)^2 + \alpha\right)L_{xy}^2},
\end{align*}
and the contraction factor is less than $1$ if
\begin{align*}
    \alpha \left((1+\alpha)^2 + \alpha\right) < \frac{\mu^2}{8L_{xy}^2}.
\end{align*}

For the case $\mu L > 2L_{xy}^2$, the optimal step is 
\begin{align*}
    \eta = \frac{1}{2\left(\alpha+(1+\alpha)^2 \right)L}, 
\end{align*}
which follows from step restriction in Lemma~\ref{lem:contraction_scsc_smooth}. The Linear convergence holds if 
\begin{align*}
    \alpha \left((1+\alpha)^2 + \alpha\right) < \frac{\mu}{2L} - \frac{L_{xy}^2}{2L^2}.
\end{align*}
\end{proof}

% \subsection{Proofs of \cref{corollary:upper_bound_stongly_monotone_Lipshitz} and \cref{teo:sim_gda_lower_bound_alpha}}
\subsection[Proofs of Corollary~\ref{corollary:upper_bound_stongly_monotone_Lipshitz} 
and Theorem~\ref{teo:sim_gda_lower_bound_alpha}]%
{Proofs of \texorpdfstring{\cref{corollary:upper_bound_stongly_monotone_Lipshitz} 
and \cref{teo:sim_gda_lower_bound_alpha}}{Corollary~\ref{corollary:upper_bound_stongly_monotone_Lipshitz} 
and Theorem~\ref{teo:sim_gda_lower_bound_alpha}}}

%contraction lemma for mu-strongly-monotone L-Lipshitz gradient (Sim-GDA)
\begin{lemma} \label{lem:contraction_stongly_monotone_Lipshitz}
Let Assumptions \ref{as:conv-conc} and \ref{as:conv-conc-inexact}  hold. Then, the sequence $(x^k,y^k)$ generated by Inexact Sim-GDA Algorithm \ref{alg:sim} with $\eta_x=\eta_y=\eta$  satisfies
% Let $f \in S_{\mu, L}$ (\cref{as:conv-conc}) For Inexact Sim-GDA with relative gradient inexactness, the following inequality holds:
\begin{align*}
    \|z^{k+1}-z^{*}\|_2^2 \leqslant \left(1 - 2(\mu-\alpha L) \eta + (1+\alpha)^2L^2\eta^2 \right)\|z^{k}-z^{*}\|_2^2.
\end{align*}
% where $\eta_x=\eta_y=\eta$ is the step size.
\end{lemma}
\begin{proof}
We have
\begin{align*}
    \|z^{k+1}-z^{*}\|_2^2 &= \|z^{k}-z^{*}\|_2^2 - 2\eta \langle \widetilde{g}^z(z^k) - g^z(z^k) + g^z(z^k), z^{k}-z^{*} \rangle + \eta^2\|\widetilde{g}^z(z^k)\|_2^2 
    \\&= \|z^{k}-z^{*}\|_2^2 - 2\eta \langle \widetilde{g}^z(z^k) - g^z(z^k), z^{k}-z^{*} \rangle
    \\&  \quad - 2\eta \langle g^z(z^k) - g^z(z^*), z^{k}-z^{*}\rangle + \eta^2\|\widetilde{g}^z(z^k)\|_2^2.
\end{align*}
By Assumption  \ref{as:conv-conc-inexact} we get
\begin{align*}
    -2\eta\langle \widetilde{g}^z(z^k) - g^z(z^k), z^{k}-z^{*} \rangle & \leqslant 2\eta \|\widetilde{g}^z(z^k) - g^z(z^k)\| \|z^{k}-z^{*}\| 
    \\& \leqslant 2\eta\alpha \|g^z(z^k) - g^z(z^k)\|\|z^{k}-z^{*}\| 
    \\& \leqslant 2\eta\alpha L\|z^{k}-z^{*}\|_2^2,
\end{align*}
and 
\[
    \|\widetilde{g}^z(z^k)\|_2^2 \leqslant (1+\alpha)^2L^2\eta^2 \|z^{k}-z^{*}\|_2^2.
\]
By substituting the last two inequalities and using the properties of $f$ we prove the statement of the lemma.
\end{proof}

%iteration complexity for mu-strongly-monotone L-Lipshitz gradient (Sim-GDA)

\textbf{\cref{corollary:upper_bound_stongly_monotone_Lipshitz}}
\textit{ 
Let Assumption \ref{as:conv-conc} hold. If Assumption \ref{as:conv-conc-inexact} holds with $\alpha<\frac{\mu}{L}$, then Algorithm \ref{alg:sim} with parameters $(z^0, \eta,\eta)$, where $\eta=\frac{\mu -\alpha L}{(1+\alpha)^2L^2}$ generates $z^N$ s.t. $\|z^N-z^*\|_2\leq \varepsilon$ in 
\begin{align*}
     N=\mathcal{O} \rbr*{\frac{L^2}{\mu^2}\frac{1}{(1 - \alpha L/\mu)^2}\ln\frac{\|z^0-z^*\|_2}{\varepsilon}} \;\; \text{iterations.}
\end{align*} 
% For $f \in S_{\mu, L}$ (\cref{as:conv-conc}), Inexact Sim-GDA linearly converges with iteration complexity
% \begin{align*}
%     \mathcal{O} \rbr*{\frac{L^2}{\mu^2}\frac{1}{(1 - \alpha L/\mu)^2}}
% \end{align*}
% where we omit logarithmic factors.
}

\begin{proof}[Proof of \cref{corollary:upper_bound_stongly_monotone_Lipshitz}]
After rewriting contraction factor from Lemma~\ref{lem:contraction_stongly_monotone_Lipshitz}:
\begin{align*}
    1 - 2(\mu-\alpha L) \eta + (1+\alpha)^2L^2\eta^2 = 1 - \frac{(\mu-\alpha L)^2}{L^2(1+\alpha)^2} + \left(\eta(1+\alpha)L - \frac{\mu-\alpha L}{L(1+\alpha)} \right)^2.
\end{align*}
One can see that a step size choice of 
\begin{align*}
    \eta =  \frac{\mu -\alpha L}{(1+\alpha)^2L^2}
\end{align*}
ensures the stated iteration complexity.
\end{proof}

%lower bound for Sim-GDA
\textbf{\cref{teo:sim_gda_lower_bound_alpha}}
\textit{
There exists a function $f$ satisfying \cref{as:conv-conc} with $d_x = d_y = 1$ such that Inexact Sim-GDA with any constant step size $\eta$ loses linear convergence when $\alpha \geqslant \frac{\mu}{L}$.
%There exists such function $f \in S_{\mu, L}$ (\cref{as:conv-conc}) with $d_x = d_y = 1$ such that Inexact Sim-GDA with any constant step size $\eta$ loses linear convergence when $\alpha \geqslant \frac{\mu}{L}$.
}
\begin{proof}[Proof of~\cref{teo:sim_gda_lower_bound_alpha}]
Consider function 
\begin{align*}
    F(x,y) = \frac{\epsilon}{2}x^2 + xy - \frac{\epsilon}{2}y^2,
\end{align*}
where $\epsilon > 0$. We make one step of exact Sim-GDA:
\begin{align*}
    &x^{k+1} = x^k - \eta \left(y^k + \epsilon x^k \right)\\
    &y^{k+1} = y^k + \eta \left(x^k-\epsilon y^k \right),
\end{align*}
and introduce operator $g(x^k,y^k) = \left[ (y^k + \epsilon x^k)^\top, (-x^k +\epsilon y^k)^\top \right]^\top$:
\begin{align*}
    z^{k+1} = z^{k} - \eta g(x^k,y^k).
\end{align*}
The operator $g$ is $\mu$-strongly-monotone and $L$-Lipshitz with $\mu = \epsilon$ and $L = \sqrt{1+\epsilon^2}$. Let $\widetilde{g}(x^k, y^k)$ be a disturbed value of $g$ in $(x^k, y^k)$. Then
\begin{align*}
    \|g(x^k,y^k) - \widetilde{g}(x^k, y^k)\| = \epsilon \sqrt{(x^k)^2 + (y^k)^2} \leqslant \frac{\epsilon}{\sqrt{1+\epsilon^2}} \|g(x^k,y^k) \| = \frac{\mu}{L} \|g(x^k,y^k) \|,
\end{align*}
which means that $g$ satisfies relative inexactness definition with $\alpha = \frac{\mu}{L}$. A step of inexact Sim-GDA with disturbed $\widetilde{g}$ leads to
\begin{align*}
    &x^{k+1} = x^k - \eta \hspace{1pt}y^k, \\
    &y^{k+1} = y^k + \eta \hspace{1pt}x^k.
\end{align*}
For any $\eta > 0$:
\begin{align*}
    (x^{k+1})^2 + (y^{k+1})^2 = \left (1 + \eta^2 \right) \left((x^{k})^2 + (y^{k})^2 \right) > (x^{k})^2 + (y^{k})^2,
\end{align*}
and we have divergence.
\end{proof}

% \valery{

%contraction lemma for mu-strongly-monotone L-Lipshitz gradient (Alt-GDA)

\begin{lemma} \label{lem:contraction_stongly_monotone_Lipshitz_ALT}
Let Assumptions \ref{as:conv-conc} and \ref{as:conv-conc-inexact}  hold. Then, the sequence $(x^k,y^k)$ generated by Inexact Alt-GDA Algorithm \ref{alg:alt} with $\eta_x=\eta_y=\eta$  satisfies
% Let $f \in S_{\mu, L}$ (\cref{as:conv-conc}) For Inexact Sim-GDA with relative gradient inexactness, the following inequality holds:
\begin{align*}
    \|z^{k+1}-z^{*}\|_2^2 \leqslant \left(1 - 2(\mu-\sqrt{2}\alpha L) \eta + 4(1+\alpha)^2L^2\eta^2 + (1+\alpha)^4L^4\eta^4\right)\|z^{k}-z^{*}\|_2^2.
\end{align*}
% where $\eta_x=\eta_y=\eta$ is the step size.
\end{lemma}
\begin{proof}
We have
\begin{align*}
    \|z^{k+1}-z^{*}\|_2^2 &= \|z^{k}-z^{*}\|_2^2 - 2\eta \langle \widetilde{g}^x(x^k,y^k), x^{k}-x^{*} \rangle + \eta^2\|\widetilde{g}^x(x^k,y^k)\|_2^2 
    \\& \quad + 2\eta \langle \widetilde{g}^y(x^{k+1},y^k), y^{k}-y^{*} \rangle + \eta^2\|\widetilde{g}^y(x^{k+1},y^k)\|_2^2
\end{align*}
By Assumption  \ref{as:conv-conc-inexact} we get
\begin{align*}
    -\langle \widetilde{g}^x(x^k,y^k), x^{k}-x^{*} \rangle &= -\langle \widetilde{g}^x(x^k,y^k)-g^x(x^k,y^k), x^{k}-x^{*} \rangle - \langle g^x(x^k,y^k), x^{k}-x^{*} \rangle 
    \\& \leqslant \alpha \|g^z(x^k,y^k)\|\|x^{k}-x^{*}\| - \langle g^x(x^k,y^k), x^{k}-x^{*} \rangle
\end{align*}

\begin{align*}
    \langle \widetilde{g}^y(x^{k+1},y^k), y^{k}-y^{*} \rangle &\leqslant  \alpha \|g^z(x^{k+1},y^k)\|\|y^{k}-y^{*}\| + \langle g^y(x^{k+1},y^k), y^{k}-y^{*} \rangle =
    \\& = \alpha \|g^z(x^{k},y^k)\|\|y^{k}-y^{*}\| 
    \\& \quad + \alpha \|y^{k}-y^{*}\| \left(\|g^z(x^{k+1},y^k)\| - \|g^z(x^k,y^k)\| \right) + 
    \\& + \langle g^y(x^{k},y^k), y^{k}-y^{*} \rangle + \langle g^y(x^{k+1},y^k) - g^y(x^{k},y^k), y^{k}-y^{*} \rangle
\end{align*}

\begin{align*}
   \lefteqn{\alpha \|y^{k}-y^{*}\| \left(\|g^z(x^{k+1},y^k)\| - \|g^z(x^k,y^k)\| \right) }
   \\ &\leqslant \alpha \|y^{k}-y^{*}\| \|g^z(x^{k+1},y^k) - g^z(x^k,y^k)\| \leqslant \alpha L \|y^{k}-y^{*}\| \|x^{k+1}-x^{k}\|
   \\& = \alpha \eta L \|y^{k}-y^{*}\| \|\widetilde{g}^x(x^{k},y^k)\| \leqslant \alpha(1-\alpha) \eta L \|y^{k}-y^{*}\| \|g^z(x^{k},y^k)\| \leqslant
   \\& \leqslant \eta \alpha(1-\alpha)  L^2 \|z^{k}-z^{*}\|^2
\end{align*}

\begin{align*}
   \langle g^y(x^{k+1},y^k) - g^y(x^{k},y^k), y^{k}-y^{*} \rangle &\leqslant \|y^{k}-y^{*}\| \|g^y(x^{k+1},y^k) - g^y(x^k,y^k)\| 
   \\& \leqslant L \|y^{k}-y^{*}\| \|x^{k+1}-x^{k}\| 
   \\& = L \eta \|y^{k}-y^{*}\|  \|\widetilde{g}^x(x^{k},y^k)\| \\& \leqslant L \eta (1+\alpha) \|y^{k}-y^{*}\|  \|g^z(x^{k},y^k)\| 
   \\& \leqslant \eta (1+\alpha) L^2 \|z^{k}-z^{*}\|^2
\end{align*}

By combining four inequalities above we get

\begin{align*}
   & 2\eta \left(-\langle \widetilde{g}^x(x^k,y^k), x^{k}-x^{*} \rangle + \langle \widetilde{g}^y(x^{k+1},y^k), y^{k}-y^{*} \rangle \right) 
   \\& \leqslant 2\eta \left(\alpha \|g^z(x^{k},y^k)\|(\|x^{k}-x^{*}\| + \|y^{k}-y^{*}\|) + \eta (1+\alpha)^2 L^2 \|z^{k}-z^{*}\|^2 \right) 
   \\& \quad - 2\eta \langle g^x(x^k,y^k), x^{k}-x^{*} \rangle 
   \\& \leqslant 2\sqrt{2}\eta \alpha L \|z^{k}-z^{*}\|^2 + 2 \eta^2 L^2 (1+\alpha)^2\|z^{k}-z^{*}\|^2 - 2\eta \mu \|z^{k}-z^{*}\|^2
\end{align*}

Now we deal with squared gradients part:

\begin{align*}
   &\eta^2\|\widetilde{g}^x(x^k,y^k)\|_2^2 + \eta^2\|\widetilde{g}^y(x^{k+1},y^k)\|_2^2 \leqslant \eta^2 (1+\alpha)^2 \left(\|g^z(x^k,y^k)\|_2^2 + \|g^z(x^{k+1},y^k)\|_2^2\right) \leqslant
   \\&\leqslant \eta^2 L^2(1+\alpha)^2 \left( \|x^{k}-x^{*}\|^2 + \|x^{k+1}-x^{*}\|^2 + 2\|y^{k}-y^{*}\|^2 \right) = 
   \\& = \eta^2 L^2(1+\alpha)^2 \left (2\|z^{k}-z^{*}\|^2 + \eta^2 L^2 (1+\alpha)^2 \|z^{k}-z^{*}\|^2\right)
\end{align*}

By substituting the last two inequalities we prove the statement of the lemma.

\end{proof}

%iteration complexity for mu-strongly-monotone L-Lipshitz gradient (Alt-GDA)

\textbf{\cref{corollary:upper_bound_stongly_monotone_Lipshitz_ALT}}
\textit{ 
Let Assumption \ref{as:conv-conc} hold. If Assumption \ref{as:conv-conc-inexact} holds with $\alpha<\frac{\mu}{\sqrt{2}L}$, then Algorithm \ref{alg:alt} with parameters $(z^0, \eta,\eta)$, where $\eta= \mathcal{O} \rbr*{\frac{\mu}{L^2}}$ generates $z^N$ s.t. $\|z^N-z^*\|_2\leq \varepsilon$ in 
\begin{align*}
     N=\mathcal{O} \rbr*{\frac{L^2}{\mu^2}\frac{1}{(1 - \alpha \sqrt{2}L/\mu)^2}\ln\frac{\|z^0-z^*\|_2}{\varepsilon}} \;\; \text{iterations.}
\end{align*} 
% For $f \in S_{\mu, L}$ (\cref{as:conv-conc}), Inexact Sim-GDA linearly converges with iteration complexity
% \begin{align*}
%     \mathcal{O} \rbr*{\frac{L^2}{\mu^2}\frac{1}{(1 - \alpha L/\mu)^2}}
% \end{align*}
% where we omit logarithmic factors.
}

\begin{proof}[Proof of \cref{corollary:upper_bound_stongly_monotone_Lipshitz_ALT}]
To prove the statement we find the smallest possible value of contraction factor from Lemma~\ref{lem:contraction_stongly_monotone_Lipshitz_ALT}:
\begin{align*}
    \rho^2(\eta) = 1 - 2(\mu-\sqrt{2}\alpha L) \eta + 4(1+\alpha)^2L^2\eta^2 + (1+\alpha)^4L^4\eta^4
\end{align*}
After taking derivative we find the only real root with Cardano's formula:
\begin{align*}
    \eta = \frac{s_1 + s_2}{L(1+\alpha)}
\end{align*}
where we let
\begin{align*}
    &s_{1,2} = \left(s \pm\sqrt{\frac{8}{27}+s^2} \right)^{1/3} \\
    &s = \frac{\mu/\sqrt{2}L -\alpha}{2\sqrt{2}(1+\alpha)}
\end{align*}
After substituting optimal step size into contraction inequality we get
\begin{align*}
    \rho_{opt}^2 = 1 - 2 \left(s_1 + s_2\right)\frac{\mu-\sqrt{2}\alpha L}{L(1+\alpha)} +4(s_1+s_2)^2 + (s_1+s_2)^4
\end{align*}
Given that with big enough condition numbers
\begin{align*}
    s_1 + s_2 = \frac{2}{3}s + o(s) = \frac{\mu -\sqrt{2}\alpha L}{6L(1+\alpha)} + o\left(\frac{1}{L}\right)
\end{align*}
we prove the stated iteration complexity.
\end{proof}

%lower bound for Alt-GDA
\textbf{\cref{teo:alt_gda_lower_bound_alpha}}
\textit{
There exists a function $f$ satisfying \cref{as:conv-conc} with $d_x = d_y = 1$ such that Inexact Alt-GDA with any constant step size $\eta$ loses linear convergence when $\alpha \geqslant \frac{\mu}{L}$.
%There exists such function $f \in S_{\mu, L}$ (\cref{as:conv-conc}) with $d_x = d_y = 1$ such that Inexact Sim-GDA with any constant step size $\eta$ loses linear convergence when $\alpha \geqslant \frac{\mu}{L}$.
}
\begin{proof}[Proof of~\cref{teo:alt_gda_lower_bound_alpha}]
The idea is the same as in the proof for Sim-GDA: we choose a disturbed operator in such a way that the inexact method would not converge if $\alpha$ exceeds a certain value. Consider function 
\begin{align*}
    F(x,y) = \frac{\epsilon}{2}x^2 + xy - \frac{\epsilon}{2}y^2,
\end{align*}
and corresponding gradient operator $g(x^k,y^k) = \left[ (y^k + \epsilon x^k)^\top, (-x^k +\epsilon y^k)^\top \right]^\top$. Because $\alpha \geqslant \mu/L$, we can choose a disturbed operator $\widetilde{g}(x^k,y^k) = \left[ y^k{}^\top, -x^k{}^\top \right]^\top$. This disturbed operator is the same as an undisturbed gradient operator of the bilinear function $xy$ for which the exact Alt-GDA is known not to converge \cite{lee2024fundamentalbenefitalternatingupdates}. This concludes the proof.

\end{proof}

% }

% \section{Missing Proofs for~\cref{sec:ne}}
% \label{appendix ne}

\section[Missing Proofs for Section~\ref{sec:ne}]%
        {Missing Proofs for~\texorpdfstring{\cref{sec:ne}}{Section~\ref{sec:ne}}}
\label{appendix ne}

\begin{lemma} \label{lem:eg_contraction_stongly_monotone_Lipshitz}
Let Assumptions \ref{as:lipqsm} and \ref{as:ne-inexact} hold. Then, the sequence $z^k$ generated by Inexact ExtraGradient Method Algorithm \ref{alg:eg}, satisfies the following inequality
\begin{align*}
    &\|z^{k+1}-z^{*}\|_2^2  \leqslant \|z^{k}-z^{*}\|_2^2 - \eta\mu\|z^{k+1/2}-z^{*} \|_2^2 + \\& \quad + \left[\left(\frac{\eta\alpha^2}{\mu} +3\eta^2\alpha^2\right) \left(2L^2+\frac{2}{\eta^2(1-\alpha)^2}\right)+3\eta^2L^2 + \frac{3\alpha^2}{(1-\alpha)^2} - 1 \right] \|z^{k+\frac{1}{2}}-z^{k}\|_2^2.
\end{align*}
\end{lemma}
\begin{proof}
The first part of the chain is straightforward:
\begin{align*}
    & \|z^{k+1}-z^*\|_2^2 = \|z^k - z^*\|_2^2 + 2 \langle z^{k+1} - z^k, z^{k+\frac{1}{2}} - z^* \rangle + \|z^{k+1} - z^{k+\frac{1}{2}}\|_2^2 - \|z^k - z^{k + \frac{1}{2}} \|_2^2  \\
    &= \|z^k - z^* \|_2^2 - 2 \eta \langle \widetilde{g}(z^{k+\frac{1}{2}}), z^{k+\frac{1}{2}} - z^*\rangle + \eta^2 \|\widetilde{g}(z^k) - \widetilde{g}(z^{k+\frac{1}{2}}) \|_2^2 - \|z^k - z^{k+\frac{1}{2}}\|_2^2  \\
    &= \|z^k-z^*\|_2^2 - 2 \eta \langle g(z^{k+\frac{1}{2}}) + \widetilde{g}(z^{k+\frac{1}{2}}) - g(z^{k+\frac{1}{2}}), z^{k+\frac{1}{2}} - z^*\rangle + \eta^2 \|\widetilde{g}(z^k) - \widetilde{g}(z^{k+\frac{1}{2}}) \|_2^2 
    \\& \quad - \|z^k - z^{k+\frac{1}{2}}\|_2^2  \\
    &\leqslant \|z^k - z^* \|_2^2 - 2 \eta \mu \| z^{k+\frac{1}{2}} - z^* \|_2^2 + 2 \eta \alpha \|g(z^{k+\frac{1}{2}})\|_2 \| z^{k+\frac{1}{2}} - z^* \|_2 + \eta^2 \|\widetilde{g}(z^k) - \widetilde{g}(z^{k+\frac{1}{2}})\|_2^2 
    \\& \quad - \|z^k - z^{k+\frac{1}{2}}\|_2^2.
\end{align*}
We will need auxiliary inequality
\begin{align*}
    \|\widetilde{g}(z^{k+\frac{1}{2}}) - \widetilde{g}(z^{k})\|_2 & \leqslant \|g(z^{k+\frac{1}{2}}) - g(z^{k}) \|_2 + \|\widetilde{g}(z^{k}) - g(z^{k}) \|_2 + \|\widetilde{g}(z^{k+\frac{1}{2}}) - g(z^{k+\frac{1}{2}}) \|_2 \\
    & \leqslant \quad  L \|z^{k+\frac{1}{2}} - z^k \|_2 + \frac{\alpha}{1-\alpha}\|\widetilde{g}(z^k)\|_2 + \alpha \|g(z^{k+\frac{1}{2}})\|_2,
\end{align*}
to continue the main chain:
\begin{align*}
    \|z^{k+1}-z^{*}\|_2^2 & \leqslant \|z^k - z^* \|_2^2 - \eta \mu \| z^{k+\frac{1}{2}} - z^* \|_2^2 + \frac{\eta \alpha^2}{\mu} \|g(z^{k+\frac{1}{2}})\|_2^2 + 3 \eta^2 L^2 \| z^{k+\frac{1}{2}} - z^k \|_2^2  \\
    & \quad + \frac{3 \eta^2 \alpha^2}{(1- \alpha)^2} \|\widetilde{g}(z^k)\|_2^2 + 3 \eta^2 \alpha^2 \|g(z_{k+ \frac{1}{2}}) \|_2^2 - \|z^k - z^{k+\frac{1}{2}}\|_2^2 
    \\& = \|z^k - z^* \|_2^2 - \eta \mu \| z^{k+\frac{1}{2}} - z^* \|_2^2 + \left(\eta\frac{\alpha^2}{\mu} + 3\eta^2\alpha^2\right) \|g(z_{k+ \frac{1}{2}}) \|_2^2 \\
    & \quad + \left(3\eta^2L^2 + \frac{3\alpha^2}{(1-\alpha)^2} - 1\right)\| z^{k+\frac{1}{2}} - z^k \|_2^2. 
\end{align*}
Finally, by using
\begin{align*}
    \|g(z^{k+\frac{1}{2}}) \|_2 \leqslant L\|z^{k+\frac{1}{2}} - z^k \|_2 + \|g(z^k)\|_2 \leqslant \left(L + \frac{1}{\eta(1-\alpha)}\right)\|z^{k+\frac{1}{2}} - z^k\|_2
\end{align*}
we prove the Lemma.
\end{proof}

For convenience, we restate the \cref{teo:eg_alpha_upper_bound}.

%Extragradient: corollary for alpha gaurantee
\textbf{\cref{teo:eg_alpha_upper_bound}}
\textit{
Let \cref{as:lipqsm} hold. There exists $\widehat \alpha = b\sqrt{\nicefrac{\mu}{L}}$ for some numerical constant $b>0$ s.t. if Assumption \ref{as:ne-inexact} holds with $\alpha < \widehat \alpha$, then Algorithm \ref{alg:eg} with parameters $(z^0,\frac{1}{cL})$ with a numerical constant $c>0$, generates $z^N$ s.t.
\begin{equation*}
    \|z^N-z^*\|_2^2 \leqslant \|z^0-z^*\|_2^2\left( 1 - \frac{\mu}{2cL} \right)^N.
\end{equation*}
% For $g \in S_{\mu, L}$ (\cref{as:lipqsm}), $\exists \widehat \alpha$: $\widehat  \alpha = \mathcal{O}\left(\sqrt{\frac{\mu}{L}}\right)$, s.t. if $\alpha < \widehat \alpha$ then inexact EG retains linear convergence with contraction factor of $1 - \frac{1}{2}\eta\mu$, where $\eta \sim \frac{1}{L}$.
}

\begin{proof}[Proof of~\cref{teo:eg_alpha_upper_bound}]
Combining inequality
\begin{align*}
    -\|z^{k+\frac{1}{2}} - z^* \|_2^2 \leqslant -\frac{1}{2} \|z^{k} - z^* \|_2^2 + \|z^{k+\frac{1}{2}} - z^k \|_2^2
\end{align*}
with Lemma~\ref{lem:eg_contraction_stongly_monotone_Lipshitz}, we obtain
\begin{align*}
    &\quad\; \|z^{k+1}-z^{*}\|_2^2 
    \\& \leqslant \|z^{k}-z^{*}\|_2^2 - \eta\mu\left(\frac{1}{2} \|z^{k} - z^* \|_2^2 - \|z^{k+\frac{1}{2}} - z^k \|_2^2\right) 
    \\& \quad + \left[\left(\frac{\eta\alpha^2}{\mu} +3\eta^2\alpha^2\right) \left(2L^2+\frac{2}{\eta^2(1-\alpha)^2}\right)+3\eta^2L^2 + \frac{3\alpha^2}{(1-\alpha)^2} - 1 \right] \|z^{k+\frac{1}{2}}-z^{k}\|_2^2 \\
    &=(1-\frac{1}{2}\eta\mu) \|z^{k}-z^{*}\|_2^2
    \\& \quad + \left(\left(\frac{\eta\alpha^2}{\mu} +3\eta^2\alpha^2\right) \left(2L^2+\frac{2}{\eta^2(1-\alpha)^2}\right)+3\eta^2L^2 + \frac{3\alpha^2}{(1-\alpha)^2} + \eta\mu - 1 \right)\|z^{k+\frac{1}{2}}-z^{k}\|_2^2 \\
    &= (1-\frac{1}{2}\eta\mu) \|z^{k}-z^{*}\|_2^2 + \left(\mathcal{O}\left(1\right) - 1 \right)\|z^{k+\frac{1}{2}}-z^{k}\|_2^2 
    \\& \leqslant \left(1-\frac{1}{2}\eta\mu\right)\|z^{k}-z^{*}\|_2^2=\left(1 - \frac{\mu}{2cL}\right)\|z^{k}-z^{*}\|_2^2
\end{align*}
where we used that $\eta = \frac{1}{cL}$ and $\alpha < \widehat \alpha = b\sqrt{\frac{\mu}{L}}$ with an appropriate choice of constants $b,c>0$ and that $\frac{L}{\mu}$ is large. Hence, the contraction factor is $1 - \frac{\mu}{2cL}$, which proves the theorem.

% \valery{
We can also evaluate numerical multiplier of $\widehat \alpha$ for when $k = L/\mu$ is large. Rewriting the multiplier before $\|z^{k+\frac{1}{2}}-z^{k}\|_2^2$:
\begin{align*}
    \frac{3}{c^3} + \frac{2\alpha^2k}{c^2} - \frac{1 - \xi}{c} + \frac{2\alpha^2k}{(1-\alpha)^2}
\end{align*}
where
\begin{align*}
    \xi = \frac{6\alpha^2}{(1-\alpha)^2} \left(\frac{(1-\alpha)^2}{c^2}+1 \right) + \frac{3\alpha^2}{(1-\alpha)^2} + \frac{1}{ck} = o(1)
\end{align*}
to guarantee linear convergence we demand the value of the expression to be non-positive:
\begin{align*}
    \frac{3}{c^3} + \frac{2\alpha^2k}{c^2} - \frac{1}{c} + 2\alpha^2k + o(1) \leqslant 0
\end{align*}
From the expression it can be concluded that for large condition numbers, convergence is being guaranteed for $\alpha < \widehat \alpha$ where $\widehat \alpha$ is approaching $\left(\frac{1}{12}(2\sqrt{7}-5)\right)^{1/4} \sqrt{\frac{\mu}{L}} \approx 0.39478\sqrt{\frac{\mu}{L}}$ from below.

% }
\end{proof}

%all SP methods alpha threshold

%%%%%%%%%%%%%%%%%%%%%%%
\newpage

\section{Additional Experiments for Sections~\ref{section min max} and~\ref{sec:ne}}\label{appendix Lyapunov min max}

In this appendix we present the results of numerical experiments with automatic generation of Lyapunov functions. The graphs are the dependencies of best verified convergence rate on step size. In all experiments we fixed $\mu = 1$ and varied $L$ to get different condition numbers $\kappa = \frac{L}{\mu}$. Curves of different color correspond to different values of $\alpha$. Black dashed line corresponds to contraction factor equal to $1$. \par

\begin{figure}[H]
    \centering
    \begin{minipage}[htp]{0.77\textwidth}
        \centering
        \includegraphics[width=1\linewidth]{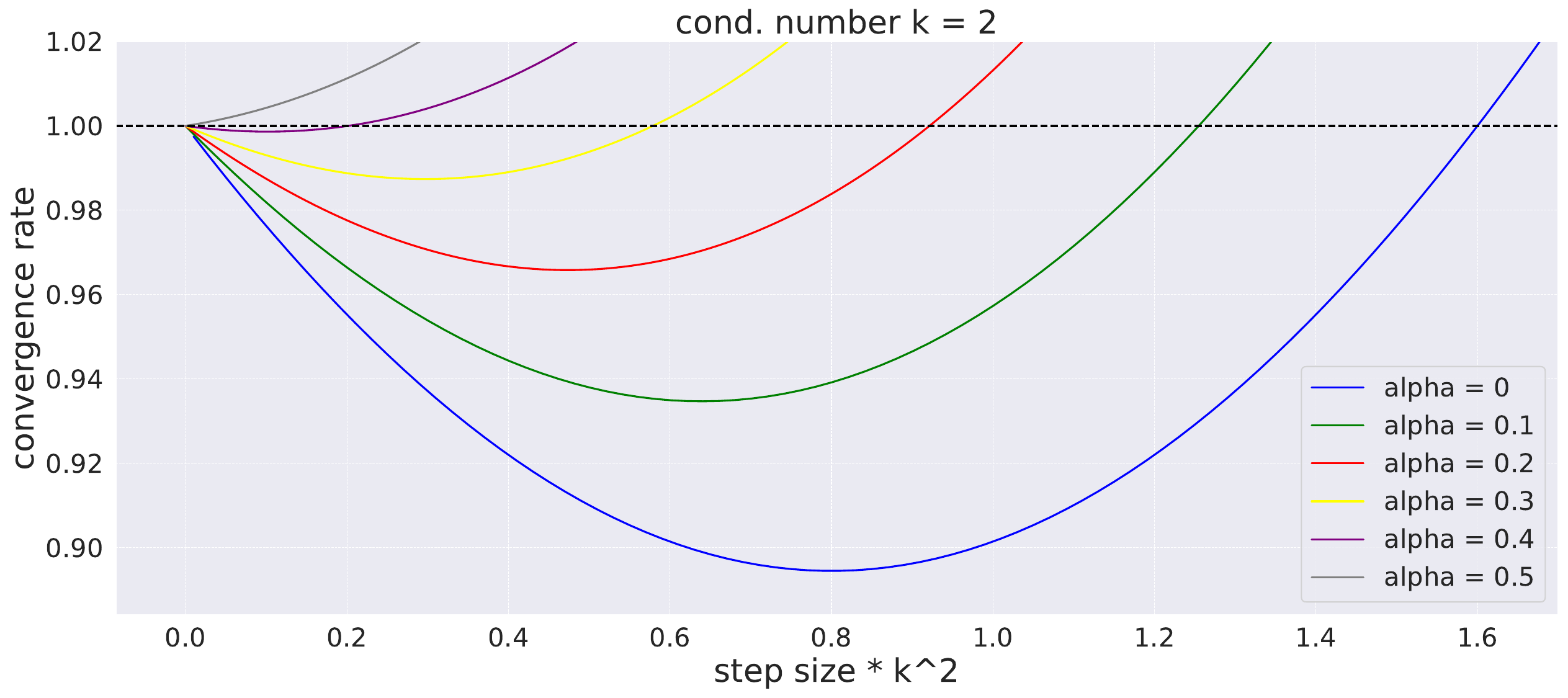}
        \caption{\centering Inexact Sim-GDA ($\kappa=2$)}
        \label{fig:lyapunov_sim2_app}
    \end{minipage}
\end{figure}
\begin{figure}[H]
    \centering    \begin{minipage}[htp]{0.77\textwidth}
        \centering
        \includegraphics[width=1\linewidth]{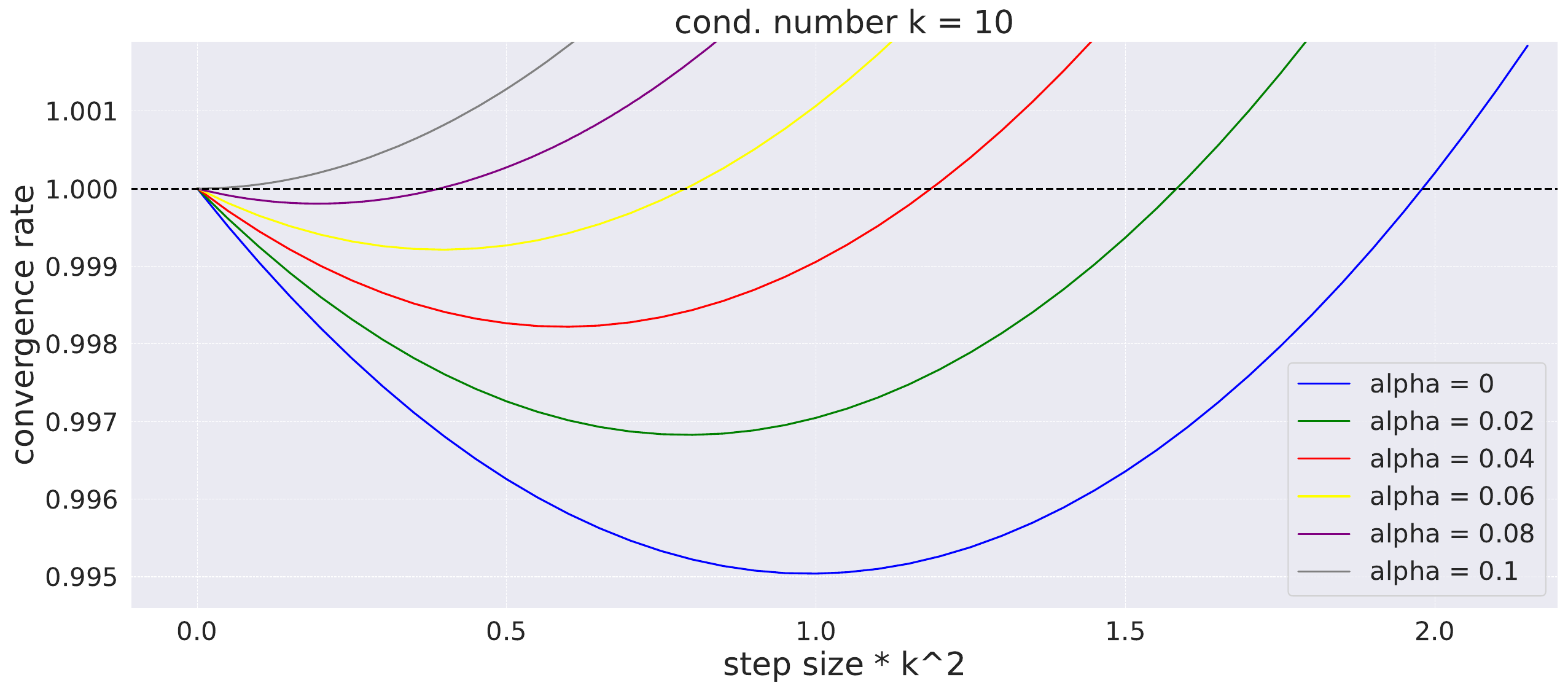}
        \caption{\centering Inexact Sim-GDA ($\kappa=10$)}
        \label{fig:lyapunov_sim10_app}
    \end{minipage}
\end{figure}
\begin{figure}[H]
    \centering   \begin{minipage}[htp]{0.77\textwidth}
        \centering
        \includegraphics[width=1\linewidth]{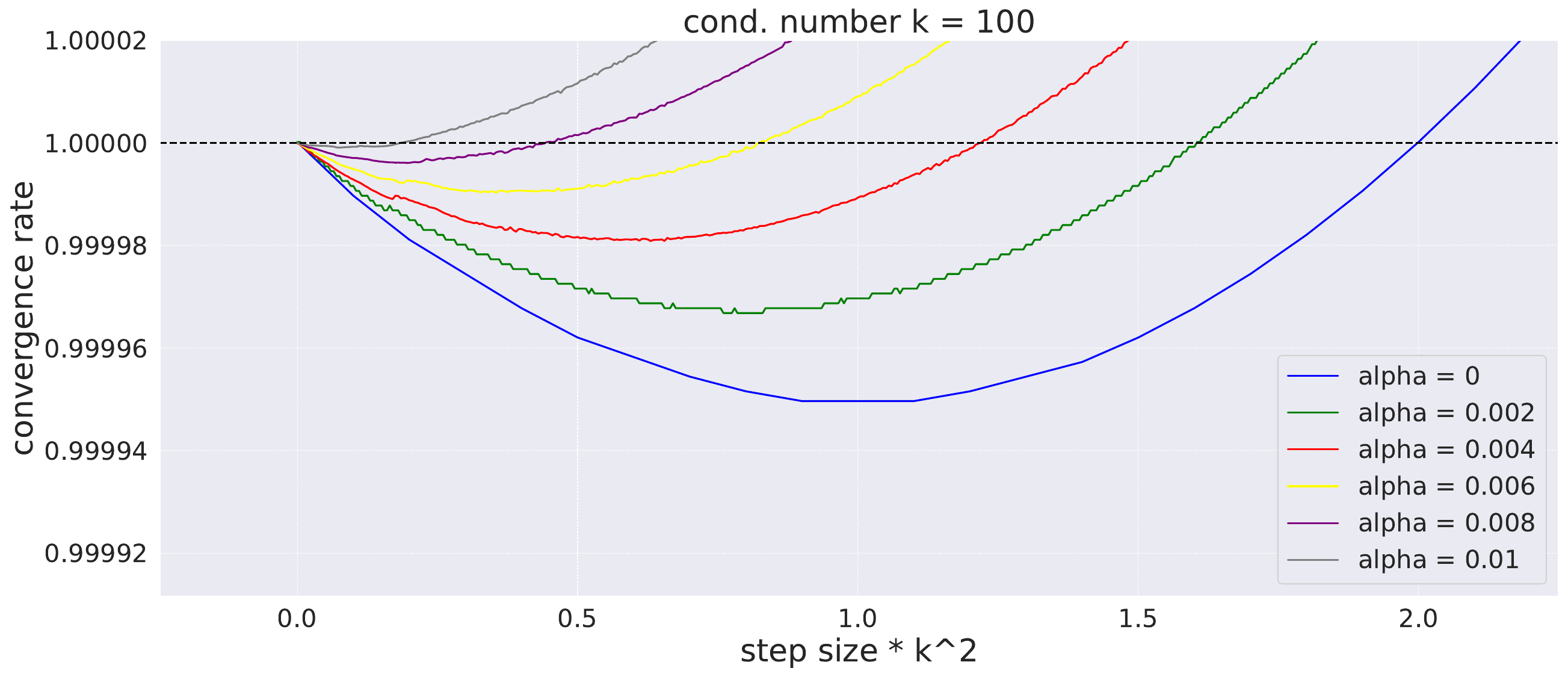}
        \caption{\centering Inexact Sim-GDA ($\kappa=100$)}
        \label{fig:lyapunov_sim100_app}
    \end{minipage}
\end{figure}

\begin{figure}[H]
    \centering
    \begin{minipage}[htp]{0.77\textwidth}
        \centering
        \includegraphics[width=1\linewidth]{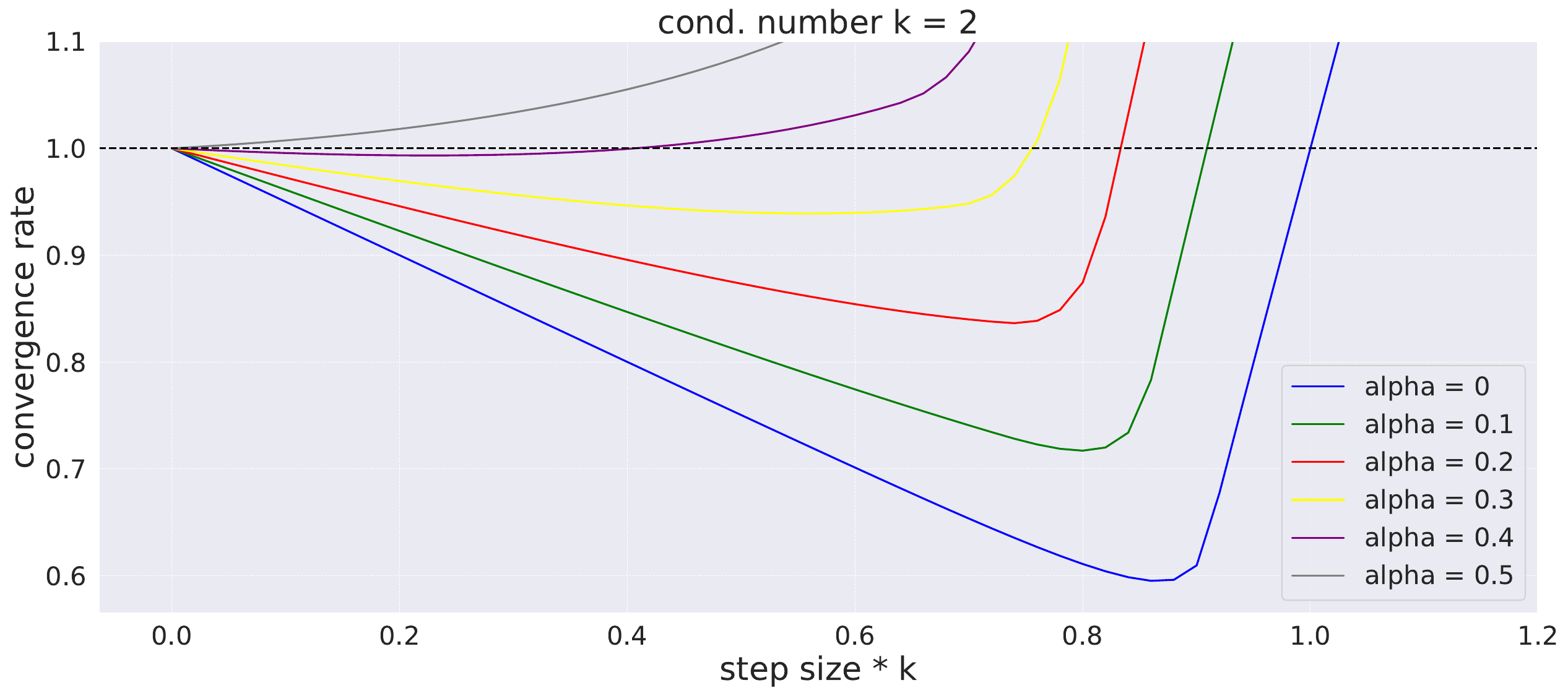}
        \caption{\centering Inexact Alt-GDA ($\kappa=2$)}
        \label{fig:lyapunov_alt2_app}
    \end{minipage}
\end{figure}
\begin{figure}[H]
    \centering    \begin{minipage}[htp]{0.77\textwidth}
        \centering
        \includegraphics[width=1\linewidth]{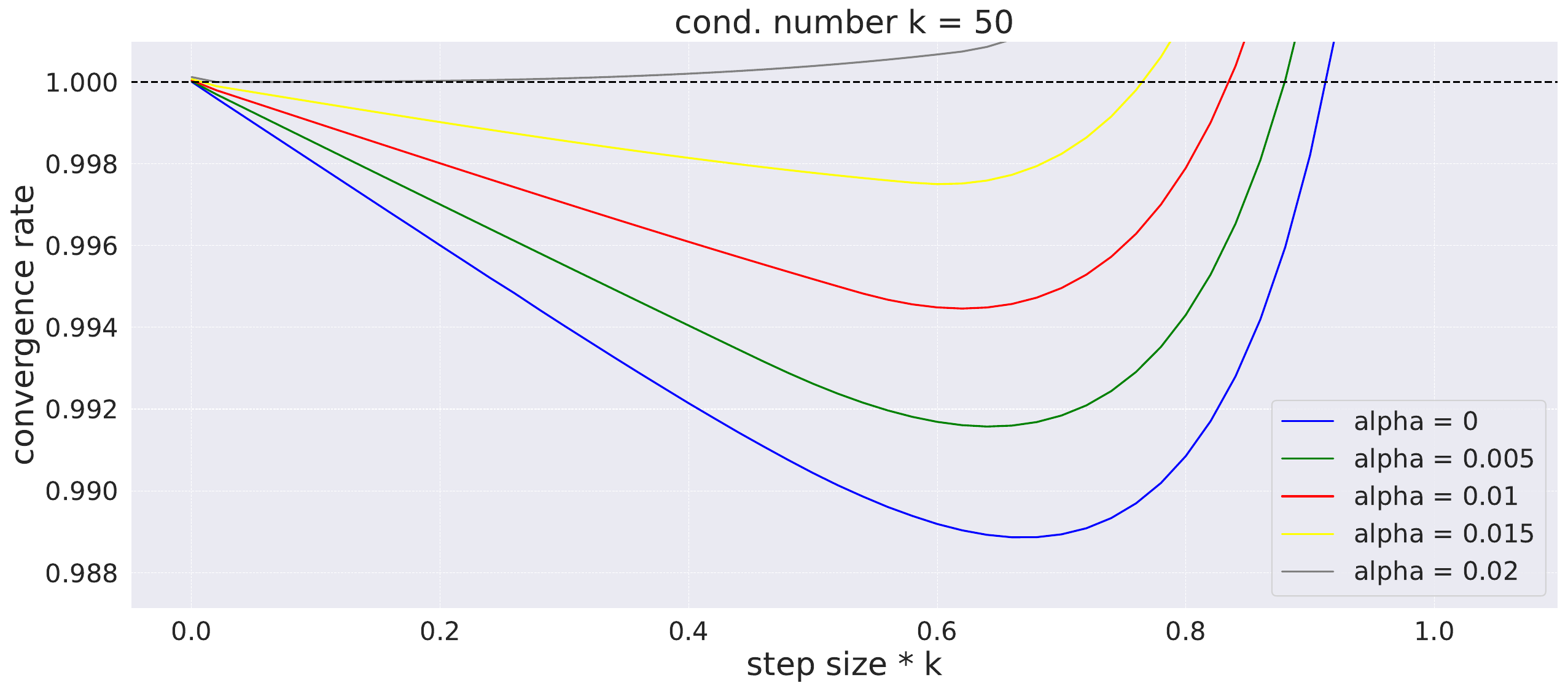}
        \caption{\centering Inexact Alt-GDA ($\kappa=50$)}
        \label{fig:lyapunov_alt50_app}
    \end{minipage}
\end{figure}
\begin{figure}[H]
    \centering    \begin{minipage}[htp]{0.77\textwidth}
        \centering
        \includegraphics[width=1\linewidth]{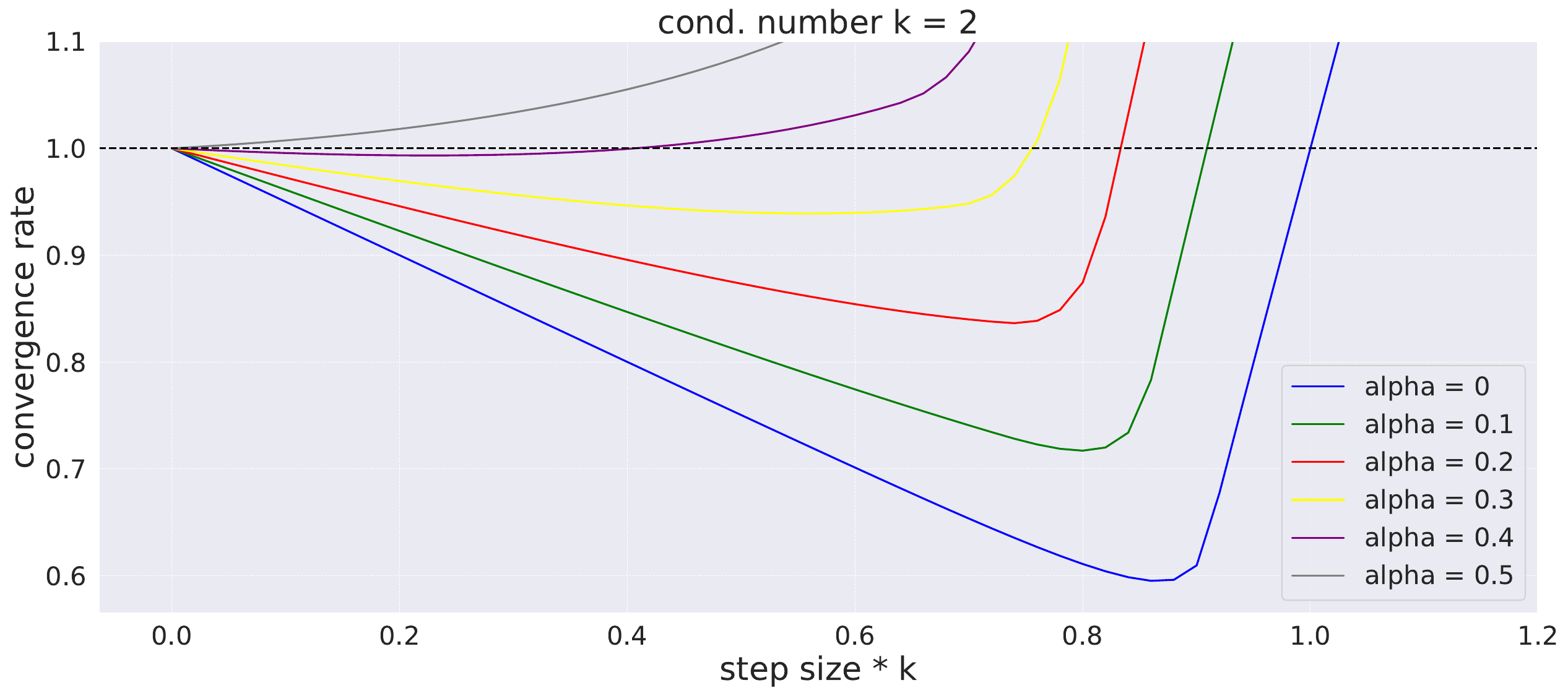}
        \caption{\centering Inexact Alt-GDA ($\kappa=100$)}
        \label{fig:lyapunov_alt100_app}
    \end{minipage}
\end{figure}
\begin{figure}[H]
    \centering    \begin{minipage}[htp]{0.77\textwidth}
        \centering
        \includegraphics[width=1\linewidth]{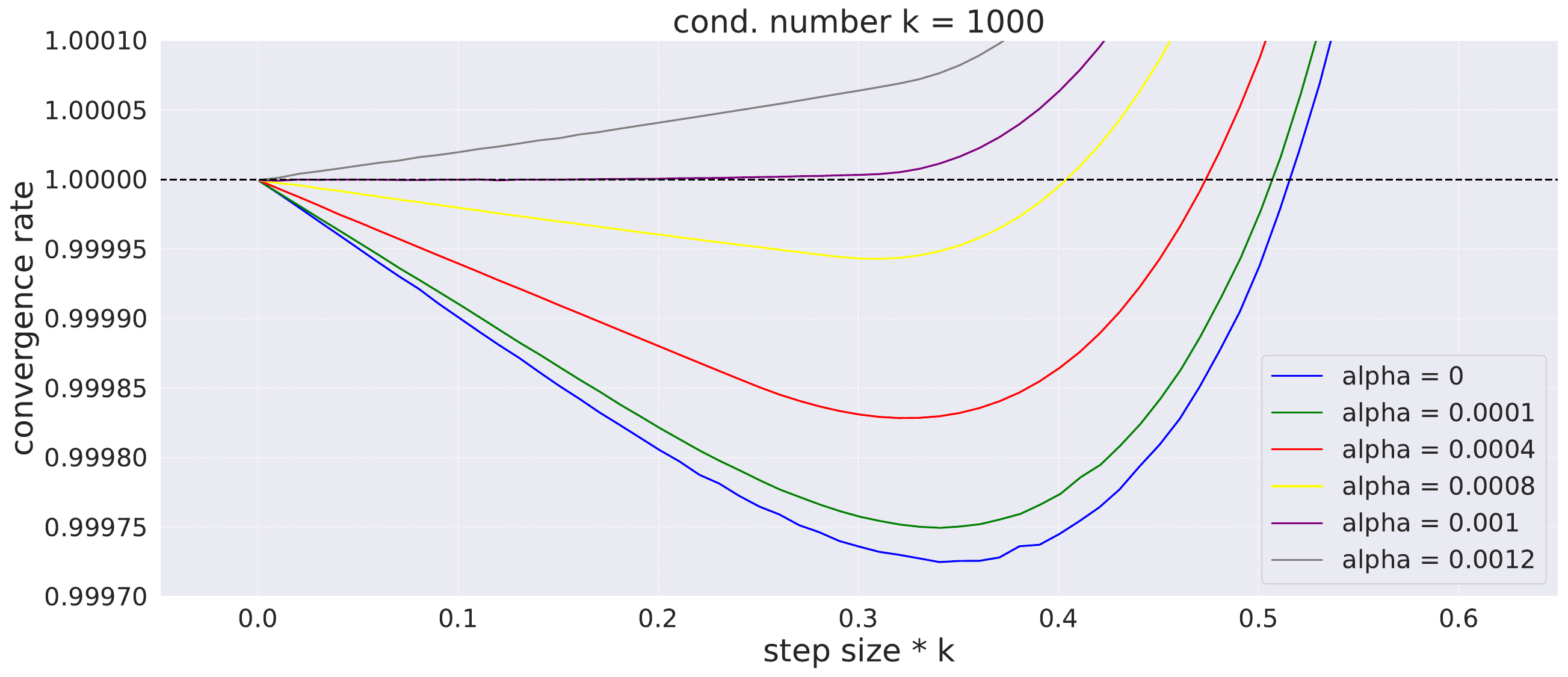}
        \caption{\centering Inexact Alt-GDA ($\kappa=1000$)}
        \label{fig:lyapunov_alt1000_app}
    \end{minipage}
\end{figure}

% \begin{figure}[H]
% \center{\includegraphics[width=1\linewidth]{Lyapunov_saddle/Sim_GDA/SIM_MERGED_2.png}}
% \na{it would be easier to consistently adjust plots if all the plots was separate, so that one could stack them in one line to save some space in the main part see also plots below}
% \na{also yellow color is indistinguishable}
% \caption{\centering Inexact Sim-GDA}
% \label{fig:lyapunov_sim}
% \end{figure}

% \begin{figure}[H]
% \center{\includegraphics[width=1\linewidth]{Lyapunov_saddle/ALT_GDA/ALT_MERGED.png}}
% \caption{\centering Inexact Alt-GDA}
% \label{fig:lyapunov_alt}
% \end{figure}

\begin{figure}[H]
    \centering
    \begin{minipage}[htp]{0.77\textwidth}
        \centering
        \includegraphics[width=1\linewidth]{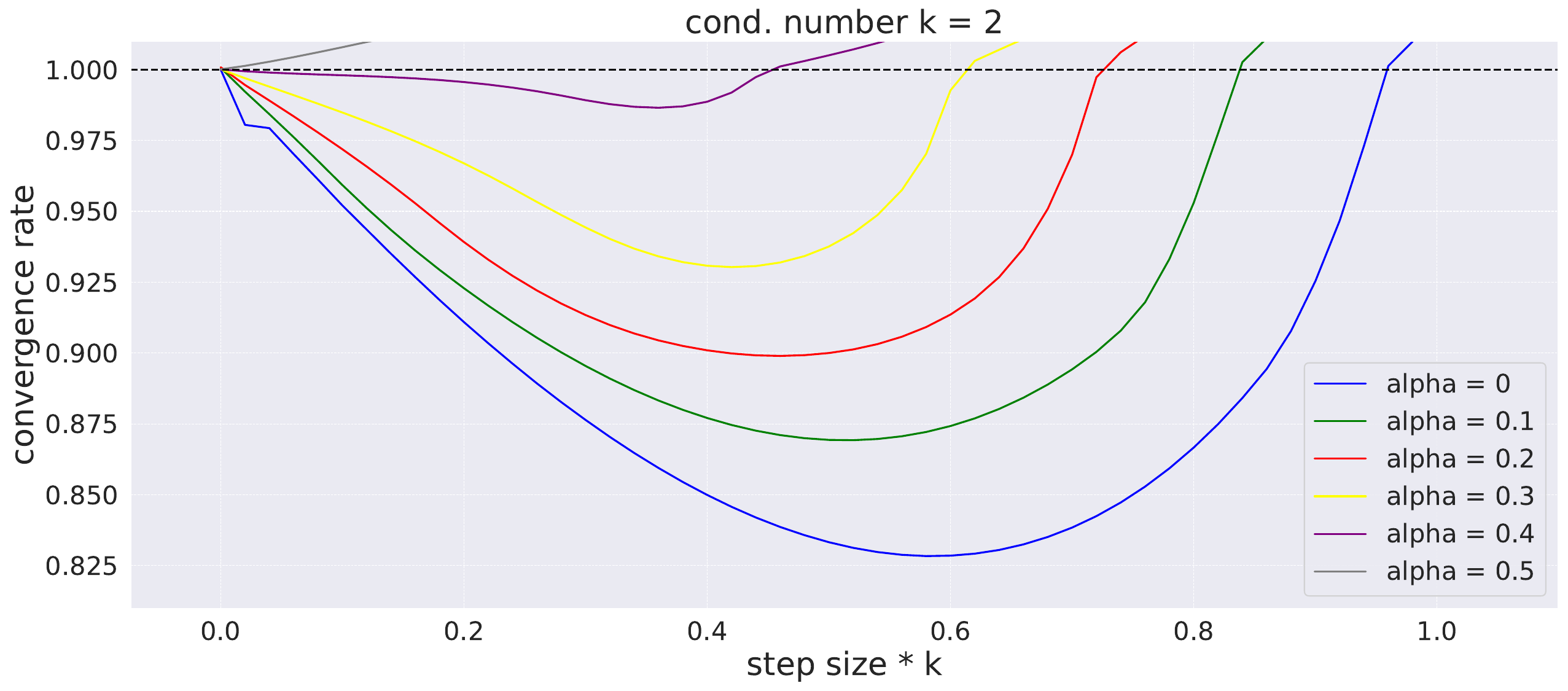}
        \caption{\centering Inexact EG($\kappa=2$)}
        \label{fig:lyapunov_eg2_app}
    \end{minipage}
\end{figure}
\begin{figure}[H]
    \centering    \begin{minipage}[htp]{0.77\textwidth}
        \centering
        \includegraphics[width=1\linewidth]{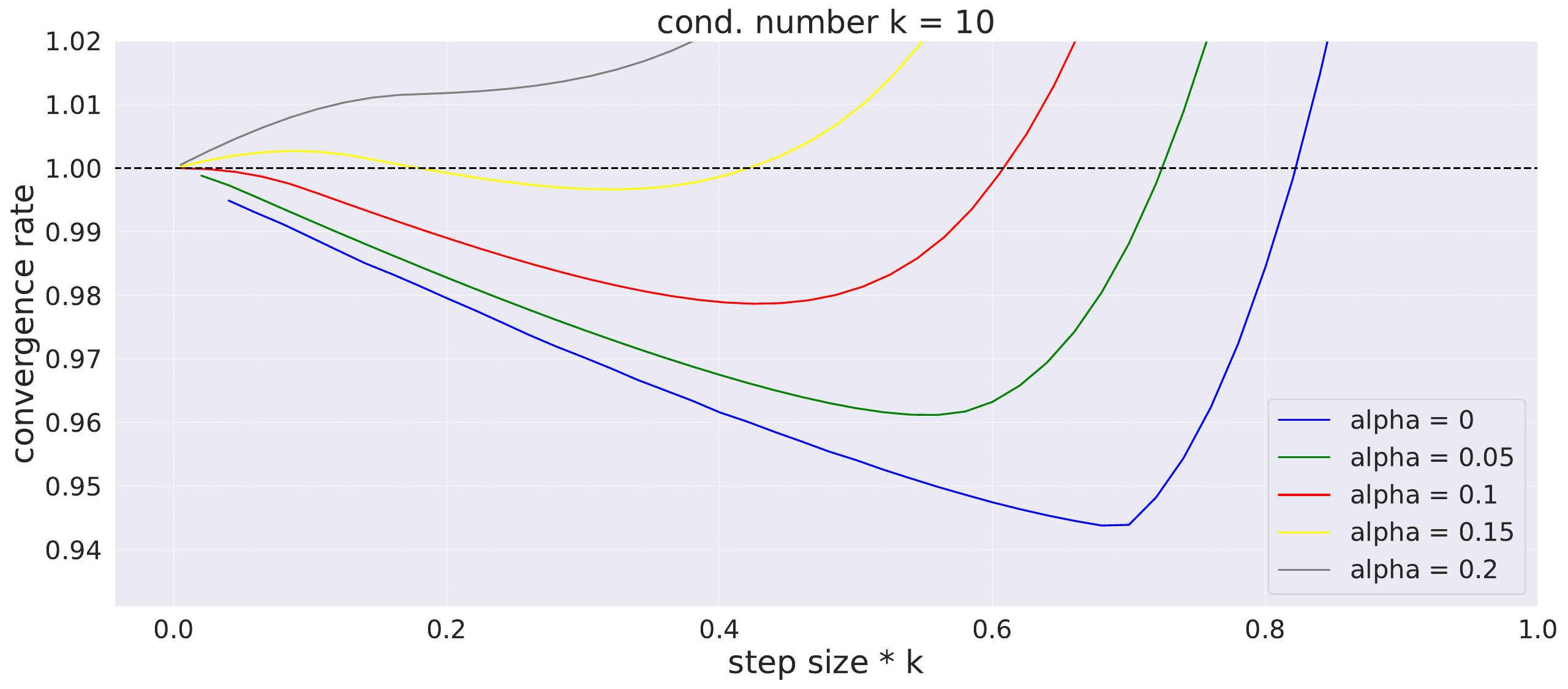}
        \caption{\centering Inexact EG($\kappa=10$)}
        \label{fig:lyapunov_eg10_app}
    \end{minipage}
\end{figure}
\begin{figure}[H]
    \centering    \begin{minipage}[htp]{0.77\textwidth}
        \centering
        \includegraphics[width=1\linewidth]{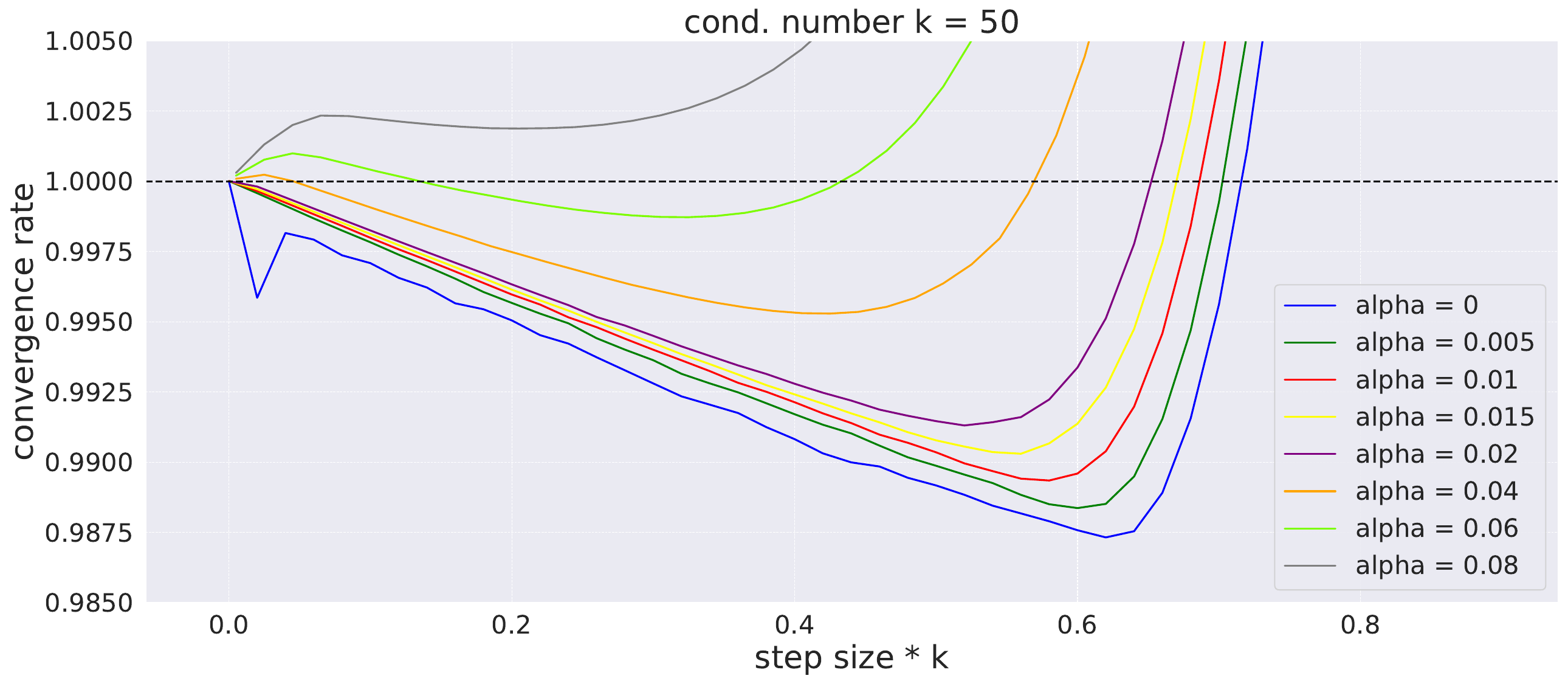}
        \caption{\centering Inexact EG($\kappa=50$)}
        \label{fig:lyapunov_eg50_app}
    \end{minipage}
\end{figure}
\begin{figure}[H]
    \centering
    \begin{minipage}[htp]{0.77\textwidth}
        \centering
        \includegraphics[width=1\linewidth]{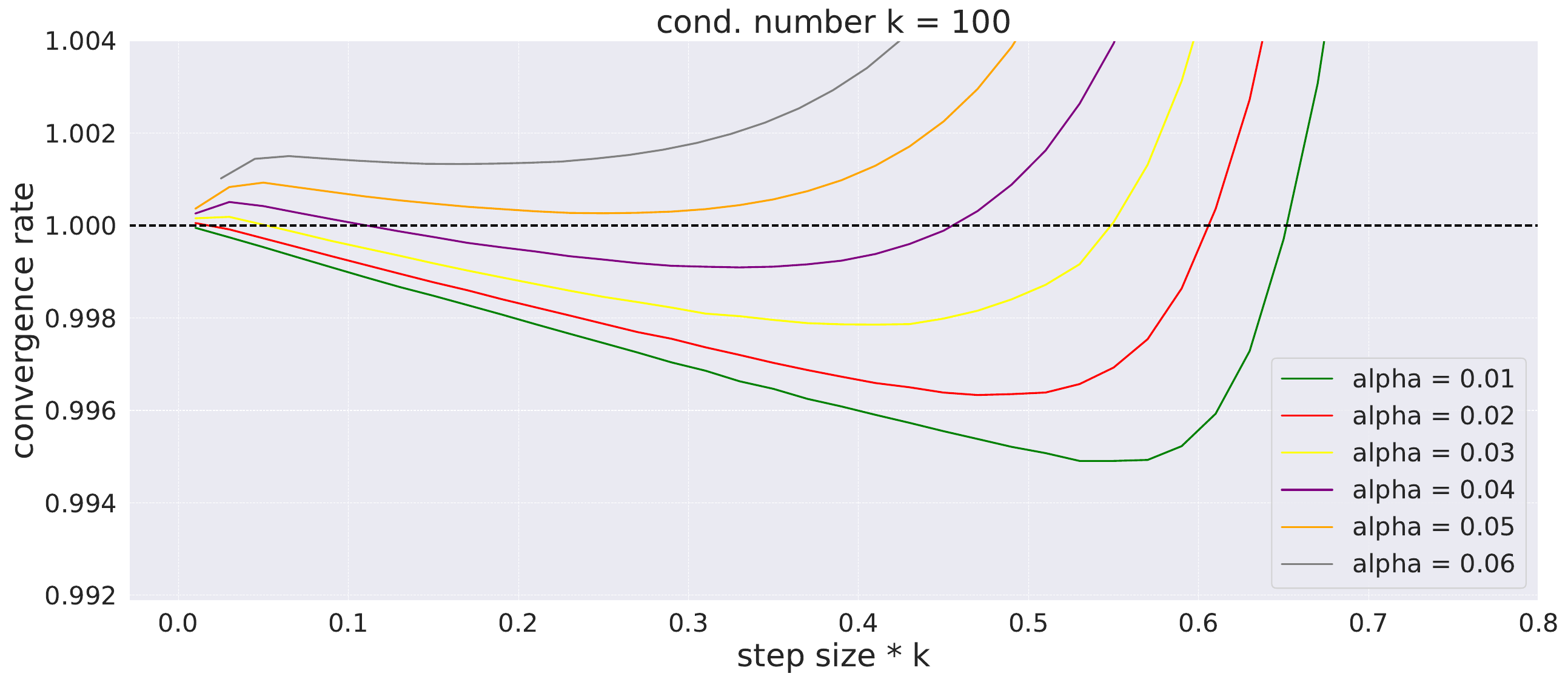}
        \caption{\centering Inexact EG($\kappa=100$)}
        \label{fig:lyapunov_eg100_app}
    \end{minipage}
\end{figure}
\begin{figure}[H]
    \centering
    \begin{minipage}[htp]{0.77\textwidth}
        \centering
        \includegraphics[width=1\linewidth]{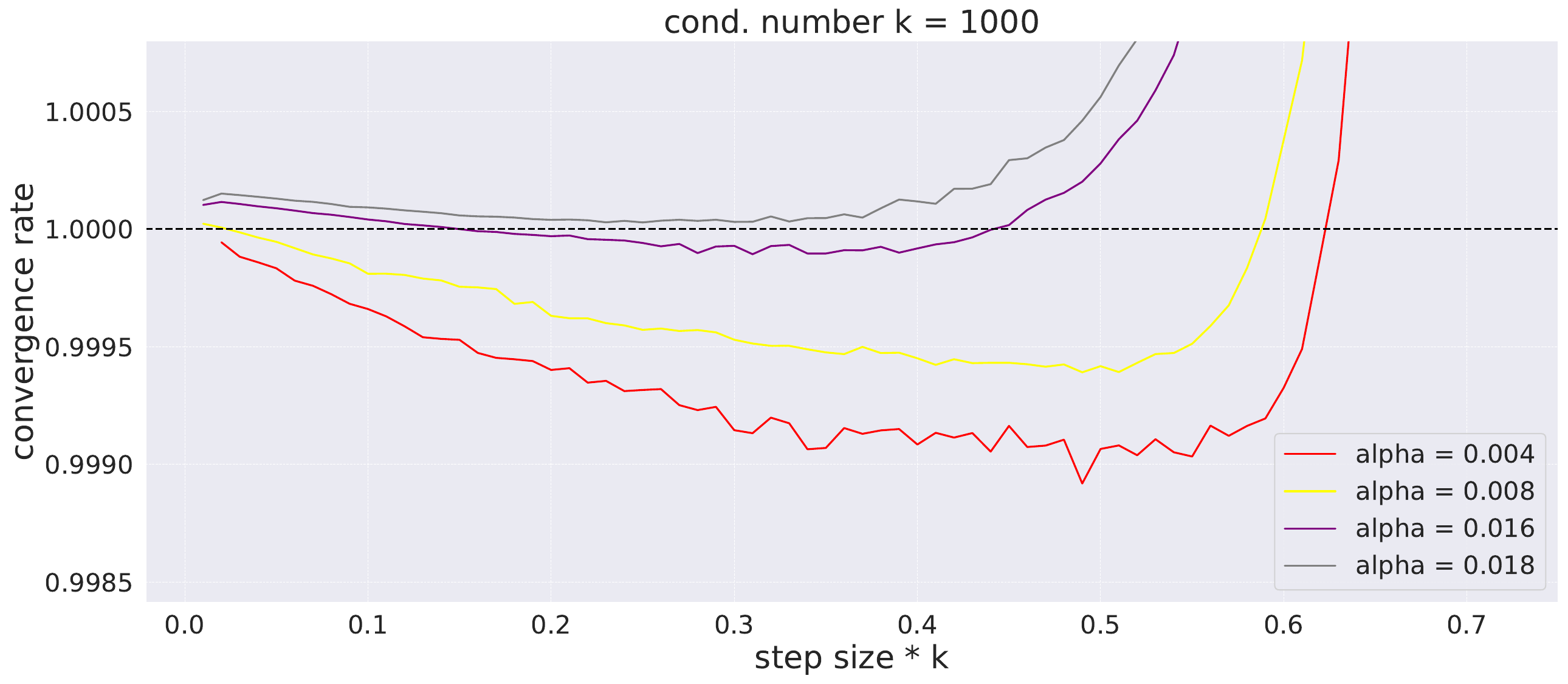}
        \caption{\centering Inexact EG($\kappa=1000$)}
        \label{fig:lyapunov_eg1000}
    \end{minipage}
\end{figure}

\end{appendices}

%%===========================================================================================%%
%% If you are submitting to one of the Nature Portfolio journals, using the eJP submission   %%
%% system, please include the references within the manuscript file itself. You may do this  %%
%% by copying the reference list from your .bbl file, paste it into the main manuscript .tex %%
%% file, and delete the associated \verb+\bibliography+ commands.                            %%
%%===========================================================================================%%
\bibliography{refs}% common bib file
%% if required, the content of .bbl file can be included here once bbl is generated
%%\input sn-article.bbl

\end{document}

%% file: defs.tex
\usepackage{url}            % simple URL typesetting
\usepackage{booktabs}       % professional-quality tables
\usepackage{amsfonts}       % blackboard math symbols
\usepackage{nicefrac}       % compact symbols for 1/2, etc.
\usepackage{microtype}      % microtypography
\usepackage{bbm}            % for indicators 1
\usepackage{bm}             % bold symbols

\usepackage{eqparbox}
\newcommand{\R}{\mathbb{R}}

\newcommand\swapifbranches[3]{#1{#3}{#2}}
\makeatletter
\MHInternalSyntaxOn
\patchcmd{\DeclarePairedDelimiter}{\@ifstar}{\swapifbranches\@ifstar}{}{}
\MHInternalSyntaxOff
\makeatother

\DeclarePairedDelimiterX{\inp}[2]{\langle}{\rangle}{#1, #2}
\DeclarePairedDelimiterX{\abs}[1]{\lvert}{\rvert}{#1}
\DeclarePairedDelimiterX{\roundup}[1]{\lceil}{\rceil}{#1}
\DeclarePairedDelimiterX{\norm}[1]{\lVert}{\rVert}{#1}
\DeclarePairedDelimiterX{\cbr}[1]{\{}{\}}{#1} % curly bracket
\DeclarePairedDelimiterX{\rbr}[1]{(}{)}{#1} % round bracket
\DeclarePairedDelimiterX{\sbr}[1]{[}{]}{#1} % 